\documentclass[11pt, reqno]{amsart}
\usepackage{amssymb,latexsym,amsmath,amsfonts,mathdots,enumitem}
\usepackage{latexsym}
\usepackage[mathscr]{eucal}
\usepackage{colortbl,xcolor}
\usepackage{lmodern}
\usepackage{sansmathaccent}
\usepackage[latin1]{inputenc}
\usepackage{tikz}
\usepackage{physics}
\usetikzlibrary{shapes,arrows}
\usetikzlibrary{matrix,calc,shapes,arrows,positioning}
\pdfmapfile{+sansmathaccent.map}

\numberwithin{equation}{section}
\theoremstyle{plain}

\newtheorem{thm}{Theorem}[section]

\newtheorem{cor}[thm]{Corollary}
\newtheorem{lem}[thm]{Lemma}
\newtheorem{prop}[thm]{Proposition}
\theoremstyle{definition}
\newtheorem{defn}[thm]{Definition}
\newtheorem{rem}[thm]{Remark}

\numberwithin{equation}{section}
\def\A{{\mathcal A}}

\def\Re{\operatorname{Re}}

\def\beq{\begin{eqnarray}}
	\def\eeq{\end{eqnarray}}
\def\beqa{\begin{eqnarray*}}
	\def\eeqa{\end{eqnarray*}}

\def\beqn{\begin{equation}}
	\def\eeqn{\end{equation}}

\def\mg#1{}

\renewcommand{\epsilon}{\varepsilon}
\renewcommand{\phi}{\varphi}

\renewcommand{\bf}[1]{\textbf{#1}}
\renewcommand{\it}[1]{\textit{#1}}

\renewcommand{\sf}[1]{\textsf{#1}}

\renewcommand{\Re}[1]{\sf{Re}(#1)}

\numberwithin{equation}{section}
\allowdisplaybreaks[4] 

\setlist[enumerate]{font=\upshape,noitemsep, topsep=0pt} 
\setlist[itemize]{noitemsep, topsep=0pt}

\begin{document}
	
	\title[Canonical Decompositions and Conditional Dilations]{Canonical Decompositions and Conditional Dilations of $\Gamma_{E(3; 3; 1, 1, 1)}$-Contraction and $\Gamma_{E(3; 2; 1, 2)}$-Contraction}
	\author{Dinesh Kumar Keshari, Avijit Pal and Bhaskar Paul}
	
	\address[ D. K. Keshari]{School of Mathematical Sciences, National Institute of Science Education and Research Bhubaneswar, An OCC of Homi Bhabha National Institute, Jatni, Khurda,  Odisha-752050, India}
\email{dinesh@niser.ac.in}

\address[A. Pal]{Department of Mathematics, IIT Bhilai, 6th Lane Road, Jevra, Chhattisgarh 491002}
\email{A. Pal:avijit@iitbhilai.ac.in}

\address[B. Paul]{Department of Mathematics, IIT Bhilai, 6th Lane Road, Jevra, Chhattisgarh 491002}
\email{B. Paul:bhaskarpaul@iitbhilai.ac.in }


	\maketitle
	
	\begin{abstract}

A $7$-tuple of commuting bounded operators $\mathbf{T} = (T_1, \dots, T_7)$ defined on a Hilbert space $\mathcal{H}$ is said to be a \textit{$\Gamma_{E(3; 3; 1, 1, 1)}$-contraction} if $\Gamma_{E(3; 3; 1, 1, 1)}$ is a spectral set for $\mathbf{T}$. Let $(S_1, S_2, S_3)$ and $(\tilde{S}_1, \tilde{S}_2)$ be tuples of commuting bounded operators on $\mathcal{H}$ satisfying $S_i \tilde{S}_j = \tilde{S}_j S_i$ for $1 \leq i \leq 3$ and $1 \leq j \leq 2$. The tuple $\mathbf{S} = (S_1, S_2, S_3, \tilde{S}_1, \tilde{S}_2)$ is called a \textit{$\Gamma_{E(3; 2; 1, 2)}$-contraction} if $\Gamma_{E(3; 2; 1, 2)}$ is a spectral set for $\mathbf{S}$.

In this paper, we establish the existence and uniqueness of the fundamental operators associated with $\Gamma_{E(3; 3; 1, 1, 1)}$-contractions and $\Gamma_{E(3; 2; 1, 2)}$-contractions. Furthermore, we obtain a Beurling-Lax-Halmos type representation for invariant subspaces corresponding to a pure $\Gamma_{E(3; 3; 1, 1, 1)}$-isometry and a pure $\Gamma_{E(3; 2; 1, 2)}$-isometry.

We also construct a conditional dilation for a $\Gamma_{E(3; 3; 1, 1, 1)}$-contraction and  a $\Gamma_{E(3; 2; 1, 2)}$-contraction and develop an explicit functional model for a certain subclass of these operator tuples. Finally, we demonstrate that every $\Gamma_{E(3; 3; 1, 1, 1)}$-contraction (respectively, $\Gamma_{E(3; 2; 1, 2)}$-contraction) admits a unique decomposition as a direct sum of a $\Gamma_{E(3; 3; 1, 1, 1)}$-unitary (respectively, $\Gamma_{E(3; 2; 1, 2)}$-unitary) and a completely non-unitary $\Gamma_{E(3; 3; 1, 1, 1)}$-contraction (respectively, $\Gamma_{E(3; 2; 1, 2)}$-contraction).

\end{abstract}
	
	\section{Introduction}
Let $\Omega$ be a compact subset of $\mathbb C ^m,$ and let $\mathcal O(\Omega)$ be the algebra of holomorphic functions defined on an open set that contains $\Omega.$ Let $\mathbf{T}=(T_1,\ldots,T_m)$ be a  $m$-tuple of commuting bounded operators on some Hilbert space $\mathcal H$. The joint spectrum of $\mathbf {T} $ is denoted by  $\sigma(\mathbf T)$. Let $\rho_{\mathbf T}:\mathcal O(\Omega)\rightarrow\mathcal B(\mathcal H)$ be a map defined as folows   $$1\to I~{\rm{and}}~z_i\to T_i~{\rm{for}}~1\leq i\leq m. $$ It is evident that $\rho_{\mathbf T}$ is a homomorphism. A compact set $\Omega\subset \mathbb C^m$ is said to be a spectral set for a $m$-tuple of commuting bounded operators $\mathbf{T}=(T_1,\ldots,T_m)$ if $\sigma(\mathbf T)\subseteq \Omega$ and the  homomorphism $\rho_{\mathbf T}:\mathcal O(\Omega)\rightarrow\mathcal B(\mathcal H)$ is contractive. The closed unit disc serves as the spectral set for every contraction defined on a  Hilbert space $\mathcal H$, as established by the following theorem [Corollary 1.2, \cite{paulsen}].
\begin{thm} 
Let $T\in \mathcal B(\mathcal H)$ be a contraction. Then
$$\|p(T)\|\leq \|p\|_{\infty, \bar {\mathbb D}}:=\sup\{|p(z)|: |z|\leq1\} $$ for every polynomial $p.$
\end{thm}
The following theorem presented here is a refined version of the Sz.-Nagy dilation theorem [Theorem 1.1, \cite{paulsen}].
\begin{thm} Let $T\in \mathcal B(\mathcal H)$ be a contraction. Then there exists a larger Hilbert space $\mathcal K$ that contains $\mathcal H$ as a subspace, and a unitary operator $U$ acting on a Hilbert space $\mathcal K \supseteq \mathcal H$ with the property that $\mathcal K$ is the smallest closed reducing subspace for $U$ containing $\mathcal H$ such that
$$P_\mathcal H\,U^n_{|\mathcal H}=T^n, ~{\rm{ for ~all}} ~n\in \mathbb N\cup \{0\}.$$
\end{thm}

Schaffer constructed  a unitary dilation for a contraction $T$. The spectral theorem for unitary operators and the existence of power dilation for contractions ensure the validity of the von Neumann inequality. Let $\Omega$ be a compact subset of $\mathbb C ^m. $ Let $F=\left(\!(f_{ij})\!\right)$ be a matrix-valued polynomial defined on $\Omega.$ We say that $\Omega$ is a complete spectral set (complete $\Omega$-contraction) for $\mathbf T$ if $\|F(\mathbf T) \| \leq \|F\|_{\infty, \Omega}$ for every $F\in \mathcal O(\Omega)\otimes \mathcal M_{k\times k}(\mathbb C), k\geq 1$.  We say that $\Omega$ possesses property $P$ if the following holds: if $\Omega$ is a spectral set for a commuting $m$-tuple of operators $\mathbf{T}$, then it is a complete spectral set for $\mathbf{T}$. We say that a $m$-tuple of  commuting bounded operators $\mathbf{T}$ with $\Omega$ as a spectral set, have a $\partial \Omega$ normal dilation if there is a Hilbert space $\mathcal K$ that contains $\mathcal H$ as a subspace, along with a commuting $m$-tuple of normal operators $\mathbf{N}=(N_1,\ldots,N_m)$ on $\mathcal K$ whose joint spectrum lies in $\partial \Omega$, and$$P_{\mathcal H}F(\mathbf N)\mid_{\mathcal H}=F(\mathbf T) ~{\rm{for~ all~}} F\in \mathcal O(\Omega).$$

In 1969, Arveson \cite{A,AW} demonstrated that a commuting $m$-tuple of operators $\mathbf{T}$ admits a $\partial \Omega$ normal dilation if and only if $\Omega$ is a spectral set for $\mathbf{T}$ and $\mathbf{T}$ satisfies the property $P.$ J. Agler \cite{agler} proved in 1984 that the annulus possesses property $P$. M. Dristchell and S. McCullough \cite{michel} showed that property $P$ does not hold for domains in the complex plane $\mathbb{C}$ with connectivity $n\geq 2$. In the multivariable setting, both the bidisc and the symmetrized bidisc are known to possess property $P$, as established by Agler and Young \cite{young} and Ando \cite{paulsen}, respectively. Parrott \cite{paulsen} provided the counterexample for $\mathbb D^n$ when $n > 2$. G. Misra \cite{GM,sastry}, V. Paulsen \cite{vern}, and E. Ricard \cite{pisier} demonstrated that no ball in $\mathbb{C}^m$, with respect to some norm $\|\cdot\|_{\Omega}$ and for $m \geq 3$, can have property $P$. It is further shown in \cite{cv} that if $B_1$ and $B_2$ are not simultaneously diagonalizable via unitary, the set $\Omega_{\mathbf B}:= \{(z_1,z_2) :\|z_1 B_1 + z_2 B_2 \|_{\rm op} < 1\}$ fails to have property $P$, where $\mathbf B=(B_1, B_2)$ in $\mathbb C^2 \otimes \mathcal M_2(\mathbb C)$ with $B_1$ and $B_2$ are linearly independent.  

Let $\mathcal M_{n\times n}(\mathbb{C})$ be the set of all $n\times n$ complex matrices and  $E$ be a linear subspace of $\mathcal M_{n\times n}(\mathbb{C}).$ We define the function $\mu_{E}: \mathcal M_{n\times n}(\mathbb{C}) \to [0,\infty)$ as follows:
\begin{equation}\label{mu}
\mu_{E}(A):=\frac{1}{\inf\{\|X\|: \,\ \det(1-AX)=0,\,\, X\in E\}},\;\; A\in \mathcal M_{n\times n}(\mathbb{C})
	\end{equation}
with the understanding that $\mu_{E}(A):=0$ if $1-AX$ is  nonsingular for all $X\in E$ \cite{ds, jcd}.   Here $\|\cdot\|$ denotes the operator norm. Let  $E(n;s;r_{1},\dots,r_{s})\subset \mathcal M_{n\times n}(\mathbb{C})$ be the vector subspace comprising block diagonal matrices, defined as follows:
\begin{equation}\label{ls}
    	E=E(n;s;r_{1},...,r_{s}):=\{\operatorname{diag}[z_{1}I_{r_{1}},....,z_{s}I_{r_{s}}]\in \mathcal M_{n\times n}(\mathbb{C}): z_{1},...,z_{s}\in \mathbb{C}\},
\end{equation}
 where $\sum_{i=1}^{s}r_i=n.$ We recall the definition of $\Gamma_{E(3; 3; 1, 1, 1)}$, $\Gamma_{E(3; 2; 1, 2)}$ and $\Gamma_{E(2; 2; 1, 1)}$ \cite{Abouhajar,Bharali, apal1}. The sets $\Gamma_{E{(2;2;1,1)}}$, $\Gamma_{E(3; 3; 1, 1, 1)}$ and $\Gamma_{E(3; 2; 1, 2)}$ are defined as 
 \begin{equation*}
\begin{aligned}
\Gamma_{E{(2;2;1,1)}}:=\Big \{\textbf{x}=(x_1=a_{11}, x_2=a_{22}, x_3=a_{11}a_{22}-a_{12}a_{21}=\det A)\in \mathbb C^3: &\\A\in \mathcal M_{2\times 2}(\mathbb C)~{\rm{and}}~\mu_{E(2;2;1,1)}(A)\leq 1\Big \},
\end{aligned}
\end{equation*}	 
 \begin{equation*}
\begin{aligned}
\Gamma_{E{(3;3;1,1,1)}}:=\Big \{\textbf{x}=(x_1=a_{11}, x_2=a_{22}, x_3=a_{11}a_{22}-a_{12}a_{21}, x_4=a_{33}, x_5=a_{11}a_{33}-a_{13}a_{31},&\\ x_6=a_{22}a_{33}-a_{23}a_{32},x_7=\det A)\in \mathbb C^7: A\in \mathcal M_{3\times 3}(\mathbb C)~{\rm{and}}~\mu_{E(3;3;1,1,1)}(A)\leq 1\Big \}
\end{aligned}
\end{equation*}	
$${\rm{and}}$$  
\begin{equation*}
\begin{aligned}
\Gamma_{E(3;2;1,2)}:=\Big\{( x_1=a_{11},x_2=\det \left(\begin{smallmatrix} a_{11} & a_{12}\\
					a_{21} & a_{22}
				\end{smallmatrix}\right)+\det \left(\begin{smallmatrix}
					a_{11} & a_{13}\\
					a_{31} & a_{33}
				\end{smallmatrix}\right),x_3=\operatorname{det}A, y_1=a_{22}+a_{33}, &\\ y_2=\det  \left(\begin{smallmatrix}
					a_{22} & a_{23}\\
					a_{32} & a_{33}\end{smallmatrix})\right)\in \mathbb C^5
:A\in \mathcal M_{3\times 3}(\mathbb C)~{\rm{and}}~\mu_{E(3;2;1,2)}(A)\leq 1\Big\}.
\end{aligned}
\end{equation*}

The sets $\Gamma_{E(3; 2; 1, 2)}$ and $\Gamma_{E(2; 2; 1, 1)}$  are referred to as $\mu_{1,3}-$\textit{quotient} and tetrablock, respectively \cite{Abouhajar, Bharali}. Young studied the symmetrized bidisc $\Gamma_{E(2;1;2)}$ and the tetrablock $\Gamma_{E(2; 2;1,1)}$, in collaboration with several co-authors \cite{Abouhajar,JAgler,young,ay,ay1,JAY, JANY}, from an operator theoretic point of view. Agler and Young established normal dilation for a pair of commuting operators with the symmetrized bidisc as a spectral set \cite{JAgler, young}. Various authors have also investigated symmetrized polydisc $\Gamma_n$ and studied the properties of $\Gamma_n$-isometries, $\Gamma_n$-unitaries, the Wold decomposition, and conditional dilation of $\Gamma_n$-contractions \cite{SS, A. Pal}. T. Bhattacharyya studied the tetrablock isometries, tetrablock unitaries, the Wold decomposition for tetrablock, and conditional dilation for tetrablock contarctions \cite{Bhattacharyya}. However, whether the tetrablock and $\Gamma_n, n>3,$ have the property $P$ remains unresolved.

Let  $$K=\{\textbf{x}=(x_1,\ldots,x_7)\in \Gamma_{E(3;3;1,1,1)} :x_1=\bar{x}_6x_7, x_3=\bar{x}_4x_7,x_5=\bar{x}_2x_7 ~{\rm{and}}~|x_7|=1\}$$ 
$$\rm{and}$$
\[ K_1 = \{x = (x_1, x_2, x_3, y_1, y_2) \in \Gamma_{E(3;2;1,2)} : x_1 = \overline{y}_2 x_3, x_2 = \overline{y}_1 x_3, |x_3| = 1 \}. \]

We begin with the following definitions that will be essential for our discussion.

\begin{defn}\label{def-1}
		\begin{enumerate}
			\item If $\Gamma_{E(3; 3; 1, 1, 1)}$ is a spectral set for $\textbf{T} = (T_1, \dots, T_7)$, then the $7$-tuple of commuting bounded operators $\textbf{T}$ defined on a  Hilbert space $\mathcal{H}$ is referred to as a \textit{$\Gamma_{E(3; 3; 1, 1, 1)}$-contraction}.
			
			\item Let $(S_1, S_2, S_3)$ and $(\tilde{S}_1, \tilde{S}_2)$ be tuples of commuting bounded operators defined on a Hilbert space $\mathcal{H}$ with $S_i\tilde{S}_j = \tilde{S}_jS_i$ for $1 \leqslant i \leqslant 3$ and $1 \leqslant j \leqslant 2$. We say that  $\textbf{S} = (S_1, S_2, S_3, \tilde{S}_1, \tilde{S}_2)$ is a $\Gamma_{E(3; 2; 1, 2)}$-contraction if $ \Gamma_{E(3; 2; 1, 2)}$ is a spectral set for $\textbf{S}$.
			
\item A commuting $7$-tuple of normal operators $\textbf{N} = (N_1, \dots, N_7)$ defined on a Hilbert space $\mathcal{H}$ is  a \textit{$\Gamma_{E(3; 3; 1, 1, 1)}$-unitary} if the Taylor joint spectrum $\sigma(\textbf{N})$ is contained in the set $K$.  		
			\item A commuting $5$-tuple of normal operators $\textbf{M} = (M_1, M_2, M_3, \tilde{M}_1, \tilde{M}_2)$ on a Hilbert space $\mathcal{H}$ is referred as a \textit{$\Gamma_{E(3; 2; 1, 2)}$-unitary} if the Taylor joint spectrum $\sigma(\textbf{M})$ is contained in $K_1.$ 
					
\item A  $\Gamma_{E(3; 3; 1, 1, 1)}$-isometry (respectively, $\Gamma_{E(3; 2; 1, 2)}$-isometry) is defined as the restriction of a $\Gamma_{E(3; 3; 1, 1, 1)}$-unitary (respectively, $\Gamma_{E(3; 2; 1, 2)}$-unitary)  to a joint invariant subspace. In other words, a $\Gamma_{E(3; 3; 1, 1, 1)}$-isometry ( respectively, $\Gamma_{E(3; 2; 1, 2)}$-isometry) is a $7$-tuple (respectively, $5$-tuple) of commuting bounded operators that possesses simultaneous extension to a \textit{$\Gamma_{E(3; 3; 1, 1, 1)}$-unitary} (respectively, \textit{$\Gamma_{E(3; 2; 1, 2)}$-unitary}). It is important to observe that a $\Gamma_{E(3; 3; 1, 1, 1)}$-isometry (respectively, $\Gamma_{E(3; 2; 1, 2)}$-isometry ) $\textbf{V}=(V_1\dots,V_7)$ (respectively,  $\textbf{W}=(W_1,W_2,W_3,\tilde{W}_1,\tilde{W}_2)$) consists of commuting subnormal operators with $V_7$ (respectively, $W_3$)  is an isometry.

\item   We say that $\textbf{V}$ (respectively, $\textbf{W}$)   is a pure $\Gamma_{E(3; 3; 1, 1, 1)}$-isometry (respectively, pure $\Gamma_{E(3; 2; 1, 2)}$-isometry) if $V_7$ (respectively, $W_3$) is a  pure isometry,  that is, a shift of some multiplicity.		
\end{enumerate}
	\end{defn}

We denote the unit circle by $\mathbb T.$  Let $\mathcal E$  be a separable Hilbert space. Let  $\mathcal B(\mathcal E)$ denote the space of bounded linear operators on $\mathcal E$ equipped with the operator norm. Let $H^2(\mathcal E)$ denote  the  Hardy space of analytic $\mathcal E$-valued functions defined on the unit disk  $\mathbb D$. Let $ L^2(\mathcal E)$ represent the Hilbert space of square-integrable $\mathcal E$-valued functions on the unit circle $\mathbb T,$ equipped with the natural inner product. The space $H^{\infty}(\mathcal B(\mathcal E))$ consists of bounded analytic $\mathcal B(\mathcal E)$-valued functions defined on $\mathbb D$. Let $L^{\infty}(\mathcal B(\mathcal E))$ denote the space of bounded measurable $\mathcal B(\mathcal E)$-valued functions on $\mathbb T$.  For $\phi \in L^{\infty}(\mathcal B(\mathcal E)),$ the Toeplitz operator associated with the symbol  $\phi$ is denoted by $T_{\phi}$ and is defined as follows: 
$$T_{\phi}f=P_{+}(\phi f), f \in H^2(\mathcal E),$$ where $P_{+} : L^2(\mathcal E) \to H^2(\mathcal E)$ is the orthogonal projecton.  In particular, $T_z$ is the
unilateral shift operator $M_z$ on $H^2(\mathcal E)$  and $T_{\bar{z}}$ is the backward shift $M_z^*$ on $H^2(\mathcal E)$.

In section $2$, we discuss various properties of $\Gamma_{E(3; 3; 1, 1, 1)}$-contractions and $\Gamma_{E(3; 2; 1, 2)}$-contractions. We further establish the existence and uniqueness of the fundamental operators associated with $\Gamma_{E(3; 3; 1, 1, 1)}$-contractions and $\Gamma_{E(3; 2; 1, 2)}$-contractions. In section $3$, we obtain a Beurling-Lax-Halmos type representation for invariant subspaces corresponding to a pure $\Gamma_{E(3; 3; 1, 1, 1)}$-isometry and a pure $\Gamma_{E(3; 2; 1, 2)}$-isometry. We  also provide an alternative proof of a Beurling-Lax-Halmos-type theorem. In section $4$, we  construct a conditional dilation for a $\Gamma_{E(3; 3; 1, 1, 1)}$-contraction and  a $\Gamma_{E(3; 2; 1, 2)}$-contraction. In section $5$, we develop an explicit functional model for a certain subclass of $\Gamma_{E(3; 3; 1, 1, 1)}$-contractions and $\Gamma_{E(3; 2; 1, 2)}$-contractions. In section $6$, we demonstrate that every $\Gamma_{E(3; 3; 1, 1, 1)}$-contraction (respectively, $\Gamma_{E(3; 2; 1, 2)}$-contraction) admits a unique decomposition as a direct sum of a $\Gamma_{E(3; 3; 1, 1, 1)}$-unitary (respectively, $\Gamma_{E(3; 2; 1, 2)}$-unitary) and a completely non-unitary $\Gamma_{E(3; 3; 1, 1, 1)}$-contraction (respectively, completely non-unitary $\Gamma_{E(3; 2; 1, 2)}$-contraction).

\section{More on $\Gamma_{E(3; 3; 1, 1, 1)}$-Contractions and $\Gamma_{E(3; 2; 1, 2)}$-Contractions}
This section begins with a review of the fundamental concepts needed for subsequent discussions. We recall the definitions of spectrum, spectral radius, and numerical radius of an operator. Let $\sigma(T)$ represent the spectrum of $T$, defined as $$\sigma(T):=\{\lambda \in \mathbb{C} \mid T-\lambda I~{\rm{ is~ not ~invertible}}\}.$$ Furthermore, the numerical radius of a bounded operator \( T \) on a Hilbert space \( \mathcal{H} \) is defined by $$\omega(T):=\sup\{|\langle Tx,x\rangle|:\| x \|=1\}.$$  A straightforward calculation yields that $r(T)\leq \omega(T)\leq \|T\|$ for a bounded operator $T$, where the spectral radius is defined as $$r(T):= \sup_{ \lambda \in \sigma(T)}|\lambda|.$$ 
We will state a lemma due to Ando,  which we will refer to several times in this paper. 
\begin{lem}[Lemma $2.9,$ \cite{Roy}]\label{numer}
The numerical radius of an operator $X$ is not greater than one if and only if ${\rm{Re }}zX \leq I$ for all complex numbers $z \in \mathbb T.$ \end{lem}
Let $T$ be a contraction on a Hilbert space $\mathcal H,$ and define the defect operator associated with $T$ as $D_{T}=(I-T^*T)^{\frac{1}{2}}.$ The closure of the range of $ D_{T}$ is represented by $\mathcal D_{T}$. The aim of this section is to study the fundamental equations for $\Gamma_{E(3; 3; 1, 1, 1)}$-contractions and $\Gamma_{E(3; 2; 1, 2)}$-contractions. The definitions of operator functions $\rho_{G_{E(2; 1; 2)}} $ and $\rho_{G_{E(2; 2; 1,1)}} $ for symmetrized bidisc and tetrablock are defined as follows:
$$\rho_{G_{E(2; 1; 2)}} (S,P)=2 (I-P^*P)-(S-S^*P)-(S^*-P^*S) $$ and 
$$\rho_{G_{E(2; 2; 1,1)}} (T_1,T_2,T_3)=(I-T_3^*T_3) -(T_2^*T_2-T_1^*T_1) -2\Re {T_2-T_1^*T_3}, $$ where $P,T_3$ are contractions and $S,P$ and $T_1,T_2,T_3$ are commuting bounded operators defined on Hilbert spaces $\mathcal H_1$ and $\mathcal H_2$, respectively. The aforementioned operator functions are crucial for characterizing $\Gamma_{E(2; 1; 2)}$-contraction and $\Gamma_{E(2; 2; 1,1)}$-contraction, respectively. The function $\rho_{G_{E(2; 2; 1,1)}}$ also plays an important role in study of $\Gamma_{E(3; 3; 1, 1, 1)}$-contractions.

\begin{defn}\label{fundamental}
		Let $(T_1, \dots, T_7)$ be a $7$-tuple of commuting contractions on a Hilbert space $\mathcal{H}. $ The equations 

		\begin{equation}\label{Fundamental 1}
\begin{aligned}
&T_i - T^*_{7-i} T_7 = D_{T_7}F_iD_{T_7}, \;\;\; 1\leq i\leq 6, 
\end{aligned}
\end{equation}
where $F_i\in \mathcal{B}(\mathcal{D}_{T_7}),$ are referred to as the  fundamental equations for $(T_1, \dots, T_7)$.
			\end{defn}
We recall the definition of tetrablock contraction from \cite{Bhattacharyya}.

\begin{defn}

Let $(A,B,P)$ be a commuting triple of bounded operators on a Hilbert space $\mathcal H$. The triple $(A,B,P)$ is called a tetrablock contraction if $\Gamma_{E(2;2;1,1)}$ is a spectral set for $(A,B,P)$. 

\end{defn}
For any $z\in \mathbb C$, we introduce the operators $S^{(i)}_z = T_i + zT_{7-i}$ for $1\leq i \leq 6$ and $P_z = zT_7$.
\begin{thm}\label{fundam}
		Let $\textbf{T} = (T_1, \dots, T_7)$ be a commuting $7$-tuple of bounded operators acting on a Hilbert space $\mathcal{H}$. Then  the following holds $(1) \Rightarrow (2) \Rightarrow (3) \Rightarrow (4) \Rightarrow (5):$
		\begin{enumerate}
			\item $\mathbf{T} = (T_1, \dots, T_7)$ be a $\Gamma_{E(3; 3; 1, 1, 1)}$-contraction.
			
			\item $(T_i, T_{7-i}, T_7)$ is a $\Gamma_{E(2; 2; 1, 1)}$-contraction for $1 \leq i \leq 6$.
			
			\item For $1\leq i \leq  6$ and $z \in \mathbb{T}$,
			\begin{equation*}
				\begin{aligned}
					&\rho_{G_{E(2; 2; 1, 1)}}(T_i, zT_{7-i}, zT_7) \geqslant 0, ~\text{and}~ \rho_{G_{E(2; 2; 1, 1)}}(T_{7-i}, zT_i, zT_7) \geqslant 0
				\end{aligned}
			\end{equation*}
			and the spectral radius of $S^{(i)}_z$ is not bigger than $2$ for $1 \leq i \leq 6.$
			
			\item The pair $(S^{(i)}_z, P_z),1\leq i \leq 6,$ is a $\Gamma_{E(2; 1; 2)}$-contraction for every $z \in \mathbb{T}$.
			
			\item The fundamental equations in \eqref{Fundamental 1} have unique solutions $F_i$ and $F_{7-i}$ in $\mathcal{B}(\mathcal{D}_{T_7})$ for $1\leq i \leq 6.$  Moreover, the operator $F_i + zF_{7-i}, 1\leq i \leq 6,$ has numerical radius not bigger than $1$ for every $z \in \mathbb{T}$.
		\end{enumerate}
	\end{thm}
	
	\begin{proof}
The implication $(1) \Rightarrow (2)$ follows from [Proposition $2.11$, \cite{apal2}]. We derive the implication $(2) \Rightarrow (3) \Rightarrow (4) \Rightarrow (5)$ from [\cite{Bhattacharyya}, Theorem $3.4$]. This completes the proof.
	\end{proof}
	
Let $(S_1, S_2, S_3, \tilde{S}_1, \tilde{S}_2)$ be a $5$-tuple of commuting bounded operators defined on some Hilbert space $\mathcal H$. We need two operator functions for the commuting tuple $(S_1, S_2, S_3, \tilde{S}_1, \tilde{S}_2)$ of bounded operators with $\|S_3\|\leq 1$, which is essential for characterising the $\Gamma_{E(3; 2; 1, 2)}$-contractions, namely, 
\small{\begin{equation}\label{fund1}
				\begin{aligned}
					&\rho_{G_{E(2; 2; 1,1)}} (S_1,\tilde{S}_2,S_3)=(I-S_3^*S_3) -(\tilde{S}_2^*\tilde{S}_2-S_1^*S_1) -2\Re {\tilde{S}_2-S_1^*S_3}
				\end{aligned}
			\end{equation}}
$${\rm{and}}$$	
\small{\begin{equation}\label{fund11}
				\begin{aligned}
					&\rho_{G_{E(2; 2; 1,1)}} \left(\frac{S_2}{2},\frac{\tilde{S}_1}{2},S_3\right)=(I-S_3^*S_3) -\frac{(\tilde{S}_1^*\tilde{S}_1-S_2^*S_2)}{4} -\frac{\Re {\tilde{S}_1-S_2^*S_3}}{2}
				\end{aligned}
			\end{equation}}
\begin{defn}\label{fundamental}
Let $(S_1, S_2, S_3, \tilde{S}_1, \tilde{S}_2)$ be a $5$-tuple of commuting bounded operators defined on a Hilbert space $\mathcal H$. The equations
\begin{equation}
			\begin{aligned}\label{funda1}
					&S_1 - \tilde{S}^*_2S_3 = D_{S_3}G_1D_{S_3},\,\, \tilde{S}_2 - S^*_1S_3 = D_{S_3}\tilde{G}_2D_{S_3},
				\end{aligned}
			\end{equation}
			$${\rm{and}}$$
		\begin{equation}
				\begin{aligned}\label{funda11}
				&\frac{S_2}{2} - \frac{\tilde{S}^*_1}{2}S_3 = D_{S_3}G_2D_{S_3}, \,\, \frac{\tilde{S}_1}{2} - \frac{S^*_2}{2}S_3 = D_{S_3}\tilde{G}_1D_{S_3},				\end{aligned}
			\end{equation}
where $G_1,2G_2,2\tilde{G}_1$ and $\tilde{G}_2$ in $\mathcal{B}(\mathcal{D}_{S_3}),$ are referred to as the  fundamental equations for $(S_1, S_2, S_3, \tilde{S}_1, \tilde{S}_2)$.		\end{defn}
For any $z\in \mathbb C,$ we define the operators $\tilde{S}_z = S_1 + z\tilde{S}_2, \tilde{P}_z = zS_3$ and $\hat{S}_z = \frac{S_2}{2} + z\frac{\tilde{S}_1}{2}, \hat{P}_z = zS_3.$ 
	
	\begin{thm}\label{s1s2}
		Let $(S_1, S_2, S_3, \tilde{S}_1, \tilde{S}_2)$ be a $5$-tuple of commuting bounded operators defined on some Hilbert space $\mathcal H$. Then the following holds $(1) \Rightarrow (2) \Rightarrow (3) \Rightarrow (4) \Rightarrow (5):$
		\begin{enumerate}
			\item $\mathbf{S} = (S_1, S_2, S_3, \tilde{S}_1, \tilde{S}_2)$ is a $\Gamma_{E(3; 2; 1, 2)}$-contraction.
			
			\item $(S_1, \tilde{S}_2, S_3)$ and $(\frac{S_2}{2}, \frac{\tilde{S}_1}{2}, S_3)$ are $\Gamma_{E(2; 2; 1, 1)}$-contractions.
	\item For every $z \in \mathbb{T}$, we have
			\begin{equation}
				\begin{aligned}
					&\rho_{G_{E(2; 2; 1, 1)}}(S_1, z\tilde{S}_2, zS_3) \geqslant 0 ~\text{and}~
					\rho_{G_{E(2; 2; 1, 1)}}(\tilde{S}_2, zS_1, zS_3) \geqslant 0,
				\end{aligned} \end{equation}
				\begin{equation}
				\begin{aligned}
					&\rho_{G_{E(2; 2; 1, 1)}}\left(\frac{S_2}{2}, z\frac{\tilde{S}_1}{2}, zS_3\right) \geqslant 0 ~\text{and}~
					\rho_{G_{E(2; 2; 1, 1)}}\left(\frac{\tilde{S}_1}{2}, z\frac{S_2}{2}, zS_3\right) \geqslant 0
				\end{aligned}
			\end{equation}
			and the spectral radius of $\tilde{S}_z$  and $\hat{S}_z$ are not bigger than $2$.
			
			\item The pair of operators $(\tilde{S}_z, \tilde{P}_z)$ and  $(\hat{S}_z, \hat{P}_z)$ are  $\Gamma_{E(2; 1; 2)}$-contractions for every $z \in \mathbb{T}$.
			
			\item The fundamental equations in \eqref{funda1} and \eqref{funda11} have unique solutions $G_1, \tilde{G}_2$ and $G_2,\tilde{G}_1$ in $\mathcal{B}(\mathcal{D}_{S_3})$, respectively. Moreover, the operators $G_1 + z\tilde{G}_2$ and $G_2 + z\tilde{G}_1$ have numerical radius not bigger than $1$ for every $z \in \mathbb{T}$.
		\end{enumerate}
	\end{thm}
	
	\begin{proof}
The implication $(1) \Rightarrow (2)$ is obtained from [Proposition $2.13$, \cite{apal2}]. We deduce the implication $(2) \Rightarrow (3) \Rightarrow (4) \Rightarrow (5)$ by using [\cite{Bhattacharyya}, Theorem $3.4$]. This completes the proof.	
	
	\end{proof}

The following lemmas provide the existence of unique fundamental operators for a $\Gamma_{E(3; 3; 1, 1, 1)}$-contraction and a $\Gamma_{E(3; 2; 1, 2)}$-contraction. These results play a crucial role in constructing conditional dilation for a $\Gamma_{E(3; 3; 1, 1, 1)}$-contraction and a $\Gamma_{E(3; 2; 1, 2)}$-contraction.
\begin{lem}\label{FiFj}
	The fundamental operators  of a $\Gamma_{E(3; 3; 1, 1, 1)}$-contraction $\textbf{T} = (T_1, \dots, T_7)$ are the unique bounded linear operators $X_i$ and $X_{7-i}$, $1\leq i \leq 6$, defined on $\mathcal D_{T_7}$ satisfying the operator equations
\begin{equation}
			\begin{aligned}
				&D_{T_7}T_i = X_iD_{T_7} + X^*_{7-i}D_{T_7}T_7 ~\text{and}~ D_{T_7}T_{7-i} = X_{7-i}D_{T_7} + X^*_iD_{T_7}T_7~{\rm{for}}~1\leq i \leq 6.
			\end{aligned}
		\end{equation}
	\end{lem}
	
	\begin{proof}
By Theorem \ref{fundam}, it follows that $(T_i, T_{7-i}, T_7)$ is a $\Gamma_{E(2; 2; 1, 1)}$-contraction for $1\leq i \leq 6$. As $(T_i, T_{7-i}, T_7)$ is a $\Gamma_{E(2; 2; 1, 1)}$-contraction, $1\leq i \leq 6$, it implies from [Corollary $4.2$, \cite{Bhattacharyya}] that there exist unique bounded linear operators $X_i$ and $X_{7-i}$, $1\leq i \leq 6$, on $\mathcal D_{T_7}$ that satisfy the following operator equations
\begin{equation}\label{fun}
\begin{aligned}
				&D_{T_7}T_i = X_iD_{T_7} + X^*_{7-i} D_{T_7}T_7 ~\text{and}~ D_{T_7}T_{7-i} = X_{7-i}D_{T_7} + X^*_iD_{T_7}T_7~{\rm{for}}~1\leq i \leq 6.
\end{aligned}
\end{equation}		This completes the proof.
	\end{proof}
	
	\begin{lem}\label{F12}
		Let $\textbf{T} = (T_1, \dots, T_7) $ be a $\Gamma_{E(3; 3; 1, 1, 1)} $-contraction on the Hilbert space $\mathcal{H}$ with commuting fundamental operators $F_i, 1\leq i \leq 6,$ defined on $\mathcal{D}_{T_7}. $ Then
		\begin{equation}\label{F_i}
			\begin{aligned}
				T_i^*T_i - T_{7-i}^*T_{7-i} = D_{T_7}(F^*_iF_i - F^*_{7-i}F_{7-i})D_{T_7},1\leq i \leq 6.
			\end{aligned}
		\end{equation}
	\end{lem}
	
	\begin{proof}
		Since $\textbf{T}$ is a $\Gamma_{E(3; 3; 1, 1, 1)}$-contraction, it yields from Theorem \ref{fundam} that the tuple $(T_i, T_{7-i}, T_7)$ is a $\Gamma_{E(2; 2; 1, 1)}$-contraction for $1\leq i\leq 6$. Hence, by [Corollary $4.4$, \cite{Bhattacharyya}], we deduce that
\begin{equation*}
\begin{aligned}
				T_i^*T_i - T_{7-i}^*T_{7-i} = D_{T_7}(F^*_iF_i - F ^*_{7-i}F_{7-i})D_{T_7}, 1\leq i \leq 6.
\end{aligned}
\end{equation*}
		This completes the proof.
	\end{proof}
	
The proof of following lemmas is identical to that of Lemma \ref{FiFj} and Lemma \ref{F12}. We therefore omit the proofs.	
	\begin{lem}\label{s1s3}
The fundamental operators  of a $\Gamma_{E(3; 2; 1, 2)}$-contraction $\textbf{S} = (S_1, S_2, S_3, \tilde{S}_1, \tilde{S}_2)$ are the unique operators $G_1,\tilde{G}_2,G_2$ and $\tilde{G}_1$  defined on $\mathcal{D}_{S_3}$ which satisfy the following operator equations
		\begin{equation}\label{s3}
			\begin{aligned}
				&D_{S_3}S_1 = G_1D_{S_3} + \tilde{G}_2^*D_{S_3}S_3, \,\, D_{S_3}\tilde{S}_2 = \tilde{G}_2D_{S_3} + G_1^*D_{S_3}S_3, \\&\,\, ~~~~~~~~~~~~~~~~~~~~~~~~~~~~~~~~~~~~~~~~~~~~\,\,\,\,\,\,\,\,\,\,\,\,\,\,\,\,\,\,\,\,\,\,\,\,\,\,\,\,\,\,\,\,\,\,\,\,\,\,\,\,\,\,\,\,\,\,\,\,\,\,\,\,\ \text{and}\\
				&D_{S_3}\frac{S_2}{2} = G_2D_{S_3} + \tilde{G}^*_1D_{S_3}S_3, \,\, D_{S_3}\frac{\tilde{S}_1}{2} = \tilde{G}_1D_{S_3} + G^*_2D_{S_3}S_3.
			\end{aligned}
		\end{equation}
		
	\end{lem}
	
	\begin{lem}
		Let $\textbf{S} = (S_1, S_2, S_3, \tilde{S}_1, \tilde{S}_2)$ be a $\Gamma_{E(3; 2; 1, 2)}$-contraction with commuting fundamental operators $G_1, \tilde{G}_2, G_2$ and $\tilde{G}_1$  defined on $\mathcal{D}_{S_3}$. Then
		\begin{equation}
			\begin{aligned}
				&S_1^*S_1 - \tilde{S}^*_2\tilde{S}_2 = D_{S_3}(G^*_1G_1 - \tilde{G}^*_2\tilde{G}_2)D_{S_3}, \\&\,\, ~~~~~~~~~~~~~~~~~~~~~~~~~~~~~~~~~~~~~~~~~~~~\,\,\,\,\,\,\,\,\,\,\,\,\,\,\,\,\,\,\,\,\,\,\,\,\,\,\,\,\,\,\,\,\,\,\,\,\,\,\,\,\,\,\,\,\,\,\,\,\,\,\, \text{and}\\
				&\frac{S^*_2S_2 - \tilde{S}^*_1\tilde{S}_1}{4} = D_{S_3}(G^*_2G_2 - \tilde{G}^*_1\tilde{G}_1)D_{S_3}.
			\end{aligned}
		\end{equation}
	\end{lem}

\section{Beurling-Lax-Halmos Representation of Pure $\Gamma_{E(3; 3; 1, 1, 1)}$-Isometries\\ and Pure $\Gamma_{E(3; 2; 1, 2)}$-Isometries}
This section aims to characterize the joint invariant subspace of pure $\Gamma_{E(3; 3; 1, 1, 1)}$-isometries and pure $\Gamma_{E(3; 2; 1, 2)}$-isometries. Various authors discussed the Beurling-Lax-Halmos theorem for pure $\Gamma_n$-isometries, $n\geq 2$ [see \cite{SS,Sarkar}]. In [Theorem $4.6$, Theorem $4.7$, \cite{apal2}], we characterize the pure $\Gamma_{E(3; 3; 1, 1, 1)}$-isometries and pure $\Gamma_{E(3; 2; 1, 2)}$-isometries. We discuss pure $\Gamma_{E(3; 3; 1, 1, 1)}$-isometries and pure $\Gamma_{E(3; 2; 1, 2)}$-isometries and their related parameters
by using Theorem $4.6$ and Theorem $4.7$ in \cite{apal2}. We will identify $H^{2}(\mathcal E)$ with $H^{2}(\mathbb D)\otimes \mathcal E$ using the mapping $z^n\xi \to z^n\otimes \xi,$ where $n$ is a non-negative integer and $\xi$ is an element of $\mathcal E$, as and when needed.

	\begin{thm}\label{Ai}
		Let $\varphi_i(z) = A_i + A^*_{7-i}z $ and $\tilde{\varphi_i}(z) = \tilde{A}_i + \tilde{A}^*_{7-i}z $ be in $H^{\infty}(\mathcal{B}(\mathcal{E}))$ and $H^{\infty}(\mathcal{B}(\mathcal{F})),$ respectively, for some $A_i \in \mathcal{B}(\mathcal{E})$ and $\tilde{A}_i \in \mathcal{B}(\mathcal{F}),$ $1\leq i \leq 6,$ respectively. Then $(M_{\varphi_1}, \dots, M_{\varphi_6}, M_z)$ on $H^2(\mathbb{D}) \otimes \mathcal{E}$ is unitarily equivalent to  $(M_{\tilde{\varphi}_1}, \dots, M_{\tilde{\varphi}_6}, M_z) $ on $H^2(\mathbb{D}) \otimes \mathcal{F}$ if and only if $(A_1, \dots, A_6)$ is unitarily equivalent to $(\tilde{A}_1, \dots, \tilde{A}_6). $
	\end{thm}
	\begin{proof}
Assume that $(M_{\varphi_1}, \dots, M_{\varphi_6}, M_z)$ on $H^2( \mathcal{E})$ is unitarily equivalent to $(M_{\tilde{\varphi}_1}, \dots, M_{\tilde{\varphi}_6}, M_z)$ on $H^2( \mathcal{F})$. So, there is a unitary operator $\tilde{U}:  H^2( \mathcal{E})\to H^2( \mathcal{F})$ such that
\begin{eqnarray}\label{1U}
\tilde{U}^*  M_{\varphi_i}\tilde{U}= M_{\tilde{\varphi}_i},\;\; 1\leq i\leq 6,\;\;\mbox{and}\;\; \tilde{U}^*M_z \tilde{U} = M_z.
\end{eqnarray}

Since $H^{2}(\mathcal E)$ is identified with $H^{2}(\mathbb D)\otimes \mathcal E$ and $H^2( \mathcal{F})$ is identified with $H^2(\mathbb{D}) \otimes \mathcal{F}$, so the operators $ M_{\varphi_i}$ and $M_{\tilde{\varphi}_i}$ are identified with the operators $I_{H^2} \otimes A_i + M_z \otimes A^*_{7-i}$ and $I_{H^2} \otimes \tilde{A}_i + M_z \otimes \tilde{A}^*_{7-i}$, respectively, that is,
\begin{equation}\label{2U}
\begin{aligned}
				 M_{\varphi_i} &= I_{H^2} \otimes A_i + M_z \otimes A^*_{7-i} ~\text{and}~  M_{\tilde{\varphi}_i} = I_{H^2} \otimes \tilde{A}_i + M_z \otimes \tilde{A}^*_{7-i}, ~1\leq  i \leq 6,
\end{aligned}
\end{equation} for all $z\in \mathbb {T}.$
From \eqref{1U} and \eqref{2U}, we get
\begin{equation}\label{UU}
\begin{aligned}
				\tilde{U}^*(I_{H^2} \otimes A_i + M_z \otimes A^*_{7-i}) \tilde{U}
				&= I_{H^2} \otimes \tilde{A}_i + M_z \otimes \tilde{A}^*_{7-i},~1 \leq  i \leq 6
\end{aligned}
\end{equation}
and
\begin{equation}\label{U1}
\begin{aligned}
				\tilde{U}^*(M_z \otimes I_{\mathcal{E}})\tilde{U} &= M_z \otimes I_{\mathcal{F}}.
\end{aligned}
\end{equation}
It follows from \eqref{U1} that there exists a unitary $U : \mathcal{E} \to \mathcal{F}$ such that $\tilde{U} = I_{H^2} \otimes U$. Thus, from \eqref{UU}, we have
$$I_{H^2} \otimes U^*A_i U+ M_z \otimes U^*A^*_{7-i}U 
				= I_{H^2} \otimes \tilde{A}_i + M_z \otimes \tilde{A}^*_{7-i},~1 \leq  i \leq 6,$$
that is,
\begin{equation}\label{U2}
\begin{aligned}
				U^*A_iU + U^*A^*_{7-i}Uz &= \tilde{A}_i + \tilde{A}^*_{7-i}z,~1\leq  i \leq  6.
\end{aligned}
\end{equation}
By comparing the coefficients of \eqref{U2}, we obtain $U^*A_iU=\tilde{A}_i, 1\leq i \leq 6.$ Consequently, we conclude that $(A_1, \dots, A_6)$ and $(\tilde{A}_1, \dots, \tilde{A}_6)$ are unitarily equivalent.
		
		\vspace{0.3cm}
		
Conversely, suppose there exists a unitary operator $U : \mathcal{E} \to \mathcal{F}$ such that $U^*A_iU = \tilde{A}_i$ for $1 \leq i \leq 6$. Let $\tilde{U} : H^2(\mathbb{D}) \otimes \mathcal{E} \to H^2(\mathbb{D}) \otimes \mathcal{F}$ be the map defined by $\tilde{U} = I_{H^2} \otimes U$. Clearly, $\tilde{U}$ is unitary. Note that for $1\leq i \leq 6,$

\begin{equation}\label{U3}
\begin{aligned} \tilde {U}^*  M_{\varphi_i} \tilde{U}
&=\tilde{U}^*(I_{H^2} \otimes A_i + M_z \otimes A^*_{7-i}) \tilde{U}\\
&= (I_{H^2} \otimes U)^*(I_{H^2} \otimes A_i + M_z \otimes A^*_{7-i})(I_{H^2} \otimes U)\\
&= I_{H^2} \otimes U^*A_iU + M_z \otimes U^*A^*_{7-i} U\\
&= I_{H^2} \otimes \tilde{A}_i + M_z \otimes \tilde{A}^*_{7-i}\\
&= M_{\tilde{\varphi}_i}
.\end{aligned}
\end{equation}
It yields from \eqref{U3} that $(M_{\varphi_1}, \dots, M_{\varphi_6}, M_z)$ on $H^2(\mathcal{E})$ and $(M_{\tilde{\varphi}_1}, \dots, M_{\tilde{\varphi}_6}, M_z)$  on $H^2(\mathcal{F})$ are unitarily equivalent. This completes the proof.
	\end{proof}
	
	As a consequence of the above theorem, we will prove the following corollary. This corollary is crucial for demonstrating the Beurling-Lax-Halmos theorem, which will be addressed subsequently.
	
	\begin{cor}\label{uii}
		Let $(V_1, \dots, V_7)$ and $(\tilde{V}_1, \dots, \tilde{V}_7)$ be two pure $\Gamma_{E(3; 3; 1, 1, 1)}$-isometries. Then $(V_1, \dots, V_7)$ is unitarily equivalent to $(\tilde{V}_1, \dots, \tilde{V}_7)$ if and only if $(V_1^*-V_6V_7^*,V_2^*-V_5V_7^*,\ldots, V_6^*-V_1V_7^*)$is unitarily equivalent to $(\tilde{V}^*_1 - \tilde{V}_{6}\tilde{V}^*_7, \tilde{V}^*_2 - \tilde{V}_{5}\tilde{V}^*_7,\ldots, \tilde{V}^*_6 - \tilde{V}_{1}\tilde{V}^*_7)$.
	\end{cor}
	
	\begin{proof}
Since $(V_1, \dots, V_7)$ is a pure $\Gamma_{E(3; 3; 1, 1, 1)}$-isometry, it follows from [Theorem $4.6$, \cite{apal2}] that $(V_1, \dots, V_7)$ is unitarily equivalent to $(M_{\varphi_1}, \dots, M_{\varphi_6}, M_z),$ where $\varphi_i(z) = A_i + A^*_{7-i}z $ in $H^{\infty}(\mathcal{B}(\mathcal{E})) $ for some $A_i \in \mathcal{B}(\mathcal{E}), 1\leq i \leq 6,$ satsfying 

	\begin{enumerate}
			
			\item the $H^{\infty}$ norm of the operator valued functions $A_i+A_{7-i}^*z$ is at most $1$ for all $z\in \mathbb T, 1\leq i \leq 6;$ 
			
			\item $[A_i, A_j] = 0$ and $ [A_i, A^*_{7-j}] = [A_j, A^*_{7-i}]$ for $1\leq i, j \leq  6$.
		\end{enumerate}


Clearly, $(V_1^*-V_6V_7^*,V_2^*-V_5V_7^*,\ldots, V_6^*-V_1V_7^*)$ is unitarily equivalent to $(M^*_{\varphi_1} - M_{\varphi_{6}}M^*_z, M^*_{\varphi_2} - M_{\varphi_{5}}M^*_z, \ldots, M^*_{\varphi_6} - M_{\varphi_{1}}M^*_z)$. Note that
		\begin{equation}\label{Uq}
			\begin{aligned}
				M^*_{\varphi_i} - M_{\varphi_{7-i}}M^*_z
				&= I_{H^2} \otimes A^*_i + M^*_z \otimes A_{7-i} - (I_{H^2} \otimes A_{7-i} + M_z \otimes A^*_i)(M^*_z \otimes I_{\mathcal{E}})\\
				&= I_{H^2} \otimes A^*_i + M^*_z \otimes A_{7-i} - M^*_z \otimes A_{7-i} - M_zM^*_z \otimes A^*_i\\
				&= (I_{H^2} - M_zM^*_z) \otimes A^*_i\\
				&= \mathbb{P}_{\mathbb{C}} \otimes A^*_i,
			\end{aligned}
			\end{equation}
where $\mathbb{P}_{\mathbb{C}}$ is  the orthogonal projection from $H^{2}(\mathbb D)$ to the scalars in $H^{2}(\mathbb D).$ Similarly, we have
		\begin{equation}\label{Uq1}
			\begin{aligned}
				M^*_{\tilde{\varphi}_i} - M_{\tilde{\varphi}_{7-i}}M^*_z
				&= \mathbb{P}_{\mathbb{C}} \otimes \tilde{A}_i^*.
			\end{aligned}
		\end{equation}

Theorem \ref{Ai} states that  $(M_{\varphi_1}, \dots, M_{\varphi_6}, M_z)$ and $(M_{\tilde{\varphi}_1}, \dots, M_{\tilde{\varphi}_6}, M_z)$ are unitarily equivalent if and only if  $(A_1, \dots, A_6)$ and $(\tilde{A}_1, \dots, \tilde{A}_6)$ are unitarily equivalent. From \eqref{Uq} and \eqref{Uq1}, it  follows that the unitary equivalence of the tuple $(M_{\varphi_1}, \dots, M_{\varphi_6}, M_z)$ depends on the unitary equivalence of $6$-tuples $(M^*_{\varphi_1} - M_{\varphi_{6}}M^*_z, M^*_{\varphi_2} - M_{\varphi_{5}}M^*_z, \ldots, M^*_{\varphi_6} - M_{\varphi_{1}}M^*_z)$ and vice-versa. Thus, we conclude that the unitary equivalence of $(V_1, \dots, V_7)$ is demonstrated via the unitary equivalence of $(V_1^*-V_6V_7^*,V_2^*-V_5V_7^*,\ldots, V_6^*-V_1V_7^*)$. This completes the proof.
	\end{proof}

The characterization of a joint invariant subspace of a $\Gamma_{E(3; 3; 1, 1, 1)}$-isometry, by the Wold-type decomposition theorem,  reduces to characterization of  a joint invariant subspace of a pure $\Gamma_{E(3; 3; 1, 1, 1)}$-isometry. The following theorem  describes a non-zero joint invariant subspace  of a pure $\Gamma_{E(3; 3; 1, 1, 1)}$-isometry.

A function $\Theta \in H^{\infty} \mathcal B(\mathcal E,\mathcal F)$ is said to be inner if $M_\Theta$ is an isometry almost everywhere on $\mathbb T.$ A closed subspace $\mathcal M$  in $H^2(\mathcal E)$ is called $(M_{\varphi_1}, \dots, M_{\varphi_6}, M_z)$-invariant if it remains invariant under the operators $M_{\varphi_i}, 1\leq i \leq 6$, and $M_z$. Let $\mathcal M$ be a non-zero closed subspace of $H^{2}(\mathcal F). $ The Beurling-Lax-Halmos theorem says that $\mathcal M$ is invariant under operator $M_z$ if and only if there exists a Hilbert space $\mathcal E$ and an inner function $\Theta \in H^{\infty} \mathcal( B(\mathcal E,\mathcal F))$ such that $$\mathcal M=M_{\Theta}H^{2}(\mathcal E). $$

	\begin{thm}\label{BLH}
		Let $\mathcal{M}$ be a non-zero closed subspace of the vector-valued Hardy space $H^2( \mathcal{F})$ and $(M_{\varphi_1}, \dots, M_{\varphi_6}, M_z)$ be a pure $\Gamma_{E(3; 3; 1, 1, 1)}$-isometry on  $H^2( \mathcal{F})$, where $\varphi_i \in H^{\infty} (\mathcal B(\mathcal F)), \;1\leq i\leq 6$. Then $\mathcal{M}$ is an invariant subspace of the pure $\Gamma_{E(3; 3; 1, 1, 1)}$-isometry $(M_{\varphi_1}, \dots, M_{\varphi_6}, M_z),$  if and only if there exist a Hilbert space $\mathcal{E}$, $\psi_i \in H^{\infty}(\mathcal{E}), \,\, 1 \leq i \leq 6,$ and an inner function $\Theta \in H^{\infty}(\mathcal{B}(\mathcal{E}, \mathcal{F}))$ such that 
\begin{enumerate}
\item $(M_{\psi_1}, \dots, M_{\psi_6}, M_z)$ is a pure $\Gamma_{E(3; 3; 1, 1, 1)}$-isometry on $H^2( \mathcal{E})$,
\item $M_{\Theta}M_{\psi_i} = M_{\varphi_i}M_{\Theta},\,\,1\leq  i \leq 6$,
\item $\mathcal M=M_{\Theta}H^{2}(\mathcal E).$
\end{enumerate}
\end{thm}
	
	\begin{proof}
Let $\mathcal{M} \neq \{0\}$ denote a non-trivial joint invariant subspace of $(M_{\varphi_{1}}, \dots, M_{\varphi_{6}}, M_z)$. By definition, $\mathcal{M}$ is invariant under $M_z$. Thus, by the \textit{Beurling-Lax-Halmos Theorem} for $M_z$, it follows that $\mathcal{M} = M_{\Theta}(H^2(\mathcal{E}))$, where $\Theta \in H^{\infty}(\mathcal{B}(\mathcal{E}, \mathcal{F}))$ denotes the inner function. Since $\mathcal{M}$ is invariant under $M_{\varphi_i}, 1\leq i \leq 6,$ we have
		\begin{equation*}
			\begin{aligned}
				M_{\varphi_i}\mathcal{M}
				&= M_{\varphi_i}M_{\Theta}(H^2( \mathcal{E}))
				\subseteq M_{\Theta}(H^2( \mathcal{E})), 1\leq i \leq 6.
			\end{aligned}
		\end{equation*}
		By Douglas' lemma \cite{Dou} there exists an operator $X_i$ defined on  $H^2( \mathcal{E})$ such that
\begin{eqnarray}\label{Doulem}
M_{\varphi_i}M_{\Theta} = M_{\Theta}X_i,\;\;\; 1\leq i \leq 6.
\end{eqnarray}
 As $\Theta$ is an inner function, it follows  $M_{\Theta}$ is also an isometry. As $M_{\Theta}$ is an isometry, we have
		\begin{equation}\label{iso}
			\begin{aligned}
				M^*_{\Theta}M^*_{\varphi_i}M_{\Theta} = X^*_i
				&\Rightarrow
				M^*_zM^*_{\Theta}M^*_{\varphi_i}M_{\Theta} = M^*_zX^*_i\\
				&\Rightarrow
				M^*_z(M^*_{\Theta}M_{\varphi_i}M_{\Theta})^* = M^*_zX^*_i\\
				&\Rightarrow
				(M^*_{\Theta}M_{\varphi_i}M_{\Theta})^*M^*_z = M^*_zX^*_i\\
				&\Rightarrow
				X^*_iM^*_z = M^*_zX^*_i.
			\end{aligned}
		\end{equation}
As $X_i$'s commute with $M_z$, there exists $\psi_i \in H^{\infty}(\mathcal{E})$ such that $X_i = M_{\psi_i}, 1\leq i \leq 6.$  Consequently, from \eqref{Doulem}, we get 
		\begin{equation}\label{psi1}
			\begin{aligned}
				M_{\varphi_i}M_{\Theta} = M_{\Theta}M_{\psi_i},\;\;\; 1\leq i \leq 6.
			\end{aligned}
		\end{equation}
This shows that $\varphi _i\Theta=\Theta \psi_i, 1\leq i \leq 6.$	In order to complete the proof, we need to show that  $(M_{\psi_1}, \dots, M_{\psi_6}, M_z)$ is a pure $\Gamma_{E(3; 3; 1, 1, 1)}$-isometry. 	As $\varphi_i$'s commute, from \eqref{psi1}, we have
\begin{equation}
			\begin{aligned}\label{psi11}
				M_{\Theta}M_{\psi_i}M_{\psi_j}
				&= M_{\varphi_i}M_{\Theta}M_{\psi_j} \\
				&= M_{\varphi_i}M_{\varphi_j}M_{\Theta}\\
				&=M_{\varphi_j}M_{\varphi_i}M_{\Theta}\\
				&=M_{\Theta}M_{\psi_j}M_{\psi_i} 
			\end{aligned}
		\end{equation}
for $1\leq i,j \leq 6.$	From \eqref{psi11}, we conclude that $M_{\psi_i}$'s commute with each other, which implies that $\psi_i\psi_j=\psi_j\psi_i$ for $1\leq i,j\leq 6.$ Note that for $1\leq i \leq 6,$
		\begin{equation*}
			\begin{aligned}
				M_{\psi_i} &= M^*_{\Theta}M_{\Theta}M_{\psi_i} \\&= M^*_{\Theta}M_{\varphi_i}M_{\Theta}\\
				&= M^*_{\Theta}M^*_{\varphi_{7-i}}M_zM_{\Theta} \\&= M^*_{\Theta}M^*_{\varphi_{7-i}}M_{\Theta}M_z\\& = M^*_{\psi_{7-i}}M_z.
			\end{aligned}
		\end{equation*} Also, we have
		\begin{equation*}
			\begin{aligned}
				||M_{\psi_i}|| &=
				||M^*_{\Theta}M_{\varphi_i}M_{\Theta}||\\& \leqslant 1.
			\end{aligned}
		\end{equation*}
By [Theorem $4.4$, \cite{apal2}], we concluded that $(M_{\psi_1}, \dots, M_{\psi_6}, M_z)$ is a pure $\Gamma_{E(3; 3; 1, 1, 1)}$-isometry on $H^{2}(\mathcal{E}).$ Converse direction follows easily. This completes the proof.		
	\end{proof}
Let $F_1, \dots, F_6 $ be in $ \mathcal{B}(\mathcal{F})$. Let $(M_{F^*_1 + F_6z}, M_{F^*_2 + F_5z}, M_{F^*_3 + F_4z}, M_{F^*_4 + F_3z}, M_{F^*_5 + F_2z}, M_{F^*_6 + F_1z}, M_z)$ denote a $7$-tuple of bounded operators on $H^{2}(\mathcal F). $ The aforementioned operator tuple is generally non-commutative. We demonstrate that if the 7- tuple is commutative, this leads to an alternative version of the Theorem \ref{BLH} for pure $\Gamma_{E(3; 3; 1, 1, 1)}$-isometry. 
	
	\begin{thm}\label{BLHA}
		Let $\mathcal{M}$ be a non-zero closed subspace of the vector-valued Hardy space $H^2(\mathcal{F})$ and \\$(M_{F^*_1 + F_6z}, \dots, M_{F^*_6 + F_1z}, M_z)$ be a pure $\Gamma_{E(3; 3; 1, 1, 1)}$-isometry on $H^2(\mathcal{F})$. Then $\mathcal{M}$ is a joint invariant subspace of $(M_{F^*_1 + F_6z}, \dots, M_{F^*_6 + F_1z}, M_z)$ if and only if there exists a Hilbert space $\mathcal{E}$, $\tilde{F}_1, \dots, \tilde{F}_6 \in \mathcal{B}(\mathcal{E})$ and an inner multiplier $\Theta \in H^{\infty}(\mathcal{B}(\mathcal{E}, \mathcal{F}))$ such that
\begin{enumerate}
\item  $(M_{\tilde{F}_1 + \tilde{F}^*_6z}, \dots, M_{\tilde{F}_6 + \tilde{F}^*_1z}, M_z)$  is a pure $\Gamma_{E(3; 3; 1, 1, 1)}$-isometry on $H^2( \mathcal{E})$,
\item  $(F^*_i + F_{7-i}z)\Theta(z) =
				\Theta(z)(\tilde{F}_i + \tilde{F}^*_{7-i}z), 1 \leqslant i \leqslant 6$,
				\item $\mathcal{M} = M_{\Theta}H^2(\mathcal{E})$.
		\end{enumerate}		
	\end{thm}
	
	\begin{proof}
		We demonstrate only the forward direction, as the proof of converse is straightforward. Let $\mathcal{M}$ be a non-zero joint invariant subspace of $(M_{F^*_1 + F_6z}, \dots, M_{F^*_6 + F_1z}, M_z)$. According to the \textit{Beurling-Lax-Halmos representation} of $M_z$, we deduce that $\mathcal{M} = M_{\Theta}(H^2(\mathcal{E}))$, where $\Theta \in H^{\infty}(\mathcal{B}(\mathcal{E}, \mathcal{F}))$ denotes the inner function. By using similar argument as in Theorem \ref{BLH}, we conclude that there exist a Hilbert space $\mathcal{E}$ and $\psi_i$  in $H^{\infty}(\mathcal{E})$ such that 
		\begin{equation}\label{THE}
			\begin{aligned}
				M_{F^*_i + F_{7-i}z}M_{\Theta} = M_{\Theta}M_{\psi_i}~{\rm{for}}~1\leq i \leq 6.
							\end{aligned}
		\end{equation}
As $M_{\Theta}$ is an isometry, it implies from \eqref{THE} that \begin{equation}\label{THET}M_{\psi_i} =M^*_{\Theta}M_{F^*_i + F_{7-i}z}M_{\Theta}~{\rm{for}}~1\leq i \leq 6. \end{equation} From \eqref{THE} and \eqref{THET}, it yields that 
\begin{equation}\label{theta}
(F^*_i + F_{7-i}z)\Theta=\Theta \psi_i~{\rm{and}}~\psi_i\psi_j=\psi_j\psi_i, 1\leq i ,j \leq 6.
\end{equation}
It follows from \eqref{THET} that		
		\begin{equation}
			\begin{aligned}\label{pol}
				p(M_{\psi_1}, \dots, M_{\psi_6}, M_z) &=
				M^*_{\Theta}p(M_{F^*_1 + F_6z}, \dots, M_{F^*_6 + F_1z}, M_z)M_{\Theta},
			\end{aligned}
		\end{equation} for any  polynomial $p\in \mathbb C[z_1,\dots,z_7].$		
As $(M_{F^*_1 + F_6z}, \dots, M_{F^*_6 + F_1z}, M_z)$ is a $\Gamma_{E(3; 3; 1, 1, 1)}$-contraction, from \eqref{pol}, we have
		\begin{equation*}
			\begin{aligned}
				||p(M_{\psi_1}, \dots, M_{\psi_6}, M_z)|| &=
				||M^*_{\Theta}p(M_{F^*_1 + F_6z}, \dots, M_{F^*_6 + F_1z}, M_z)M_{\Theta}||\\
				&\leqslant
				||M^*_{\Theta}||\,||p(M_{F^*_1 + F_6z}, \dots, M_{F^*_6 + F_1z}, M_z)||\,||M_{\Theta}||\\
				&\leqslant
				||p||_{\infty, \Gamma_{E(3; 3; 1, 1, 1)}}.
			\end{aligned}
		\end{equation*}
This shows that $(M_{\psi_1}, \dots, M_{\psi_6}, M_z)$ is a $\Gamma_{E(3; 3; 1, 1, 1)}$-contraction. As $M_z$ is a pure isometry, by [Theorem $4.4$, \cite{apal2}], we conclude that $(M_{\psi_1}, \dots, M_{\psi_6}, M_z)$ is a pure $\Gamma_{E(3; 3; 1, 1, 1)}$-isometry on $H^2(\mathcal{E})$ and hence we have $M_{\psi_i} = M^*_{\psi_{7-i}}M_z$ for $1 \leqslant i \leqslant 6$. As $(M_{\psi_1}, \dots, M_{\psi_6}, M_z)$ is a pure $\Gamma_{E(3; 3; 1, 1, 1)}$-isometry on $H^2(\mathcal{E})$, from [Theorem $4.6$,\cite{apal2}], it yields that there exists $\tilde{F}_1, \dots, \tilde{F}_6 \in \mathcal{B}(\mathcal{E})$ such that 
$\psi_i(z) = \tilde{F}_i + \tilde{F}^*_{7-i}z$ and $\psi_{7-i}(z) = \tilde{F}_{7-i} + \tilde{F}^*_iz$ for $1 \leqslant i \leqslant 6.$ Thus, from \eqref{theta}, we deduce that \[(F^*_i + F_{7-i}z)\Theta(z) =
		\Theta(z)(\tilde{F}_i + \tilde{F}^*_{7-i}z) \,\, \text{for} \,\, 1 \leqslant i \leqslant 6.\]
		 This completes the proof.
	\end{proof}
We state a theorem about the unitary equivalence of two pure $\Gamma_{E(3; 2; 1, 2)}$-isometries, and the proof is similar to that of Theorem \ref{Ai}. Therefore, we skip the proof.	
	\begin{thm}
Let $\phi_i(z) = B_i + {B}^*_{3-i}z$ and $\psi_j(z) = C_j + {C}^*_{3-j}z, 1\leq i,j\leq 2$ be in $H^{\infty}(\mathcal{B}(\mathcal{E}))$ for some $B_1,B_2,C_1, C_2$ in $ \mathcal{B}(\mathcal{E})$. Let $ \tilde{\phi}_i(z) = \tilde{B}_i+ \tilde{B}^*_{3-i}z$ and $\tilde{\psi}_j(z) = \tilde{C}_j + \tilde{C}^*_{3-j}z, 1\leq i,j \leq 2$ be in $H^{\infty}(\mathcal{B}(\mathcal{F}))$  for  some $ \tilde{B}_1, \tilde{B}_2, \tilde{C}_1, \tilde{C}_2$ in  $\mathcal{B}(\mathcal{F})$.  Then   $(B_1, B_2, C_1,C_2)$ and $(\tilde{B}_1, \tilde{B}_2, \tilde{C}_1, \tilde{C}_2)$ are unitarily equivalent  if and only if   $(M_{\phi_1}, M_{\phi_2}, M_z, M_{\psi_1}, M_{\psi_2})$ on $H^2(\mathcal{E})$ is unitarily equivalent to  $(M_{\tilde{\phi_1}}, M_{\tilde{\phi_2}}, M_z, M_{\tilde{\psi}_1}, M_{\tilde{\psi}_2})$ on $H^2(\mathcal{F})$.
	\end{thm}
As a consequence of the above theorem, we only state the following corollary, because its proof is similar to Corollary \ref{uii}. Therefore, we skip the proof.
	\begin{cor}
Let $(W_1, W_2, W_3, \widetilde{W}_1, \widetilde{W}_2)$ and $(\hat{W}_1, \hat{W}_2, \hat{W}_3, \hat{\widetilde{W_1}}, \hat{\widetilde{W_2}}) $ be two pure $\Gamma_{E(3; 2; 1, 2)}$-isometries. Then $(W_1, W_2, W_3, \widetilde{W}_1, \widetilde{W}_2)$ and $(\hat{W}_1, \hat{W}_2, \hat{W}_3, \hat{\widetilde{W_1}}, \hat{\widetilde{W_2}})$ are unitarily equivalent if and only if 
\begin{enumerate}
\item  $(W^*_1 - \widetilde{W}_{2}W^*_3, W^*_2 - \widetilde{W}_{1}W^*_3)$
is unitarily equivalent to  $(\hat{W}^*_1 - \hat{\widetilde{W}}_{2}\hat{W}^*_3, \hat{W}^*_2 - \hat{\widetilde{W}}_{1}\hat{W}^*_3)$ and

\item  $(\widetilde{W}^*_{2} - W_1W^*_3, \widetilde{W}^*_{1} - W_2W^*_3)$ is unitarily equivalent to  $(\hat{\widetilde{W}}^*_{2} - \hat{W}_1\hat{W}^*_3, \hat{\widetilde{W}}^*_{1} - \hat{W}_2\hat{W}^*_3)$.
\end{enumerate}
	\end{cor}
We state only the Beurling-Lax-Halmos Representation for a pure $\Gamma_{E(3; 2; 1, 2)}$-isometry, as its proof is analogous to Theorem \ref{BLH}. Therefore, we omit the proof.
	
\begin{thm}\label{BLH1}
		Let $(M_{\phi_1}, M_{\phi_2}, M_z, M_{\psi_1}, M_{\psi_2})$, where  $\phi_i, \psi_j$ in $H^{\infty} (\mathcal B(\mathcal F))$ for $1\leq i, j \leq 2$, be a pure $\Gamma_{E(3; 2; 1, 2)}$-isometry on  $H^2( \mathcal{F})$ and $\mathcal{M}$ be a non-zero closed subspace of the vector-valued Hardy space $H^2( \mathcal{F})$. Then $\mathcal{M}$ is an invariant subspace of a pure $\Gamma_{E(3; 2; 1, 2)}$-isometry $(M_{\phi_1}, M_{\phi_2}, M_z, M_{\psi_1}, M_{\psi_2})$  if and only if there exist $\tilde{\phi}_i$ and $\tilde{\psi}_j$ in $H^{\infty}(\mathcal{E})$ for $1\leq i, j \leq 2$ and an inner function $\Theta \in H^{\infty}(\mathcal{B}(\mathcal{E}, \mathcal{F}))$ such that  
\begin{enumerate}
\item $(M_{\tilde{\phi_1}}, M_{\tilde{\phi_2}}, M_z, M_{\tilde{\psi}_1}, M_{\tilde{\psi}_2})$ is a pure $\Gamma_{E(3; 2; 1, 2)}$-isometry on $H^2( \mathcal{E})$;
\item $M_{\Theta}M_{\tilde{\phi}_j} = M_{\phi_i}M_{\Theta}, \,\, \text{for} \,\,1\leq  i,j\leq 2;$
\item $M_{\Theta}M_{\tilde{\phi}_j} = M_{\psi_i}M_{\Theta}, \,\, \text{for} \,\,1\leq  i,j\leq 2;$
\item $M_{\Theta}M_{\tilde{\psi}_j} = M_{\phi_i}M_{\Theta}, \,\, \text{for} \,\,1\leq  i,j\leq 2;$
\item $M_{\Theta}M_{\tilde{\psi}_j} = M_{\psi_i}M_{\Theta}, \,\, \text{for} \,\,1\leq  i,j\leq 2;$
\item $\mathcal M=M_{\Theta}H^{2}(\mathcal E).$
\end{enumerate}
\end{thm}	
 We state only a different version of the Beurling-Lax-Halmos Representation for a pure $\Gamma_{E(3; 2; 1, 2)}$-isometry. Its proof is similar to Theorem \ref{BLHA}. Let $G_1,G_2,\tilde{G}_1$ and $\tilde{G}_2$ be in $\mathcal B(\mathcal F). $ We consider the following $5$-tuple of bounded operators:
\[(M_{G^*_1 + \tilde{G}_2z}, M_{G^*_2 + \tilde{G}_1z}, M_z, M_{\tilde{G}^*_1 + G_2z}, M_{\tilde{G}^*_2 + G_1z})\]
on the vector-valued Hardy space $H^2(\mathcal{F})$. The above operator tuple is non-commutative in general. When the above tuple is commutative, it gives an alternative version of \textit{Beurling-Lax-Halmos theorem} for pure $\Gamma_{E(3; 2; 1, 2)}$-isometry. 	
	\begin{thm}
		Let $(M_{G^*_1 + \tilde{G}_2z}, M_{G^*_2 + \tilde{G}_1z}, M_z, M_{\tilde{G}^*_1 + G_2z}, M_{\tilde{G}^*_2 + G_1z})$ be  a pure $\Gamma_{E(3; 2; 1, 2)}$-isometry on $H^2(\mathcal{F})$ for some $G_1, G_2, \tilde{G}_1, \tilde{G}_2 \in \mathcal{B}(\mathcal{F})$ and $\mathcal{M}$ be a non-zero closed subspace of the vector-valued Hardy space $H^2(\mathcal{F})$. Then $\mathcal{M}$ is a joint invariant subspace of  $$(M_{G^*_1 + \tilde{G}_2z}, M_{G^*_2 + \tilde{G}_1z}, M_z, M_{\tilde{G}^*_1 + G_2z}, M_{\tilde{G}^*_2 + G_1z})$$ if and only if there exists a Hilbert space $\mathcal{E}$, and $H_1, H_2, \tilde{H}_1, \tilde{H}_2 \in \mathcal{B}(\mathcal{E})$,  an inner function $\Theta \in H^{\infty}(\mathcal{B}(\mathcal{E}, \mathcal{F}))$ such that
\begin{enumerate}
\item $(M_{H_1 + \tilde{H}^*_2z}, M_{H_2 + \tilde{H}^*_1z}, M_z, M_{\tilde{H}_1 + H^*_2z}, M_{\tilde{H}_2 + H^*_1z})$ is a pure $\Gamma_{E(3; 2; 1, 2)}$-isometry on $H^2(\mathcal{E})$;
\item $\mathcal{M} = M_{\Theta}(H^2(\mathcal{E}));$ 
\item  $(G^*_1 + \tilde{G}_2z)\Theta(z) = \Theta(z)(H_1 + \tilde{H}^*_2z),(\tilde{G}^*_2 + G_1z)\Theta(z) = \Theta(z)(\tilde{H}_2 + H^*_1z)$ and 
\item $(G^*_2 + \tilde{G}_1z)\Theta(z) = \Theta(z)(H_2 + \tilde{H}^*_1z), \,\, (\tilde{G}^*_1 + G_2z)\Theta(z) = \Theta(z)(\tilde{H}_1 + H^*_2z).$

\end{enumerate}		
\end{thm}

\section{Conditional Dilation of $\Gamma_{E(3; 3; 1, 1, 1)}$-Contraction and $\Gamma_{E(3; 2; 1, 2)}$-Contraction}

In this section, we aim to discuss the conditional dilations of $\Gamma_{E(3; 3; 1, 1, 1)}$-contraction and $\Gamma_{E(3; 2; 1, 2)}$-contraction. We  start with the definition of the $\Gamma_{E(3; 3; 1, 1, 1)}$-isometric dilation of the $\Gamma_{E(3; 3; 1, 1, 1)}$-contraction and the $\Gamma_{E(3; 2; 1, 2)}$-isometric dilation of the $\Gamma_{E(3; 2; 1, 2)}$-contraction.

\begin{defn}\label{isometric dilation2}
A commuting $7$-tuple of operators $(V_1,\ldots,V_{7})$ acting on a Hilbert space $\mathcal K \supseteq \mathcal H$ is referred to as a $\Gamma_{E(3; 3; 1, 1, 1)}$ -isometric dilation of a $\Gamma_{E(3; 3; 1, 1, 1)}$-contraction $(T_1,\ldots,T_7)$ acting on a Hilbert space $\mathcal H$, if it has following properties:
\begin{itemize}
\item  $(V_1,\ldots,V_{7})$ is $\Gamma_{E(3; 3; 1, 1, 1)}$-isometry;
\item $V_i^{*}|_{\mathcal H}=T_i^{*}$ for all $1\leq i \leq 7.$
\end{itemize}

\end{defn}	
\begin{defn}\label{isometric dilation21}
A commuting $5$-tuple of operators $(W_1,W_2,W_3,\tilde{W}_1,\tilde{W}_2)$ acting on a Hilbert space $\mathcal K \supseteq \mathcal H$ is said to be a $\Gamma_{E(3; 2; 1, 2)}$-isometric dilation of a $\Gamma_{E(3; 2;1, 2)}$-contraction $(S_1,S_2,S_3,\tilde{S}_1,\tilde{S}_2)$ acting on a Hilbert space $\mathcal H,$ if it satisfies the following properties:
\begin{itemize}
\item  $(W_1,W_2,W_3,\tilde{W}_1,\tilde{W}_2)$ is $\Gamma_{E(3; 2; 1, 2)}$-isometry;
\item $W_i^{*}|_{\mathcal H}=S_i^{*}$ for $1\leq i \leq 3$ and $\tilde{W}_j^{*}|_{\mathcal H}=\tilde{S}_j^{*}$ for $1\leq j \leq 2.$
\end{itemize}

\end{defn}
	
The following theorems  are crucial for the construction of dilations of $\Gamma_{E(3; 3; 1, 1, 1)}$-contraction and $\Gamma_{E(3; 2; 1, 2)}$-contraction.
\begin{thm}[Theorem $4.4$,\cite{apal2}]\label{thm-12}
Let $\textbf{V} = (V_1, \dots, V_7)$ be a $7$-tuple of commuting bounded operators on a Hilbert space $\mathcal{H}$. Then the following are equivalent:
\begin{enumerate}
\item $\textbf{V}$ is a $\Gamma_{E(3; 3; 1, 1, 1)}$-isometry.

\item $\textbf{V}$ is a $\Gamma_{E(3; 3; 1, 1, 1)}$-contraction and $V_7$ is an isometry.

\item $V_i$ is a contraction, $V_i = V^*_{7-i} V_7$ for $1\leq i\leq 6$ and $V_7$ is isometry.
\item $(V_i,V_{7-i},V_7)$ is a $\Gamma_{E(2; 2; 1, 1)}$-isometry for $1\leq i\leq 3$.
\item  $V_i$ is a contraction and $\rho_{G_{E(2; 2; 1,1)}} (V_i,zV_{7-i},zV_7)=0, 1\leq i \leq 6,$ for all $z\in \mathbb T.$

\item $V_7$ is an isometry, $r(V_i)\leqslant 1$ for $1\leq i \leq 6$ and $V_1 = V^*_6V_7, V_2 = V^*_5V_7, V_3 = V^*_4V_7$.
\end{enumerate}
\end{thm}
	
\begin{thm}[Theorem $4.5$,\cite{apal2}]\label{thm-13}
		Let $\textbf{W} = (W_1, W_2, W_3, \tilde{W}_1, \tilde{W}_2)$ be a $5$-tuple of commuting bounded operators on a Hilbert space $\mathcal{H}$. Then the following are equivalent:
		\begin{enumerate}
			\item $\textbf{W}$ is a $\Gamma_{E(3; 2; 1, 2)}$-isometry.
			
			\item $\textbf{W}$ is a $\Gamma_{E(3; 2; 1, 2)}$-contraction and $W_3$ is an isometry.
			
			\item $(W_1, \tilde{W}_2, W_3), ( \frac{\tilde{W}_1}{2},\frac{W_2}{2},  W_3)$ and $(\frac{W_2}{2}, \frac{\tilde{W}_1}{2}, W_3)$ are $\Gamma_{E(2; 2; 1,1)}$-isometries.
			
			\item $W_3$ is an isometry, $W_1, \frac{W_2}{2}, \frac{\tilde{W}_1}{2}, \tilde{W}_2$ are contractions, and $W_1 = \tilde{W}^*_2W_3, W_2 = \tilde{W}^*_1W_3$.
			
			\item $W_3$ is an isometry, $r(W_1)\leq 1, r(\frac{W_2}{2})\leq 1, r(\frac{\tilde{W}_1}{2})\leq 1, r(\tilde{W}_2) \leqslant 1$, $W_1 = \tilde{W}^*_2W_3$ and $W_2 = \tilde{W}^*_1W_3$.
\end{enumerate}
\end{thm}
A $\Gamma_{E(3; 3; 1, 1, 1)}$-isometry (respectively, $\Gamma_{E(3; 2; 1, 2)}$-isometry), by definition, is restriction of a $\Gamma_{E(3; 3; 1, 1, 1)}$-unitary (respectively, $\Gamma_{E(3; 2; 1, 2)}$-unitary). Hence, if a $\Gamma_{E(3; 3; 1, 1, 1)}$-contraction (respectively, $\Gamma_{E(3; 2; 1, 2)}$-contraction) possesses a $\Gamma_{E(3; 3; 1, 1, 1)}$-isometric (respectively, $\Gamma_{E(3; 2; 1, 2)}$-isometric) dilation, it follows that, by definition, it also has a $\Gamma_{E(3; 3; 1, 1, 1)}$-unitary (respectively, $\Gamma_{E(3; 2; 1, 2)}$-unitary) dilation. We will construct a dilation under the assumption that $\mathbf{T} = (T_1, \dots, T_7)$ is a $\Gamma_{E(3; 3; 1, 1, 1)}$-contraction, with its fundamental operators $F_i$ and $F_{7-i}$  satisfying the following conditions:
\begin{equation}
[F_i,F_j]=0~~{\rm{and}}~~[F_{7-i}^*,F_j]=[F_{7-j}^*,F_i], 1\leq i,j \leq 6.
\end{equation}

We will construct the minimal $\Gamma_{E(3; 3; 1, 1, 1)}$-isometric dilation of a $\Gamma_{E(3; 3; 1, 1, 1)}$-contraction based on Schaffer's construction.

\begin{thm}[Conditional Dilation of $\Gamma_{E(3; 3; 1, 1, 1)}$-Contraction]\label{conddilation}
		Let $\mathbf{T} = (T_1, \dots, T_7) $ be a $\Gamma_{E(3; 3; 1, 1, 1)} $-contraction defined on a Hilbert space $\mathcal H$ with the fundamental operator $F_i$ and $F_{7-i},$ for $1\le i \leq 6,$ which satisfy the following conditions: \begin{enumerate}
			\item[$(i)$] $[F_i, F_j] = 0, 1\leq i, j \leq 6$;
			
			\item[$(ii)$] $[F^*_{7-i}, F_j] = [F^*_{7-j}, F_i], 1\leq i, j \leq  6$.
		\end{enumerate} Let 
		\[\mathcal{K} = \mathcal{H} \oplus \mathcal{D}_{T_7} \oplus \mathcal{D}_{T_7} \oplus \dots = \mathcal{H} \oplus l^2(\mathcal{D}_{T_7}).\] Let  $\mathbf{V}=(V_1, \dots, V_7)$ be a $7$-tuple of operators defined on $\mathcal K$
by
		\begin{equation}\label{v7}
			\begin{aligned}
				&V_i =
				\begin{bmatrix}
					T_i & 0 & 0 & \dots\\
					F^*_{7-i}D_{T_7} & F_i & 0 & \dots\\
					0 & F^*_{7-i} & F_i & \dots\\
					0 & 0 & F^*_{7-i} & \dots\\
					\vdots & \vdots & \vdots & \ddots
				\end{bmatrix}~{\rm{and}}~~
				V_7 =
				\begin{bmatrix}
					T_7 & 0 & 0 & \dots\\
					D_{T_7} & 0 & 0 & \dots\\
					0 & I & 0 & \dots\\
					0 & 0 & I & \dots\\
					\vdots & \vdots & \vdots & \ddots
				\end{bmatrix}.
			\end{aligned}
		\end{equation}
Then we have the following:
\begin{enumerate}

\item $\mathbf{V}$ is a minimal $\Gamma_{E(3; 3; 1, 1, 1)} $-isometric dilation of $\mathbf{T}$.

\item If there exists a $\Gamma_{E(3; 3; 1, 1, 1)} $-isometric dilation $\textbf{W}= (W_1, \dots, W_7)$ of $\textbf{T}$ such that $W_7$ is a minimal isometric dilation of $T_7$, then $\textbf{W}$ is unitarily equivalent to $\textbf{V}$. Furthermore, the above conditions $(i)$ and $(ii)$ are also valid.
\end{enumerate}
	\end{thm}
	
	\begin{proof}
It is easy to see that $V_7$ on $\mathcal{K}$ is the minimal isometric dilation of $T_7$ . The above construction is due to Sch\"{a}ffer. The Sch\"{a}ffer construction was for the unitary dilation, but we only present the isometry part in \eqref{v7}. It is clear from \eqref{v7} that $V_i^{*}|_{\mathcal H}=T_i^{*}$ for all $1\leq i \leq 7.$ To show that $\mathbf{V}$ a $\Gamma_{E(3; 3; 1, 1, 1)}$-isometry, we need to verify the following operator identities:		\begin{enumerate}
			\item $V_iV_7=V_7V_i, 1 \leq i \leq 6;$
			\item $V_iV_j = V_jV_i$ for $1 \leq i, j \leq 6$;
			
			\item $V_i = V^*_{7-i}V_7, 1 \leq i \leq 6$;
			
			\item $r(V_i) \leqslant 1, 1\leq i \leq 6.$ 
		\end{enumerate}
\noindent{\bf{Step-1:}} To demonstrate that $V_i$ commutes with $V_7$,  $1 \leq i \leq 6$, it is necessary to first compute $V_iV_7$ and $V_7V_i$. Observe that		\begin{equation*}
			\begin{aligned}
				&V_iV_7 =
				\begin{bmatrix}
					T_iT_7 & 0 & 0 & 0 & \dots\\
					F^*_{7-i}D_{T_7}T_7 + F_iD_{T_7} & 0 & 0 & 0 & \dots\\
					F^*_{7-i}D_{T_7} & F_i & 0 & 0 & \dots\\
					0 & F^*_{7-i} & F_i & 0 & \dots\\
					0 & 0 & F^*_{7-i} & F_i & \dots\\
					\vdots & \vdots & \vdots & \vdots & \ddots
				\end{bmatrix}, ~
				V_7V_i =
				\begin{bmatrix}
					T_7T_i & 0 & 0 & 0 & \dots\\
					D_{T_7}T_i & 0 & 0 & 0 & \dots\\
					F^*_{7-i}D_{T_7} & F_i & 0 & 0 & \dots\\
					0 & F^*_{7-i} & F_i & 0 & \dots\\
					0 & 0 & F^*_{7-i} & F_i & \dots\\
					\vdots & \vdots & \vdots & \vdots & \ddots
				\end{bmatrix}.
			\end{aligned}
		\end{equation*}
It follows from Lemma \ref{FiFj} that  $F^*_{7-i}D_{T_7}T_7 + F_iD_{T_7}=D_{T_7}T_i$ which implies that $V_iV_7=V_7V_i$ for $1\leq i \leq 6.$

\noindent{\bf{Step-2:}} We now show that  $V_iV_j = V_jV_i$ for $1 \leq i, j \leq 6$. Note that \begin{equation*}
			\begin{aligned}
				V_iV_j &=
				\begin{bmatrix}
					T_iT_j & 0 & 0 & 0 & \dots\\
					F^*_{7-i}D_{T_7}T_j + F_iF^*_{7-j}D_{T_7} & F_iF_j & 0 & 0 & \dots\\
					F^*_{7-i}F^*_{7-j}D_{T_7} & F^*_{7-i}F_j + F_iF^*_{7-j} & F_iF_j & 0 & \dots\\
					0 & F^*_{7-i}F^*_{7-j} & F^*_{7-i}F_j + F_iF^*_{7-j} & F_iF_j & \dots\\
					\vdots & \vdots & \vdots & \vdots & \ddots
				\end{bmatrix},
			\end{aligned}
		\end{equation*}
		and
		\begin{equation*}
			\begin{aligned}
				V_jV_i &=
				\begin{bmatrix}
					T_jT_i & 0 & 0 & 0 & \dots\\
					F^*_{7-j}D_{T_7}T_i + F_jF^*_{7-i}D_{T_7} & F_jF_i & 0 & 0 & \dots\\
					F^*_{7-j}F^*_{7-i}D_{T_7} & F^*_{7-j}F_i + F_jF^*_{7-i} & F_jF_i & 0 & \dots\\
					0 & F^*_{7-j}F^*_{7-i} & F^*_{7-j}F_i + F_jF^*_{7-i} & F_jF_i & \dots\\
					\vdots & \vdots & \vdots & \vdots & \ddots
				\end{bmatrix}.
			\end{aligned}
		\end{equation*}
To verify that $V_iV_j = V_jV_i$, it is necessary to demonstrate the following operator identities:
		\begin{equation}\label{f11}
		\begin{aligned}
		 F^*_{7-i}D_{T_7}T_j + F_iF^*_{7-j}D_{T_7} = F^*_{7-j}D_{T_7}T_i + F_jF^*_{7-i}D_{T_7},\; 1 \leq i, j \leq 6.
		\end{aligned}
		\end{equation}
It follows from \eqref{fun} and \eqref{f11} that
		\begin{equation}\label{fun12}
			\begin{aligned}
	 F^*_{7-i}D_{T_7}T_j- F^*_{7-j}D_{T_7}T_i	&=F^*_{7-i}(F_jD_{T_7} + F^*_{7-j}D_{T_7}T_7) - F^*_{7-j}(F_iD_{T_7} + F^*_{7-i}D_{T_7}T_7),\; 1 \leq i,j\leq 6.  \end{aligned}
		\end{equation}
As $[F_{7-i},F_{7-j}]=0$ and $[F^*_{7-i}, F_j] = [F^*_{7-j}, F_i], 1\leq i, j \leq  6$, from \eqref{fun12}, it yields that 
\begin{equation}\label{fun123}
\begin{aligned}
	 F^*_{7-i}D_{T_7}T_j- F^*_{7-j}D_{T_7}T_i	&=(F^*_{7-i}F_j- F^*_{7-j}F_i)D_{T_7}  \\&=(F_jF^*_{7-i} - F_iF^*_{7-j})D_{T_7}.
			\end{aligned}
		\end{equation}		
Thus, we conclude that $V_iV_j = V_jV_i$ for $1 \leq i, j \leq 6$.

\noindent{\bf{Step-3:}}  We demonstrate that $V_i = V^*_{7-i}V_7$ for $1\leq i \leq 6$. We note that
		\begin{equation}\label{viv7i}
			\begin{aligned}
				V^*_{7-i}V_7 &=
				\begin{bmatrix}
					T^*_{7-i}T_7 + D_{T_7}F_iD_{T_7} & 0 & 0 & 0 & \dots\\
					F^*_{7-i}D_{T_7} & F_i & 0 & 0 & \dots\\
					0 & F^*_{7-i} & F_i & 0 & \dots\\
					0 & 0 & F^*_{7-i} & F_i & \dots\\
					\vdots & \vdots & \vdots & \vdots & \ddots
				\end{bmatrix}.
				\end{aligned}
		\end{equation}
By Theorem \ref{fundam} and \eqref{viv7i}, we deduce  that 
\begin{equation}\label{viv7i1}
			\begin{aligned}
				V^*_{7-i}V_7 &=\begin{bmatrix}
					T_i & 0 & 0 & 0 & \dots\\
					F^*_{7-i}D_{T_7} & F_i & 0 & 0 & \dots\\
					0 & F^*_{7-i} & F_i & 0 & \dots\\
					0 & 0 & F^*_{7-i} & F_i & \dots\\
					\vdots & \vdots & \vdots & \vdots & \ddots
				\end{bmatrix} \\&= V_i.
			\end{aligned}
		\end{equation}

\noindent{\bf{Step-4:}} To complete the proof, it is necessary to demonstrate that the spectral radius of $V_i$, $1 \leq i \leq 6$, is less than or equal to $1$. From the definition of $V_i$, it  is expressed in following form 
		\begin{equation*}
			\begin{aligned}
				V_i &=
				\begin{bmatrix}
					T_i & 0\\
					C_i & D_i
				\end{bmatrix},
			\end{aligned}
		\end{equation*}
		where
		\begin{equation*}
			\begin{aligned}
				&C_i =
				\begin{bmatrix}
					F^*_{7-i}D_{T_7}\\
					0\\
					0\\
					\vdots
				\end{bmatrix}
				~\text{and}~
				D_i =
				\begin{bmatrix}
					F_i & 0 & 0 & \dots\\
					F^*_{7-i} & F_i & 0 & \dots\\
					0 & F^*_{7-i} & F_i & \dots\\
					\vdots & \vdots & \vdots & \ddots
				\end{bmatrix}, \;\;1\leq i \leq 6.
			\end{aligned}
		\end{equation*}
It yields from  [Lemma $1$, \cite{Du}] that  $\sigma(V_i) \subseteq \sigma(T_i) \cup \sigma(D_i), 1\leq i \leq 6$.  We will establish that $r(T_i)$ and $r(D_i)$, $1 \leq i \leq 6,$ do not exceed $1$.  Consequently, we conclude that the numerical radius of \(D_i\), for \(1 \leq i \leq 6\), does not exceed $1$, as the spectral radius is not bigger than the numerical radius.  For $1\leq i \leq 6,$ let $\varphi_i : \mathbb{D} \to \mathcal{B}(\mathcal{D}_{T_7}) $ be the map defined by $$\varphi_i(z) = F_i + zF^*_{7-i}. $$ It is easy to verify that the map $\varphi_i$ is holomorphic on $\mathbb{D}$ and continuous on the boundary $\mathbb{T}$. The operator $D_i$  corresponds to multiplication by the functions $\varphi_i$, $1\leq i \leq 6$, under the Hilbert space isomorphism that maps $\mathcal{D}_{T_7} \oplus \mathcal{D}_{T_7} \oplus \dots$ to $H^2(\mathbb{D}) \otimes \mathcal{D}_{T_7}.$ As $\mathbf{T} = (T_1, \dots, T_7) $ is a $\Gamma_{E(3; 3; 1, 1, 1)} $-contraction, it follows from Theorem \ref{fundam} that $\omega(F_i + zF_{7-i}) \leqslant 1$ for $1\leq i \leq 6$ and for all $z \in \mathbb{T}$  and hence we have $\omega(z_1F_i + z_2F_{7-i}) \leqslant 1$ for $1\leq i \leq 6$ and for all $z_1,z_2 \in \mathbb{T}.$  Let $X_i = F_i + zF^*_{7-i}$ and $Y_i = e^{i\theta}X_i + e^{-i\theta}X^*_i$ for $1\leq i \leq 6$ and $z\in \mathbb{T}$. Clearly, $Y_i$ is a self-adjoint operator.  Observe that 
\begin{equation}\label{num}
		\begin{aligned}		
\omega(e^{i\theta}X_i)&=\sup_{||x|| = 1} |\langle e^{i\theta}X_i x, x \rangle|\\&=\sup_{||x|| = 1} |\langle  x, e^{-i\theta}X_i^*x \rangle|\\&=\omega(e^{-i\theta}X^*_i).
\end{aligned}
		\end{equation}		
Observe that 
\begin{equation}\label{num1}
\begin{aligned}
			\omega(Y_i) &= \sup_{||x|| = 1} |\langle (e^{i\theta}X_i + e^{-i\theta}X_i^*)x, x \rangle| \\
				&\leqslant w(F_i + e^{-2i\theta}\overline{z}F_{7-i}) + w((F_i + e^{-2i\theta}\overline{z}F_{7-i})^*) \leqslant 2.
\end{aligned}
\end{equation}
It yields from \eqref{num},\eqref{num1} and Lemma \ref{numer} that 
\begin{equation*}
\begin{aligned}
e^{i\theta}X_i + e^{-i\theta}X_i^* \leq 2,\;\; 1\leq i \leq 6
\end{aligned}
\end{equation*}
which is equivalent to \begin{equation*}
\begin{aligned}
e^{i\theta}(F_i+e^{-2i\theta}\overline{z}F_{7-i}^*) + e^{-i\theta} (F_i+e^{-2i\theta}\overline{z}F_{7-i}^*)^*\leq 2,
\;\; 1\leq i \leq 6.
\end{aligned}
\end{equation*}
Again by Lemma \ref{numer}, we deduce that $\omega(F_i+zF_{7-i}^*)\leq 1$ for $1\leq i \leq 6$ and for all $z$ in the unit circle. This demonstrates that  
\begin{equation*}\begin{aligned}
\omega (D_i)&=\omega (M_{\varphi_i})\\&\leq 
\sup \{\omega(\varphi_i(z)):z \in \mathbb T \}\\&\leq 1.
\end{aligned}
\end{equation*}
This implies that $r(V_i) \leqslant 1$ for $1\leq i \leq 6$.  Thus, analysing all the aforementioned facts, we conclude that $\textbf{V}$ is a minimal $\Gamma_{E(3; 3; 1, 1, 1)}$-isometric dilation of $\textbf{T}$.
		
\noindent{\bf{Proof of unitary equivalence:}}  Suppose that  $\textbf{T}$ has an another $\Gamma_{E(3; 3; 1, 1, 1)}$-isometric dilation $\textbf{W} = (W_1, \dots, W_7)$ on $\mathcal{K}$ such that $W_7$ is a minimal isometric dilation of $T_7$. Since $W_7$ is an isometry,  $W_7$ has the following form\begin{equation*}
			\begin{aligned}
				W_7 &=
				\begin{bmatrix}
					T_7 & 0\\
					C_7 & E_7
				\end{bmatrix}
			\end{aligned}
		\end{equation*}
with respect to the decomposition $\mathcal{H} \oplus \ell^2(\mathcal{D}_{T_7})$ of $\mathcal{K},$ where
\begin{equation*}
			\begin{aligned}
				C_7 &=
				\begin{bmatrix}
					D_{T_7}\\
					0\\
					0\\
					\vdots
				\end{bmatrix} : \mathcal{H} \to \mathcal{D}_{T_7}\oplus \mathcal{D}_{T_7}\oplus\mathcal{D}_{T_7}\oplus \dots\end{aligned}
		\end{equation*}
	$${\rm{and}}$$
				\begin{equation*}
			\begin{aligned}E_7 =
				\begin{bmatrix}
					0 & 0 & 0 & \dots\\
					I & 0 & 0 & \dots\\
					0 & I & 0 & \dots\\
					\vdots & \vdots & \vdots & \ddots
				\end{bmatrix} :  \mathcal{D}_{T_7}\oplus \mathcal{D}_{T_7}\oplus\mathcal{D}_{T_7}\oplus  \dots \to  \mathcal{D}_{T_7}\oplus \mathcal{D}_{T_7}\oplus\mathcal{D}_{T_7}\oplus \dots.
			\end{aligned}
		\end{equation*}
Since $W_iW_7=W_7W_i$ for $1 \leq i \leq 6,$ it follows by a straightforward calculation that the operators $W_i'$s  have the following operator matrix forms: 
\begin{equation}\label{maa}
\begin{aligned}
W_i & =
\begin{bmatrix}
T_i & 0\\
C_i & E_i
\end{bmatrix},\;\;\; 1\leq i\leq 6
\end{aligned}
\end{equation}
for some $C_i$ and $E_i$, $1\leq i \leq 6,$ with respect to the decomposition of $\mathcal K$ as $\mathcal H\oplus \ell^{2}(\mathcal {D}_{T_7}).$ 

There is a natural identification between Hardy space $H^2(\mathcal{D}_{T_7})$ of functions that take values in $\mathcal{D}_{T_7}$ on the unit disk and $\ell^2(\mathcal{D}_{T_7})= \mathcal{D}_{T_7}\oplus \mathcal{D}_{T_7}\oplus\mathcal{D}_{T_7}\oplus \dots. $ 
The  multiplication operator $M^{\mathcal{D}_{T_7}}_z$ on $H^2(\mathcal{D}_{T_7})$ is the same as the operator $E_7$ under the above Hilbert space isomorphism. As $W_iW_7=W_7W_i$, it implies that $E_iM^{\mathcal{D}_{T_7}}_z=M^{\mathcal{D}_{T_7}}_zE_i$ for $1\leq i \leq 6$ and hence  $E_i$'s are of the form $M^{\mathcal{D}_{T_7}}_{\varphi_i}$ for some $\varphi_i \in H^{\infty}(\mathcal{D}_{T_7})$.   

As $\textbf{W}$ is a $\Gamma_{E(3; 3; 1, 1, 1)}$-isometry, according to the characterization given in [Theorem $4.4$, \cite{apal2}], it follows that $W_i = W_{7-i}^* W_7$ for $1 \leq i \leq 6$. Consequently, we obtain 
\begin{equation}\label{W1W7}
\begin{aligned}
W_i & =
\begin{bmatrix}
T_i & 0\\
C_i & M^{\mathcal{D}_{T_7}}_{\varphi_i}
\end{bmatrix}\\&=W_{7-i}^* W_7\\&=\begin{bmatrix}
T_{7-i}& 0\\
C_{7-i} & M^{\mathcal{D}_{T_7}}_{\varphi_{7-i}}
\end{bmatrix}^*\begin{bmatrix}
					T_7 & 0\\
					C_7 & M^{\mathcal{D}_{T_7}}_z
				\end{bmatrix}\\&=\begin{bmatrix}
					T_{7-i}^*T_7+C_{7-i}^*C_7 & C_{7-i}^*M^{\mathcal{D}_{T_7}}_z\\
					\Big( M^{\mathcal{D}_{T_7}}_{\varphi_{7-i}}\Big)^*C_7 & \Big( M^{\mathcal{D}_{T_7}}_{\varphi_{7-i}}\Big)^*M^{\mathcal{D}_{T_7}}_z
				\end{bmatrix},\;\;\; 1 \leq i \leq 6.
\end{aligned}
\end{equation}
It yields from \eqref{W1W7} that
\begin{equation}\label{Ti}
\begin{aligned}
T_i - T^*_{7-i}T_7 = C^*_{7-i}C_7, \,\,\,\,
					C_i = (M^{\mathcal{D}_{T_7}}_{\varphi_{7-i}})^*C_7, \,\,\,\,
					M^{\mathcal{D}_{T_7}}_{\varphi_i} = (M^{\mathcal{D}_{T_7}}_{\varphi_{7-i}})^*M^{\mathcal{D}_{T_7}}_z,\;\;\; 1\leq i\leq 6.
				\end{aligned}
		\end{equation}
Let us consider the following functions represented by the power series expansions: $\Phi(z)=\sum_{i=0}^{\infty}\tilde{C}_iz^i$ and $\Psi(z)=	\sum_{i=0}^{\infty}\tilde{D}_iz^i$. The equation $M^{\mathcal{D}_{T_7}}_{\Phi}=(M^{\mathcal{D}_{T_7}}_{\Psi})^*M^{\mathcal{D}_{T_7}}_z $ suggests that 
\begin{equation}\label{phipsi}
\sum_{i=0}^{\infty}\tilde{C}_iz^i=z\sum_{j=0}^{\infty}\tilde{D}^*_j\bar{z}^j=\tilde{D}_0^*z+\tilde{D}^*_1+\sum_{j=2}^{\infty}\tilde{D}^*_{j}\bar{z}^{j-1}~\text{for~all}~z\in\mathbb T.
\end{equation}
 From \eqref{phipsi}, we deduce that $\tilde{C}_1=\tilde{D}_0^*, \tilde{C}_0=\tilde{D}^*_1$ and $\tilde{C}_i=\tilde{D}_i=0$ for $2\leq i <\infty$. This follows from \eqref{Ti} that $\varphi_i(z)=A_i+A_{7-i}^*z$ for some $A_i, A_{7-i} \in \mathcal B(\mathcal{D}_{T_7})$ for $1 \leq i \leq 6.$ Thus,  we have 
		\begin{equation}\label{maa1}
			\begin{aligned}
				M^{\mathcal{D}_{T_7}}_{\varphi_i}&=E_i &=
				\begin{bmatrix}
					A_i & 0 & 0 & \dots\\
					A^*_{7-i} & A_i & 0 & \dots\\
					0 & A^*_{7-i} & A_i & \dots\\
					\vdots & \vdots & \vdots & \ddots
				\end{bmatrix},\;\;\; 1\leq i \leq 6,
			\end{aligned}
		\end{equation} on $\mathcal{D}_{T_7}\oplus \mathcal{D}_{T_7}\oplus\mathcal{D}_{T_7}\oplus \dots.$
It follows from \eqref{Ti} that 
		\begin{equation}\label{maa2}
			\begin{aligned}
				C_i &= (M^{\mathcal{D}_{T_7}}_{\varphi_{7-i}})^*C_7\\&=
				\begin{bmatrix}
					A^*_{7-i}D_{T_7}\\
					0\\
					0\\
					\vdots
				\end{bmatrix},\;\;\; 1\leq i\leq 6.
			\end{aligned}
		\end{equation} It yields from \eqref{maa1} and \eqref{maa2} that  
		\begin{equation*}
			\begin{aligned}
				W_i &=
				\begin{bmatrix}
					T_i & 0 & 0 & 0  &\dots\\
					A^*_{7-i}D_{T_7} & A_i & 0 & 0 & \dots\\
					0 & A^*_{7-i} & A_i & 0 & \dots\\
					0 & 0 & A^*_{7-i} & A_i & \dots\\
					\vdots & \vdots & \vdots & \vdots & \ddots
				\end{bmatrix},\;\;\; 1\leq i \leq 6,
			\end{aligned}
		\end{equation*}
with respect the decomposition  $\mathcal{H} \oplus \ell^2(\mathcal{D}_{T_7})$ of $\mathcal{K}.$ As $W_iW_7=W_7W_i,1\leq i \leq 6,$  we get
		\small{\begin{equation*}
			\begin{aligned}
				\begin{bmatrix}
					T_i & 0 & 0 & \dots\\
					A^*_{7-i}D_{T_7} & A_i & 0 & \dots\\
					0 & A^*_{7-i} & A_i & \dots\\
					0 & 0 & A^*_{7-i} & \dots\\
					\vdots & \vdots & \vdots & \ddots
				\end{bmatrix}
				\begin{bmatrix}
					T_7 & 0 & 0 & \dots\\
					D_{T_7} & 0 & 0 & \dots\\
					0 & I & 0 & \dots\\
					0 & 0 & I & \dots\\
					\vdots & \vdots & \vdots & \ddots
				\end{bmatrix}
				&=
				\begin{bmatrix}
					T_7 & 0 & 0 & \dots\\
					D_{T_7} & 0 & 0 & \dots\\
					0 & I & 0 & \dots\\
					0 & 0 & I & \dots\\
					\vdots & \vdots & \vdots & \ddots
				\end{bmatrix}
				\begin{bmatrix}
					T_i & 0 & 0 & \dots\\
					A^*_{7-i}D_{T_7} & A_i & 0 & \dots\\
					0 & A^*_{7-i} & A_i & \dots\\
					0 & 0 & A^*_{7-i} & \dots\\
					\vdots & \vdots & \vdots & \ddots
				\end{bmatrix}.
			\end{aligned}
		\end{equation*}}
By equating the $(2,1)$ entries on both sides of the aforementioned equation, we obtain 
\begin{equation}\label{maa3}
			\begin{aligned}
				D_{T_7}T_i = A_iD_{T_7} + A^*_{7-i}D_{T_7}T_7,\;\;\;\;1\leq i \leq 6.
			\end{aligned}
			\end{equation}
By Lemma \ref{FiFj}, we deduce that $F_i=A_i,1\leq i \leq 6,$ are the fundamental operators of the $\Gamma_{E(3; 3; 1, 1, 1)}$-contraction $\mathbf{T} = (T_1, \dots, T_7).$  Since $W_i$ 's commute with each other, we have   $M^{\mathcal{D}_{T_7}}_{\varphi_i}M^{\mathcal{D}_{T_7}}_{\varphi_j}=M^{\mathcal{D}_{T_7}}_{\varphi_j}M^{\mathcal{D}_{T_7}}_{\varphi_i}$  which implies that 
\begin{equation}\label{A_i}
(F_i+F_{7-i}^*z)(F_j+F_{7-j}^*z)=(F_j+F_{7-j}^*z)(F_i+F_{7-i}^*z)~\text{for}~1 \leq i,j \leq 6\;\;\rm{and}\;\;z\in \mathbb{T}.
\end{equation}
It follows from \eqref{A_i} that 	$[F_i, F_j] = 0$ and $ [F_i, F^*_{7-j}] = [F_j, F^*_{7-i}]$ for $1\leq i, j \leq  6$. Thus, we conclude that $\textbf{V}$ and $\textbf{W}$ are unitarily equivalent. This completes the proof.
\end{proof}
We will construct a dilation assuming that $\textbf{S} = (S_1, S_2, S_3, \tilde{S}_1, \tilde{S}_2)$ is a $\Gamma_{E(3; 2; 1, 2)}$-contraction, with its fundamental operators $G_1,2G_2,2\tilde{G}_1$ and $\tilde{G}_2$ which satisfy the following conditions:

		\begin{enumerate}
			\item[$(i)$] $[G_1,\tilde{G}_i]=0$ for $1 \leq i \leq 2,$ $[G_2,\tilde{G}_j]=0$ for $1 \leq j \leq 2,$ and $[G_1, G_2] = [\tilde{G}_1, \tilde{G}_2]  = 0$;
			
			\item[$(ii)$] $[G_1, G^*_1] = [\tilde{G}_2, \tilde{G}^*_2], [G_2, G^*_2] = [\tilde{G}_1, \tilde{G}^*_1], [G_1, \tilde{G}^*_1] = [G_2, \tilde{G}^*_2], [\tilde{G}_1, G^*_1] = [\tilde{G}_2, G^*_2],$\\$ [G_1, G^*_2] = [\tilde{G}_1, \tilde{G}^*_2], [G^*_1, G_2] = [\tilde{G}^*_1, \tilde{G}_2]$.
		\end{enumerate}		
The following construction of dilation  is based on the Sch\"{a}ffer.
\begin{thm}[Conditional Dilation of $\Gamma_{E(3; 2; 1, 2)}$-Contraction]\label{condilation1}
		Let $\textbf{S} = (S_1, S_2, S_3, \tilde{S}_1, \tilde{S}_2)$ be a $\Gamma_{E(3; 2; 1, 2)}$-contraction define on a Hilbert space $\mathcal H$ with the fundamental operators $G_1,2G_2,2\tilde{G}_1$ and $\tilde{G}_2$ which satisfy the following conditions:

		\begin{enumerate}
			\item[$(i)$] $[G_1,\tilde{G}_i]=0$ for $1 \leq i \leq 2,$ $[G_2,\tilde{G}_j]=0$ for $1 \leq j \leq 2,$ and $[G_1, G_2] = [\tilde{G}_1, \tilde{G}_2]  = 0$;
			
			\item[$(ii)$] $[G_1, G^*_1] = [\tilde{G}_2, \tilde{G}^*_2], [G_2, G^*_2] = [\tilde{G}_1, \tilde{G}^*_1], [G_1, \tilde{G}^*_1] = [G_2, \tilde{G}^*_2], [\tilde{G}_1, G^*_1] = [\tilde{G}_2, G^*_2],$\\$ [G_1, G^*_2] = [\tilde{G}_1, \tilde{G}^*_2], [G^*_1, G_2] = [\tilde{G}^*_1, \tilde{G}_2]$.
		\end{enumerate}		
Let \begin{equation*}
			\begin{aligned}
				\mathcal{\tilde{K}} &= \mathcal{H} \oplus \mathcal{D}_{S_3} \oplus \mathcal{D}_{S_3} \oplus \dots = \mathcal{H} \oplus l^2(\mathcal{D}_{S_3}).
			\end{aligned}
		\end{equation*}
Suppose that $\textbf{W}=(W_1,W_2,W_3,\tilde{W}_1, \tilde{W}_2)$ is a $5$-tuple of bounded operators is defined  on $\tilde{K}$ by
		\begin{equation}\label{w3}
			\begin{aligned}
				&W_1 =
				\begin{bmatrix}
					S_1 & 0 & 0 & \dots\\
					\tilde{G}^*_2D_{S_3} & G_1 & 0 & \dots\\
					0 & \tilde{G}^*_2 & G_1 & \dots\\
					0 & 0 & \tilde{G}^*_2 & \dots\\
					\vdots & \vdots & \vdots & \ddots
				\end{bmatrix}, \,\,
				W_2 =
				\begin{bmatrix}
					S_2 & 0 & 0 & \dots\\
					2\tilde{G}^*_1D_{S_3} & 2G_2 & 0 & \dots\\
					0 & 2\tilde{G}^*_1 & 2G_2 & \dots\\
					0 & 0 & 2\tilde{G}^*_1 & \dots\\
					\vdots & \vdots & \vdots & \ddots
				\end{bmatrix},
				W_3 =
				\begin{bmatrix}
					S_3 & 0 & 0 & \dots\\
					D_{S_3} & 0 & 0 & \dots\\
					0 & I & 0 & \dots\\
					0 & 0 & I & \dots\\
					\vdots & \vdots & \vdots & \ddots
				\end{bmatrix},\\
				&\hspace{2cm}
				\tilde{W}_1 =
				\begin{bmatrix}
					\tilde{S}_1 & 0 & 0 & \dots\\
					2G^*_2D_{S_3} & 2\tilde{G}_1 & 0 & \dots\\
					0 & 2G^*_2 &2 \tilde{G}_1 & \dots\\
					0 & 0 & 2G^*_2 & \dots\\
					\vdots & \vdots & \vdots & \ddots
				\end{bmatrix} \,\,{\rm{and}}~
				\tilde{W}_2 =
				\begin{bmatrix}
					\tilde{S}_2 & 0 & 0 & \dots\\
					G^*_1D_{S_3} & \tilde{G}_2 & 0 & \dots\\
					0 & G^*_1 & \tilde{G}_2 & \dots\\
					0 & 0 & G^*_1 & \dots\\
					\vdots & \vdots & \vdots & \ddots
				\end{bmatrix}.
			\end{aligned}
		\end{equation}
Then we have the following: 
\begin{enumerate}
\item $\textbf{W}$ is a minimal $\Gamma_{E(3; 2; 1, 2)}$-isometric dilation of $\textbf{S}$.
		
\item If there exists a $\Gamma_{E(3; 2; 1, 2)}$-isometric dilation $\textbf{X} = (X_1, X_2, X_3, \tilde{X}_1, \tilde{X}_2)$ of $\textbf{S}$ such that $X_3$ is a minimal isometric dilation of $S_3$, then $\textbf{X}$ is unitarily equivalent to $\textbf{W}$. Moreover, the above identities  $(i)$ and $(ii)$ are also valid.
\end{enumerate}
	\end{thm}
	
	\begin{proof}
It is easy to see that $W_3$  is the minimal isometric dilation of $S_3$. The above construction is due to Sch\"{a}ffer. The Sch\"{a}ffer construction was for the unitary dilation, but we only present the isometry part in \eqref{w3}. It is evident from \eqref{w3} that $W_i^{*}|_{\mathcal H}=S_i^{*}$ for $1\leq i \leq 3$ and $\tilde{W}_j^{*}|_{\mathcal H}=\tilde{S}_j^{*}$ for $1\leq j \leq 2.$ To demonstrate  the $\mathbf{W}$ a $\Gamma_{E(3; 2; 1, 2)}$-isometry, we need to verify the following operator identities:		
		\begin{enumerate}
			\item $W_1, W_2, W_3$ and $\tilde{W}_1, \tilde{W}_2$ commute;
			
			\item $W_1 = \tilde{W}^*_2W_3, W_2 = \tilde{W}^*_1W_3$;
			
			\item $r(W_1)\leq  1, r(\frac{W_2}{2})\leq 1,  r(\frac{\tilde{W}_1}{2})\leq 1 $ and $ r(\tilde{W}_2)\leq 1$.
		\end{enumerate}
\noindent{\bf{Step-$1:$}} To verify the commutativity of $W_1, W_2, W_3$ and $\tilde{W}_1, \tilde{W}_2$, it is necessary to demonstrate the following identities:
		\begin{enumerate}
			\item[$(a)$] $W_iW_3=W_3W_i$ for $1 \leq i \leq 2$ and $W_1W_2=W_2W_1;$			
			\item[$(b)$] $W_i\tilde{W}_j=\tilde{W}_jW_i$ for $1\leq i,j\leq 2;$
			
			\item[$(c)$] $W_3 \tilde{W}_i= \tilde{W}_iW_3$ for $1\leq i \leq 2$ and $\tilde{W}_1\tilde{W}_2=\tilde{W}_2\tilde{W}_1.$
		\end{enumerate}
We first show that $W_1W_3 = W_3W_1$. Observe that
		\begin{equation}\label{w1w3}
			\begin{aligned}
				&W_1W_3 =
				\begin{bmatrix}
					S_1S_3 & 0 & 0 & 0 & \dots\\
					\tilde{G}^*_2D_{S_3}S_3 + G_1D_{S_3} & 0 & 0 & 0 & \dots\\
					\tilde{G}^*_2D_{S_3}& G_1 & 0 & 0 & \dots\\
					0 & \tilde{G}^*_2 & G_1 & 0 & \dots\\
					0 & 0 & \tilde{G}^*_2 & G_1 & \dots\\
					\vdots & \vdots & \vdots & \vdots & \ddots
				\end{bmatrix}, ~
				W_3W_1 =
				\begin{bmatrix}
					S_3S_1 & 0 & 0 & 0 & \dots\\
					D_{S_3}S_1 & 0 & 0 & 0 & \dots\\
					\tilde{G}^*_2D_{S_3} & G_1 & 0 & 0 & \dots\\
					0 & \tilde{G}^*_2 & G_1 & 0 & \dots\\
					0 & 0 & \tilde{G}^*_2 & G_1 & \dots\\
					\vdots & \vdots & \vdots & \vdots & \ddots
				\end{bmatrix}.
			\end{aligned}
		\end{equation}
It yields from Lemma \ref{s1s3} that $D_{S_3}S_1 = G_1D_{S_3} + \tilde{G}_2^*D_{S_3}S_3,$  which implies that $W_1W_3=W_3W_1.$ By using the similar argument, it follows that $W_2W_3=W_3W_2$  and $W_3 \tilde{W}_i= \tilde{W}_iW_3$ for $1\leq i \leq 2$. We will now show that $W_1W_2 = W_2W_1$. Note  that
		\begin{equation}
			\begin{aligned}
				W_1W_2 &=
				\begin{bmatrix}
					S_1S_2 & 0 & 0 & 0 & \dots\\
					\tilde{G}^*_2D_{S_3}S_2 + 2G_1\tilde{G}^*_1D_{S_3} & 2G_1G_2 & 0 & 0 & \dots\\
					2\tilde{G}^*_2\tilde{G}^*_1D_{S_3} & 2\tilde{G}^*_2G_2 + 2G_1\tilde{G}^*_1 & 2G_1G_2 & 0 & \dots\\
					0 & 2\tilde{G}^*_2\tilde{G}^*_1 & 2\tilde{G}^*_2G_2 + 2G_1\tilde{G}^*_1 & 2G_1G_2 & \dots\\
					\vdots & \vdots & \vdots & \vdots & \ddots
				\end{bmatrix}
			\end{aligned}
		\end{equation}
		and
		\begin{equation}
			\begin{aligned}
				W_2W_1 &=
				\begin{bmatrix}
					S_2S_1 & 0 & 0 & 0 & \dots\\
					2\tilde{G}^*_1D_{S_3}S_1 + 2G_2\tilde{G}^*_2D_{S_3} & 2G_2G_1 & 0 & 0 & \dots\\
					2\tilde{G}^*_1\tilde{G}^*_2D_{S_3} & 2\tilde{G}^*_1G_1 + 2G_2\tilde{G}^*_2 & 2G_2G_1 & 0 & \dots\\
					0 & 2\tilde{G}^*_1\tilde{G}^*_2 & 2\tilde{G}^*_1G_1 + 2G_2\tilde{G}^*_2 & 2G_2G_1 & \dots\\
					\vdots & \vdots & \vdots & \vdots & \ddots
				\end{bmatrix}.
			\end{aligned}
		\end{equation}
In order to verify $W_1W_2=W_2W_1$, we need to check the following opertor identities: \begin{equation}\label{G11}\tilde{G}^*_2D_{S_3}S_2 + 2G_1\tilde{G}^*_1D_{S_3} = 2\tilde{G}^*_1D_{S_3}S_1 + 2G_2\tilde{G}^*_2D_{S_3}.\end{equation}
It follows from \eqref{G11} and \eqref{s3} that 
\begin{equation}
			\begin{aligned}
				\tilde{G}^*_2D_{S_3}S_2 - 2\tilde{G}^*_1D_{S_3}S_1 &=\tilde{G}^*_2(2G_2D_{S_3}+2\tilde{G}^*_1D_{S_3}S_3)- 2\tilde{G}^*_1(G_1D_{S_3}+\tilde{G}^*_2D_{S_3}S_3)\\&=
				2(\tilde{G}^*_2G_2 - \tilde{G}^*_1G_1)D_{S_3}\\
				&=2(G_2\tilde{G}^*_2 - G_1\tilde{G}^*_1)D_{S_3}.
			\end{aligned}
		\end{equation}
Thus, we have $W_1W_2=W_2W_1.$ We will now show that $W_1\tilde{W}_2 = \tilde{W}_2W_1.$  Notice that
		\begin{equation}\label{w11}
			\begin{aligned}
				W_1\tilde{W}_2 &=
				\begin{bmatrix}
					S_1\tilde{S}_2 & 0 & 0 & 0 & \dots\\
					\tilde{G}^*_2D_{S_3}\tilde{S}_2 + G_1G^*_1D_{S_3} & G_1\tilde{G}_2 & 0 & 0 & \dots\\
					\tilde{G}^*_2G^*_1D_{S_3} & \tilde{G}^*_2\tilde{G}_2 + G_1G^*_1 & G_1\tilde{G}_2 & 0 & \dots\\
					0 & \tilde{G}^*_2G^*_1 & \tilde{G}^*_2\tilde{G}_2 + G_1G^*_1 & G_1\tilde{G}_2 & \dots\\
					\vdots & \vdots & \vdots & \vdots & \ddots
				\end{bmatrix}
			\end{aligned}
		\end{equation}
		and
		\begin{equation}\label{w12}
			\begin{aligned}
				\tilde{W}_2W_1 &=
				\begin{bmatrix}
					\tilde{S}_2S_1 & 0 & 0 & 0 & \dots\\
					G^*_1D_{S_3}S_1 + \tilde{G}_2\tilde{G}^*_2D_{S_3} & \tilde{G}_2G_1 & 0 & 0 & \dots\\
					G^*_1\tilde{G}^*_2D_{S_3} & G^*_1G_1 + \tilde{G}_2\tilde{G}^*_2 & \tilde{G}_2G_1 & 0 & \dots\\
					0 & G^*_1\tilde{G}^*_2 & G^*_1G_1 + \tilde{G}_2\tilde{G}^*_2 & \tilde{G}_2G_1 & \dots\\
					\vdots & \vdots & \vdots & \vdots & \ddots
				\end{bmatrix}.
			\end{aligned}
		\end{equation}
To prove $W_1\tilde{W}_2 = \tilde{W}_2W_1,$ it requires to demonstrate the following operator identities: \begin{equation}\label{w123}
		\tilde{G}^*_2D_{S_3}\tilde{S}_2 + G_1G^*_1D_{S_3} = G^*_1D_{S_3}S_1 + \tilde{G}_2\tilde{G}^*_2D_{S_3}.
		\end{equation}
It implies from \eqref{s3} and \eqref{w123} that 		
		\begin{equation}
			\begin{aligned}
				G^*_1D_{S_3}S_1 - \tilde{G}^*_2D_{S_3}\tilde{S}_2 
				&= G^*_1(G_1D_{S_3}+\tilde{G}^*_2D_{S_3}{S}_3)-\tilde{G}^*_2(\tilde{G}_2 D_{S_3}+G^*_1D_{S_3}S_3)\\&=(G^*_1G_1 - \tilde{G}^*_2\tilde{G}_2)D_{S_3}\\&=(G_1G^*_1 - \tilde{G}_2\tilde{G}^*_2)D_{S_3}.
			\end{aligned}
		\end{equation}
Similarly, we can also show that $W_2\tilde{W}_1=\tilde{W}_1W_2$ and $\tilde{W}_1\tilde{W}_2=\tilde{W}_2\tilde{W}_1.$

\noindent{\bf{Step-$2:$}} We now show that $W_1 = \tilde{W}^*_2W_3$. Observe that		
\begin{equation}\label{tildew}
			\begin{aligned}
				\tilde{W}^*_2W_3 &=
				\begin{bmatrix}
					\tilde{S}^*_2S_3 + D_{S_3}G_1D_{S_3} & 0 & 0 & 0 & \dots\\
					\tilde{G}^*_2D_{S_3} & G_1 & 0 & 0 & \dots\\
					0 & \tilde{G}^*_2 & G_1 & 0 & \dots\\
					0 & 0 & \tilde{G}^*_2 & G_1 & \dots\\
					\vdots & \vdots & \vdots & \vdots & \ddots
				\end{bmatrix}
				\end{aligned}
		\end{equation}
It follows from Theorem \ref{s1s2} and \eqref{tildew} that 
\begin{equation}\label{tildew1}
\begin{aligned}
				\tilde{W}^*_2W_3 &=
				\begin{bmatrix}
					S_1 & 0 & 0 & 0 & \dots\\
					\tilde{G}^*_2D_{S_3} & G_1 & 0 & 0 & \dots\\
					0 & \tilde{G}^*_2 & G_1 & 0 & \dots\\
					0 & 0 & \tilde{G}^*_2 & G_1 & \dots\\
					\vdots & \vdots & \vdots & \vdots & \ddots
				\end{bmatrix}
				\\&= W_1.
			\end{aligned}
		\end{equation}
Similarly, we can also show that $W_2 = \tilde{W}^*_1W_3$.
		
\noindent{\bf{Step-$3:$}} In order to complete the proof, it is essential to show the spectral radii of $W_1, \frac{W_2}{2}, \frac{\tilde{W}_1}{2}$, and $\tilde{W}_2$ do not exceed  $1$.  According to the definition, we can express $W_i,1\leq i \leq 2$ and $\tilde{W}_j,1\leq j \leq 2$ in the following form:
		\begin{equation*}
			\begin{aligned}
				W_i &=
				\begin{bmatrix}
					S_i & 0\\
					C_i & D_i
				\end{bmatrix}, 1\leq i \leq 2,
				\,\, and \,\,
				\tilde{W}_2 =
				\begin{bmatrix}
					\tilde{S}_j & 0\\
					\tilde{C}_j & \tilde{D}_j
				\end{bmatrix}, 1\leq j \leq 2,
			\end{aligned}
		\end{equation*}
		where
		\begin{equation*}
			\begin{aligned}
				&C_1 =
				\begin{bmatrix}
					\tilde{G}^*_2D_{S_3}\\
					0\\
					0\\
					\vdots
				\end{bmatrix},C_2 =
				\begin{bmatrix}
					2\tilde{G}^*_1D_{S_3}\\
					0\\
					0\\
					\vdots
				\end{bmatrix},\tilde{C}_1 =
				\begin{bmatrix}
					2G^*_2D_{S_3}\\
					0\\
					0\\
					\vdots
				\end{bmatrix}~\text{and}~\tilde{C}_2 =
				\begin{bmatrix}
					G^*_1D_{S_3}\\
					0\\
					0\\
					\vdots
				\end{bmatrix}
				~\text{and}~D_1 =
				\begin{bmatrix}
					G_1 & 0 & 0 & \dots\\
					\tilde{G}^*_2 & G_1 & 0 & \dots\\
					0 & \tilde{G}^*_2 & G_1 & \dots\\
					\vdots & \vdots & \vdots & \ddots
				\end{bmatrix},\\&D_2 =
				\begin{bmatrix}
					2G_2 & 0 & 0 & \dots\\
					2\tilde{G}^*_1& 2G_2 & 0 & \dots\\
					0 &2 \tilde{G}^*_1& 2G_2 & \dots\\
					\vdots & \vdots & \vdots & \ddots
				\end{bmatrix},\tilde{D}_1 =
				\begin{bmatrix}
					2\tilde{G}_1 & 0 & 0 & \dots\\
					2G^*_2 & 2\tilde{G}_1 & 0 & \dots\\
					0 & 2G^*_2& 2\tilde{G}_1 & \dots\\
					\vdots & \vdots & \vdots & \ddots
				\end{bmatrix} ~\text{and}~
				\tilde{D}_2 =
				\begin{bmatrix}
					\tilde{G}_2 & 0 & 0 & \dots\\
					G^*_1 & \tilde{G}_2 & 0 & \dots\\
					0 & G^*_1 & \tilde{G}_2 & \dots\\
					\vdots & \vdots & \vdots & \ddots
				\end{bmatrix}.
			\end{aligned}
		\end{equation*}
It follows from [Lemma $1$,  \cite{Du}] that  $\sigma(W_i) \subseteq \sigma(S_i) \cup \sigma(D_i),1\leq i \leq 2,$ and $\sigma(\tilde{W}_j)\subseteq \sigma(\tilde{S}_j) \cup \sigma(\tilde{D}_j)$ for $1\leq j \leq 2.$  We will arrive at the conclusion if we determine that $r(S_1),r(\frac{S_2}{2}),r(\frac{\tilde{S}_1}{2}),r(\tilde{S}_2),r(D_1),r(\frac{D_2}{2}),r(\frac{\tilde{D}_1}{2})$ and $r(\tilde{D}_2)$  do not exceed $1.$  As a result, we will demonstrate that the numerical radii of $\omega(D_1),\omega(\frac{D_2}{2}),\omega(\frac{\tilde{D}_1}{2})$ and $\omega(\tilde{D}_2)$ do not exceed $1$; because the spectral radius is not bigger than the numerical radius, this will complete our proof. Let $\tilde{\varphi}_i : \mathbb{D} \to \mathcal{B}(\mathcal{D}_{S_3}) $ for $1\leq i \leq 2 $ and  $\psi_j : \mathbb{D} \to \mathcal{B}(\mathcal{D}_{S_3}) $ for $1\leq j\leq 2, $ be the maps defined by $$\tilde{\varphi}_1(z) = G_1+ z\tilde{G}^*_2,~~\tilde{\varphi}_2(z)=\tilde{G}_2+zG_1^*,~~\psi_1(z)=G_2+z\tilde{G}^*_1,~~\psi_2(z)=\tilde{G}_1+zG_2^* .$$ Clearly, the maps $\tilde{\varphi}_i$ for $1\leq i \leq 2$ and $\psi_j$ for $1\leq j\leq 2$ are holomorphic on $\mathbb{D}$ and continuous on the boundary $\mathbb{T}$. The operators $D_1, \tilde{D}_2$ and $\frac{D_2}{2},\frac{\tilde{D}_1}{2}$ correspond to multiplication by the functions $\tilde{\varphi}_i,1\leq i \leq 2,$ and $\psi_j,1\leq j\leq 2,$ respectively, under the Hilbert space isomorphism that maps $\mathcal{D}_{S_3} \oplus \mathcal{D}_{S_3} \oplus \dots$ to $H^2(\mathbb{D}) \otimes \mathcal{D}_{S_3}.$ Because $\mathbf{S} = (S_1, S_2,S_3,\tilde{S}_1,\tilde{S}_2) $ is a $\Gamma_{E(3; 2; 1, 2)}$-contraction, it implies from Theorem \ref{s1s2} that 
$\omega(G_1 + z\tilde{G}_2)\leqslant 1$ and $\omega(G_2 + z\tilde{G}_1)\leqslant 1$ for all $z \in \mathbb{T}$ and hence we have $\omega(z_1G_1 + z_2\tilde{G}_2)\leqslant 1$ and $\omega(z_1G_2 + z_2\tilde{G}_1)\leqslant 1$ for all $z_1,z_2 \in \mathbb{T}.$
Let  $\tilde{Y}_1 = G_1 + z\tilde{G}^*_2$ and $\tilde{Y}_2=G_2+z\tilde{G}^*_1$ for all $z\in \mathbb T$ and $Z_1=e^{i\theta_1}\tilde{Y}_1 + e^{-i\theta_1}\tilde{Y}^*_1$ and $Z_2=e^{i\theta_2}\tilde{Y}_2 + e^{-i\theta_2}\tilde{Y}^*_2.$ Note that $Z_i$ for $1\leq i \leq 2$ is self-adjoint. By using an analogous argument as presented in Theorem \ref{conddilation}, we deduce that $\omega(G_1+ z\tilde{G}^*_2)\leq 1,$ $\omega(\tilde{G}_2+zG_1^*)\leq 1,$ $\omega(G_2+z\tilde{G}^*_1)\leq 1$ and $\omega(\tilde{G}_1+zG_2^*)\leq 1$ for all $z\in \mathbb T.$ This indicates that $\omega(D_1)\leq 1,\omega(\frac{D_2}{2})\leq 1,\omega(\frac{\tilde{D}_1}{2})\leq 1$ and $\omega(\tilde{D}_2)\leq 1$ and hence we have $r(D_1)\leq 1,r(\frac{D_2}{2})\leq 1,r(\frac{\tilde{D}_1}{2})\leq 1$ and $r(\tilde{D}_2)\leq 1.$ By using the aforementioned facts, we conclude that $\textbf{W}$ is a minimal $\Gamma_{E(3; 2; 1, 2)}$-isometric dilation of $\textbf{S}$.

\noindent{\bf{Proof of unitary equivalence:}} Suppose that $\textbf{S} = (S_1, S_2, S_3, \tilde{S}_1, \tilde{S}_2)$ possesses another $\Gamma_{E(3; 2; 1, 2)}$-isometric dilation \mbox{$\textbf{X} = (X_1, X_2, X_3, \tilde{X}_1, \tilde{X}_2)$} on $\mathcal{\tilde{K}}$ such that $X_3$ is a minimal isometric dilation of $S_3$. As $X_3$ is an isometry, it can be expressed in the following form \begin{equation}
			\begin{aligned}
				X_3 &=
				\begin{bmatrix}
					S_3 & 0\\
					C_3 & E_3
				\end{bmatrix}
			\end{aligned}
		\end{equation} with respect to the decomposition $\mathcal{\tilde{K}} = \mathcal{H} \oplus l^2(\mathcal{D}_{S_3}),$ where
		\begin{equation}
			\begin{aligned}
				C_3 &=
				\begin{bmatrix}
					D_{S_3}\\
					0\\
					0\\
					\vdots
				\end{bmatrix} : \mathcal{H} \to \ell^2(\mathcal{D}_{S_3}), \,\, \text{and} \,\,
				E_3 =
				\begin{bmatrix}
					0 & 0 & 0 & \dots\\
					I & 0 & 0 & \dots\\
					0 & I & 0 & \dots\\
					\vdots & \vdots & \vdots & \ddots
				\end{bmatrix} : \ell^2(\mathcal{D}_{S_3}) \to \ell^2(\mathcal{D}_{S_3}).
			\end{aligned}
		\end{equation}
Since $X_iX_3=X_3X_i$ for $1\leq i \leq 2$ and $\tilde{X}_jX_3=X_3\tilde{X}_j$ for $1\leq j \leq 2,$ it follows by a straightforward calculation that the operators $X_i, 1\leq i \leq 2,$ and $\tilde{X}_j, 1\leq j\leq 2,$ have the following operator matrix forms: 
		\begin{equation}\label{xi}
			\begin{aligned}
				X_i &=
				\begin{bmatrix}
					S_i & 0\\
					C_i & E_i
				\end{bmatrix} \,\, \,\, \,\, \text{and} \,\,
				\tilde{X}_j =
				\begin{bmatrix}
					\tilde{S}_j & 0\\
					\tilde{C}_j & \tilde{E}_j
				\end{bmatrix} \,\, 
			\end{aligned}
		\end{equation}
for some bounded linear operators $C_i$ and $E_i, 1\leq i \leq 2,$ and $\tilde{C}_j$ and $\tilde{E}_j,1\leq j \leq 2$,  with respect to the decomposition $\mathcal{\tilde{K}} = \mathcal{H} \oplus l^2(\mathcal{D}_{S_3}).$ There is a natural isomorphism between Hardy space $H^2(\mathcal{D}_{S_3}),$ consisting of functions that take values in $\mathcal{D}_{S_3}$ on the unit disk and $\ell^2(\mathcal{D}_{S_3})= \mathcal{D}_{S_3}\oplus \mathcal{D}_{S_3}\oplus\mathcal{D}_{S_3}\oplus \dots. $  The $\mathcal{D}_{S_3}$-valued multiplication operator $M^{\mathcal{D}_{S_3}}_z$ on $H^2(\mathcal{D}_{S_3})$ under the above Hilbert space isomorphism, is the operator $E_3$.  Since $X_iX_3=X_3X_i$  and $\tilde{X}_jX_3=X_3\tilde{X}_j, 1\leq i, j \leq 2,$  it yields that $E_iM^{\mathcal{D}_{S_3}}_z=M^{\mathcal{D}_{S_3}}_zE_i$ and $\tilde{E}_jM^{\mathcal{D}_{S_3}}_z=M^{\mathcal{D}_{S_3}}_z\tilde{E}_j$ for $1\leq i, j \leq 2$ which implies that  $E_i$'s are of the form $M^{\mathcal{D}_{S_3}}_{\tilde{\varphi}_i}$ for some $\tilde{\varphi}_i \in H^{\infty}(\mathcal{D}_{S_3})$ and $\tilde{E}_j$'s are of the form $M^{\mathcal{D}_{S_3}}_{\psi_j}$ for some $\psi_j \in H^{\infty}(\mathcal{D}_{S_3}).$

As $\textbf{X}$ is a $\Gamma_{E(3; 2; 1, 2)}$-isometry, it follows from  the characterization given in [Theorem $4.5$, \cite{apal2}] that $X_1 = \tilde{X}^*_2X_3$ and $X_2 = \tilde{X}^*_1X_3$. By using the similar argument as in Theorem \ref{conddilation}, we deduce the following operator identities:
\begin{equation}\label{s111}
			\begin{aligned}
				\begin{cases}
					(i) &S_1 - \tilde{S}^*_2S_3 = \tilde{C}^*_2C_3, \,\,\,\,
					C_1 = (M^{\mathcal{D}_{S_3}}_{\psi_2})^*C_3, \,\,\,\,
					M^{\mathcal{D}_{S_3}}_{\tilde{\varphi}_1} = (M^{\mathcal{D}_{S_3}}_{\psi_2})^*M^{\mathcal{D}_{S_3}}_z,\\
					(ii) &\tilde{S}_2 - S^*_1S_3 = C^*_1C_3, \,\,\,\,
					\tilde{C}_2 = (M^{\mathcal{D}_{S_3}}_{\tilde{\varphi}_1})^*C_3, \,\,
					M^{\mathcal{D}_{S_3}}_{{\psi}_2} = (M^{\mathcal{D}_{S_3}}_{\tilde{\varphi}_1})^*M^{\mathcal{D}_{S_3}}_z
				\end{cases}
			\end{aligned}
		\end{equation}
		$${\rm{and}}$$
		\begin{equation}\label{s1112}
			\begin{aligned}
				\begin{cases}
				(i) &S_2 - \tilde{S}^*_1S_3 = \tilde{C}^*_1C_3, \,\,\,\,
					C_2 = (M^{\mathcal{D}_{S_3}}_{{\psi}_1})^*C_3, \,\,\,\,
					M^{\mathcal{D}_{S_3}}_{\tilde{\varphi}_2} = (M^{\mathcal{D}_{S_3}}_{{\psi}_1})^*M^{\mathcal{D}_{S_3}}_z,\\
					(ii) &\tilde{S}_1 - S^*_2S_3 = C^*_2C_3, \,\,\,\,
					\tilde{C}_1 = (M^{\mathcal{D}_{S_3}}_{\tilde{\varphi}_2})^*C_3, \,\,\,\,
					M^{\mathcal{D}_{S_3}}_{{\psi}_1} = (M^{\mathcal{D}_{S_3}}_{\tilde{\varphi}_2})^*M^{\mathcal{D}_{S_3}}_z.
				\end{cases}
			\end{aligned}
		\end{equation}
By using the similar argument as in Theorem \ref{conddilation}	, it follows from 	\eqref{s111} and \eqref{s1112} that 
\begin{equation}\label{h11}\tilde{\varphi}_1(z) = H_1 + \tilde{H}^*_2z, \,\, \tilde{\varphi}_2(z) = 2H_2 + 2\tilde{H}^*_1z, \,\, \psi_1(z) =2 \tilde{H}_1 + 2H^*_2z, \,\, \psi_2(z) = \tilde{H}_2 + H^*_1z,\end{equation}
		where $H_1, H_2, \tilde{H}_1, \tilde{H}_2 \in \mathcal{B}(\mathcal{D}_{S_3})$. Hence, it follows from \eqref{xi}, \eqref{s111}, \eqref{s1112} and \eqref{h11} that 
		\begin{equation}\label{Eii}
			\begin{aligned}
				&M^{\mathcal{D}_{S_3}}_{\tilde{\varphi}_1}=E_1 =
				\begin{bmatrix}
					H_1 & 0 & 0 & \dots\\
					\tilde{H}^*_2 & H_1 & 0 & \dots\\
					0 & \tilde{H}^*_2 & H_1 & \dots\\
					\vdots & \vdots & \vdots & \ddots
				\end{bmatrix}, \,\,
				M^{\mathcal{D}_{S_3}}_{\tilde{\varphi}_2}=E_2 =
				\begin{bmatrix}
					2H_2 & 0 & 0 & \dots\\
					2\tilde{H}^*_1 & 2H_2 & 0 & \dots\\
					0 &2 \tilde{H}^*_1 & 2H_2 & \dots\\
					\vdots & \vdots & \vdots & \ddots
				\end{bmatrix},\\
				&M^{\mathcal{D}_{S_3}}_{\psi_1}=\tilde{E}_1 =
				\begin{bmatrix}
					2\tilde{H}_1 & 0 & 0 & \dots\\
					2H^*_2 & 2\tilde{H}_1 & 0 & \dots\\
					0 & 2H^*_2 & 2\tilde{H}_1 & \dots\\
					\vdots & \vdots & \vdots & \ddots
				\end{bmatrix}, \,\,
				M^{\mathcal{D}_{S_3}}_{\psi_2}=\tilde{E}_2 =
				\begin{bmatrix}
					\tilde{H}_2 & 0 & 0 & \dots\\
					H^*_1 & \tilde{H}_2 & 0 & \dots\\
					0 & H^*_1 & \tilde{H}_2 & \dots\\
					\vdots & \vdots & \vdots & \ddots
				\end{bmatrix}
			\end{aligned}
		\end{equation} on  $\mathcal{D}_{S_3} \oplus \mathcal{D}_{S_3} \oplus \dots.$
From \eqref{s111}, we have 
\begin{equation}\label{c111}
\begin{aligned}
C_1&= (M^{\mathcal{D}_{S_3}}_{\tilde{\psi}_2})^*C_3\\&=
				\begin{bmatrix}
					\tilde{H}^*_2D_{S_3}\\
					0\\
					0\\
					\vdots
				\end{bmatrix}.
			\end{aligned}
		\end{equation}
Similarly, from \eqref{s111},\eqref{s1112} and \eqref{Eii}, we deduce that 		
		\begin{equation}\label{c112}
			\begin{aligned}
			C_2 =\begin{bmatrix}
					2\tilde{H}^*_1D_{S_3}\\
					0\\
					0\\
					\vdots
				\end{bmatrix}, 
				\tilde{C}_1 =
				\begin{bmatrix}
					2H^*_2D_{S_3}\\
					0\\
					0\\
					\vdots
				\end{bmatrix}, 
				\tilde{C}_2 =
				\begin{bmatrix}
					H^*_1D_{S_3}\\
					0\\
					0\\
					\vdots
				\end{bmatrix},
			\end{aligned}
		\end{equation} respectively.
From \eqref{Eii}, \eqref{c111} and \eqref{c112}, we have
		\begin{equation*}
			\begin{aligned}
				&X_1 =
				\begin{bmatrix}
					S_1 & 0 & 0 & \dots\\
					\tilde{H}^*_2D_{S_3} & H_1 & 0 & \dots\\
					0 & \tilde{H}^*_2 & H_1 & \dots\\
					0 & 0 & \tilde{H}^*_2 & \dots\\
					\vdots & \vdots & \vdots & \ddots
				\end{bmatrix}, \,\,
				X_2 =
				\begin{bmatrix}
					S_2 & 0 & 0 & \dots\\
					2\tilde{H}^*_1D_{S_3} & 2H_2 & 0 & \dots\\
					0 & 2\tilde{H}^*_1 &2 H_2 & \dots\\
					0 & 0 & 2\tilde{H}^*_1 & \dots\\
					\vdots & \vdots & \vdots & \ddots
				\end{bmatrix},\\
				&\tilde{X}_1 =
				\begin{bmatrix}
					\tilde{S}_1 & 0 & 0 & \dots\\
					2H^*_2D_{S_3} & 2\tilde{H}_1 & 0 & \dots\\
					0 & 2H^*_2 & 2\tilde{H}_1 & \dots\\
					0 & 0 & 2H^*_2 & \dots\\
					\vdots & \vdots & \vdots & \ddots
				\end{bmatrix}, \,\,
				\tilde{X}_2 =
				\begin{bmatrix}
					\tilde{S}_2 & 0 & 0 & \dots\\
					H^*_1D_{S_3} & \tilde{H}_2 & 0 & \dots\\
					0 & H^*_1 & \tilde{H}_2 & \dots\\
					0 & 0 & H^*_1 & \dots\\
					\vdots & \vdots & \vdots & \ddots
				\end{bmatrix}
			\end{aligned}
		\end{equation*} with respect to the decomposition $\mathcal{\tilde{K}} = \mathcal{H} \oplus l^2(\mathcal{D}_{S_3}).$
As $X_1X_3=X_3X_1,$ we obtain
		\begin{equation}\label{s11s33}
			\begin{aligned}
				\begin{bmatrix}
					S_1 & 0 & 0 & \dots\\
					\tilde{H}^*_2D_{S_3} & G_1 & 0 & \dots\\
					0 & \tilde{H}^*_2 & G_1 & \dots\\
					0 & 0 & \tilde{H}^*_2 & \dots\\
					\vdots & \vdots & \vdots & \ddots
				\end{bmatrix}
				\begin{bmatrix}
					S_3 & 0 & 0 & \dots\\
					D_{S_3} & 0 & 0 & \dots\\
					0 & I & 0 & \dots\\
					0 & 0 & I & \dots\\
					\vdots & \vdots & \vdots & \ddots
				\end{bmatrix}
				&=
				\begin{bmatrix}
					S_3 & 0 & 0 & \dots\\
					D_{S_3} & 0 & 0 & \dots\\
					0 & I & 0 & \dots\\
					0 & 0 & I & \dots\\
					\vdots & \vdots & \vdots & \ddots
				\end{bmatrix}
				\begin{bmatrix}
					S_1 & 0 & 0 & \dots\\
					\tilde{H}^*_2D_{S_3} & G_1 & 0 & \dots\\
					0 & \tilde{H}^*_2 & G_1 & \dots\\
					0 & 0 & \tilde{H}^*_2 & \dots\\
					\vdots & \vdots & \vdots & \ddots
				\end{bmatrix}.
			\end{aligned}
		\end{equation}
By equating the $(2,1)$ entries on both sides of the aforementioned equation, we get
		\begin{equation}\label{h123}
			\begin{aligned}
				D_{S_3}S_1 = H_1D_{S_3} + \tilde{H}^*_2D_{S_3}S_3.
			\end{aligned}
		\end{equation}
As $X_2X_3=X_3X_2$ and  $\tilde{X}_jX_3=X_3\tilde{X}_j$ for $1\leq j \leq 2,$  similarly we can also deduce that 
		\begin{equation}\label{h124}
			\begin{aligned}
				D_{S_3}\tilde{S}_2 = \tilde{H}_2D_{S_3} + H^*_1D_{S_3}S_3, ~~D_{S_3}\frac{S_2}{2} = H_2D_{S_3} + \tilde{H}^*_1D_{S_3}S_3 \,\, \text{and} \,\,
				D_{S_3}\frac{\tilde{S}_1}{2} = \tilde{H}_1D_{S_3} + H^*_2D_{S_3}S_3.
			\end{aligned}
		\end{equation}
It follows from Lemma \ref{s1s3}, \eqref{h123} and \eqref{h124} that 	$$H_1 = G_1, H_2 = G_2, \tilde{H}_1 = \tilde{G}_1~{\rm{and}}~~\tilde{H}_2 = \tilde{G}_2.$$  As $X_1, X_2, X_3, \tilde{X}_1$ and $ \tilde{X}_2$ commute with each other,  we can also derive the conditions $(i)$ and $(ii)$ as well. Combining all the aforementioned arguments, we conclude that  $\textbf{X}$ and $\textbf{W}$ are unitarily equivalent.  This completes the proof.
	\end{proof}
	
	\section{Functional model for $\Gamma_{E(3; 3; 1, 1, 1)}$-Contractions and $\Gamma_{E(3; 2; 1, 2)}$-Contractions}
	
	In this section, we construct a concrete and explicit functional model for a  $\Gamma_{E(3; 3; 1, 1, 1)}$-contraction $\textbf{T}= (T_1, \dots, T_7)$ and for a $\Gamma_{E(3; 2; 1, 2)}$- contraction $\textbf{S}= (S_1, S_2, S_3, \tilde{S}_1, \tilde{S}_2).$ Let $T\in \mathcal B(\mathcal H)$ such that  $\|T\|\leq 1.$ The dimensions of the defect spaces $\mathcal{D}_T$ and $\mathcal{D}_{T^*}$ for operators $T$ and $T^*$ are denoted by $\delta_T$ and $\delta_{T^*}$, respectively. It is important to observe that $\delta_T = 0$ signifies that $T$ is an isometry, $\delta_{T^*} = 0$ implies that $T$ is a co-isometry, and the condition $\delta_T = \delta_{T^*} = 0$ indicates that $T$ is a unitary operator. We  state the following proposition without proof. This proposition is useful  in establishing the functional model for a  $\Gamma_{E(3; 3; 1, 1, 1)} $-contraction $\textbf{T}$ and for a $\Gamma_{E(3; 2; 1, 2)}$-contraction $\textbf{S}.$ 	
\begin{prop}[\cite{Spal 2014}, Proposition 5.2.]\label{vstar}
		Let $T$ be a contraction and $V$ be the minimal isometric dilation of $T$. Then the defect spaces of $T^*$ and $V^*$ are of equal dimension.
	\end{prop}
We now state a functional model for  pure $\Gamma_{E(3; 3; 1, 1, 1)}$-isometry.
	\begin{thm}\label{thm-14}[Theorem $4.6$, \cite{apal2}]\label{pure}
		Let $\textbf{V} = (V_1, \dots, V_7)$ be a commuting $7$-tuple of bounded operators on a separable Hilbert space $\mathcal{H}$. Then $\textbf{V}$ is a pure $\Gamma_{E(3; 3; 1, 1, 1)}$-isometry if and only if there exists a separable Hilbert space $\mathcal{E}$, a unitary $U : \mathcal{H} \to H^2(\mathcal{E})$ , functions $\Phi_1, \dots, \Phi_6$ in $H^{\infty}(\mathcal B(\mathcal{E}))$ and bounded operators $F_1,F_2,\ldots,F_6$ such that
		\begin{enumerate}
			\item $V_7 = U^*M^{\mathcal{E}}_zU$ and $V_i = U^*M^{\mathcal{E}}_{\Phi_i}U,$ where $\Phi_i(z) = F_i + F^*_{7-i}z,  1\leq i \leq  6$;
			
			\item the $H^{\infty}$ norm of the operator valued functions $F_i+F_{7-i}^*z$, $1\leq i \leq  6$, is at most $1$ for all $z\in \mathbb T;$ 
			
			\item $[F_i, F_j] = 0$ and $ [F_i, F^*_{7-j}] = [F_j, F^*_{7-i}]$ for $1\leq i, j \leq  6$.
		\end{enumerate}
	\end{thm}
We will construct a concrete and explicit functional model for a  $\Gamma_{E(3; 3; 1, 1, 1)} $-contraction $\mathbf{T} = (T_1, \dots, T_7)$, under the assumption that the adjoint $\mathbf{T}^* = (T_1^*, \dots, T_7^*)$ has fundamental operators $F_i$ and $F_{7-i},$ for $1\leq i \leq 6,$ which satisfy the following conditions:
\begin{equation}
[F_i,F_j]=0 ~~{\rm{and}}~~[F_{7-i}^*,F_j]=[F_{7-j}^*,F_i], 1\leq i,j \leq 6.
\end{equation}
\begin{thm}
		Let $\textbf{T} = (T_1, \dots, T_7)$ be a $\Gamma_{E(3; 3; 1, 1, 1)}$-contraction on some Hilbert space $\mathcal{H}$ such that the adjoint $\textbf{T}^* = (T^*_1, \dots, T^*_7)$  has fundamental operators $F_i$ and $F_{7-i}, 1\leq i \leq 6,$ which satisfy the following conditions:
\begin{enumerate}
			\item[$(i)$] $[F_i, F_j] = 0$ for $ 1\leq i,j \leq 6$,
			
			\item[$(ii)$] $[F^*_{7-i}, F_j] = [F^*_{7-j}, F_i]$ for $1 \leq i,j \leq 6$.
		\end{enumerate}
		Let  $\mathcal{K}^* = \mathcal{H} \oplus \mathcal{D}_{T^*_7} \oplus \mathcal{D}_{T^*_7} \oplus \dots =  \mathcal{H} \oplus \ell^2(\mathcal{D}_{T^*_7}).$ Let $\tilde{\textbf{V}} = (\tilde{V}_1, \dots, \tilde{V}_7)$ be a $7$-tuple of operators defined on $\mathcal{K}^*$ by
		\begin{equation}\label{tildev}
			\begin{aligned}
				\tilde{V}_i =
				\begin{bmatrix}
					T_i & D_{T^*_7}F_{7-i} & 0 & 0 & \dots\\
					0 & F^*_i & F_{7-i} & 0 & \dots\\
					0 & 0 & F^*_i & F_{7-i} & \dots\\
					0 & 0 & 0 & F^*_i & \dots\\
					\vdots & \vdots & \vdots & \vdots & \ddots
				\end{bmatrix}
				  ~~~~\text{and}~~~~
				\tilde{V}_7 =
				\begin{bmatrix}
					T_7 & D_{T^*_7} & 0 & \dots\\
					0 & 0 & I & 0 & \dots\\
					0 & 0 & 0 & I & \dots\\
					\vdots & \vdots & \vdots & \vdots & \ddots
				\end{bmatrix}.
			\end{aligned}
		\end{equation}
		Then
		\begin{enumerate}
			\item $\tilde{\textbf{V}}$ is a $\Gamma_{E(3; 3; 1, 1, 1)} $-co-isometric dilation of $\textbf{T}$, $\mathcal{H}$ serves as a common invariant subspace of $\tilde{V}_1, \dots, \tilde{V}_7$ and $\tilde{V}_i|_{\mathcal{H}} = T_i$ for $1\leq i \leq 7;$
			
			\item there exists an orthogonal decomposition $\mathcal{K}^* = \mathcal{K}_1 \oplus \mathcal{K}_2$ of $\mathcal{K}^*$ into reducing subspaces of $\tilde{V}_1, \dots, \tilde{V}_7$ such that $(\tilde{V}_1|_{\mathcal{K}_1}, \dots, \tilde{V}_7|_{\mathcal{K}_1})$ is a $\Gamma_{E(3; 3; 1, 1, 1)}$-unitary and $(\tilde{V}_1|_{\mathcal{K}_2}, \dots, \tilde{V}_7|_{\mathcal{K}_2})$ is a \textit{pure} $\Gamma_{E(3; 3; 1, 1, 1)}$-co-isometry;

			\item $\mathcal{K}_2$ can be identified as $H^2({\mathcal{D}_{\tilde{V}_7}})$, where $\mathcal{D}_{\tilde{V}_7}$ has same dimension as of $\mathcal{D}_{T_7}$. The operators $\tilde{V}_1|_{\mathcal{K}_2}, \dots, \tilde{V}_7|_{\mathcal{K}_2}$ are unitarily equivalent to $M_{A_1 + A^*_6\overline{z}}, \dots, M_{A_6 + A^*_1\overline{z}}, M_{\overline{z}}$  on the vector-valued Hardy space $H^2({\mathcal{D}_{\tilde{V}_7}})$, where $A_i,1\leq i \leq 6$, are the fundamental operators of $\tilde{\textbf{V}} = (\tilde{V}_1, \dots, \tilde{V}_7)$.		\end{enumerate}
	\end{thm}
	
	\begin{proof}
It follows from  \eqref{tildev} that $\tilde{V}^*_7$ on $\mathcal{K}^*$ is the minimal isometric dilation of $T^*_7.$  Observe that for $1\leq i \leq 6$ 
		\begin{equation*}
			\begin{aligned}
				&\tilde{V}^*_i =
				\begin{bmatrix}
					T^*_i & 0 & 0 & \dots\\
					F^*_{7-i}D_{T^*_7} & F_i & 0 & \dots\\
					0 & F^*_{7-i} & F_i & \dots\\
					0 & 0 & F^*_{7-i} & \dots\\
					\vdots & \vdots & \vdots & \ddots
				\end{bmatrix}, \,\,
				\tilde{V}^*_7 =
				\begin{bmatrix}
					T^*_7 & 0 & 0 & 0& \dots\\
					D_{T^*_7} & 0 & 0 & 0&\dots\\
					0 &  I& 0& 0 &\dots\\
					0 & 0 & I & 0 &\dots\\
					\vdots & \vdots & \vdots &\vdots & \ddots
				\end{bmatrix}.
			\end{aligned}
		\end{equation*}
		By Theorem \ref{conddilation}, it follows that $\tilde{V}^* = (\tilde{V}^*_1, \dots, \tilde{V}^*_7)$ is minimal $\Gamma_{E(3; 3; 1, 1, 1)}$-isometric dilation of $\textbf{T}^* = (T_1^*, \dots, T_7^*).$ Consequently,  we conclude that $\tilde{\textbf{V}}$ is a $\Gamma_{E(3; 3; 1, 1, 1)}$-co-isometric dilation of $\textbf{T}$. Clearly, $\mathcal{H}$ is a common invariant subspace of $\tilde{V}_1, \dots, \tilde{V}_7,$ and  $\tilde{V}_i|_{\mathcal{H}} = T_i$ for $1\leq i \leq 7$. This proves $(1)$.
		
As ${\tilde{V}}^*$ is a $\Gamma_{E(3; 3; 1, 1,1)}$-isometry,  by \textit{Wold decomposition} [Theorem $4.1$,\cite{apal2}], there exists an orthogonal decomposition of  $\mathcal{K}^* = \mathcal{K}_1 \oplus \mathcal{K}_2$ into two reducing subspaces $\mathcal{K}_1$ and $\mathcal{K}_2$ such that $(\tilde{V}^*_1|_{\mathcal{K}_1}, \dots, \tilde{V}^*_7|_{\mathcal{K}_1})$ is a $\Gamma_{E(3; 3; 1, 1, 1)}$-unitary and $(\tilde{V}^*_1|_{\mathcal{K}_2}, \dots, \tilde{V}^*_7|_{\mathcal{K}_2})$ is a pure $\Gamma_{E(3; 3; 1, 1, 1)}$-isometry. Thus $(\tilde{V}_1|_{\mathcal{K}_2}, \dots, \tilde{V}_7|_{\mathcal{K}_2})$ is a pure $\Gamma_{E(3; 3; 1, 1, 1)}$-co-isometry. As $(\tilde{V}^*_1|_{\mathcal{K}_1}, \dots, \tilde{V}^*_7|_{\mathcal{K}_1})$ is a $\Gamma_{E(3; 3; 1, 1, 1)}$-unitary, it yields from [Theorem $3.2$,\cite{apal2}] that  $(\tilde{V}^*_1|_{\mathcal{K}_1}, \dots, \tilde{V}^*_7|_{\mathcal{K}_1})$ is a $\Gamma_{E(3; 3; 1, 1, 1)}$-contraction and  $\tilde{V}^*_7|_{\mathcal{K}_1}$ is unitary, and so $(\tilde{V}_1|_{\mathcal{K}_1}, \dots, \tilde{V}_7|_{\mathcal{K}_1})$ is a $\Gamma_{E(3; 3; 1, 1, 1)}$-contraction and  $\tilde{V}_7|_{\mathcal{K}_1}$ is a unitary. This indicates that $(\tilde{V}_1|_{\mathcal{K}_1}, \dots, \tilde{V}_7|_{\mathcal{K}_1})$ is a $\Gamma_{E(3; 3; 1, 1, 1)}$-unitary.  Set $\tilde{V}^{(1)}_{i}:=\tilde{V}_i|_{\mathcal{K}_1}$ and $\tilde{V}^{(2)}_{i}:=\tilde{V}_i|_{\mathcal{K}_2}$ for $1\leq i\leq 7$. Thus, $\tilde{V}_i$ have the following form  
		\begin{equation*}
			\begin{aligned}
				&\tilde{V}_i =
				\begin{bmatrix}
					\tilde{V}^{(1)}_{i} & 0\\
					0 & \tilde{V}^{(2)}_{i}
				\end{bmatrix},\;\; 1\leq i\leq 7,
			\end{aligned}
		\end{equation*}
 with respect to the decomposition $\mathcal{K}^* = \mathcal{K}_1 \oplus \mathcal{K}_2$. 
  As $(\tilde{V}^{(1)}_{1}, \dots, \tilde{V}^{(1)}_{7})$ is a $\Gamma_{E(3; 3; 1, 1, 1)}$-unitary,  we deduce that  $(\tilde{V}_1, \dots, \tilde{V}_7)$ and $(\tilde{V}_{1}^{(2)}, \dots, \tilde{V}_{7}^{(2)})$ have the same fundamental operators $A_i,1\leq i \leq 6.$ It follows from Theorem \ref{pure} that 
		\begin{enumerate}
			\item[$(i)$]  $\mathcal{K}_2$ can be identified with $H^2({\mathcal{D}_{\tilde{V}_7}})$;
			
			\item[$(ii)$] the operators $\tilde{V}_1|_{\mathcal{K}_2}, \dots, \tilde{V}_7|_{\mathcal{K}_2}$ are unitarily equivalent to $M_{A_1 + A^*_6\overline{z}}, \dots, M_{A_6 + A^*_1\overline{z}}, M_{\overline{z}},$ respectively, on the vector-valued Hardy space $H^2({\mathcal{D}_{\tilde{V}_7}})$.
		\end{enumerate}
Since $\tilde{V}^*_7$ is the  minimal isometric dilation of $T^*_7$, it follows from  Proposition \ref{vstar} that the dimensions of $\mathcal{D}_{T_7}$ and $\mathcal{D}_{\tilde{V}_7}$ are equal. This completes the proof.
	\end{proof}
We now state the model for a pure $\Gamma_{E(3; 2; 1, 2)}$ isometry.
\begin{thm}\label{thm-15} [Theorem $4.7$, \cite{apal2}]
		Let $\textbf{W} = (W_1, W_2, W_3, \tilde{W}_1, \tilde{W}_2)$ be a $5$-tuple of commuting bounded operators on a separable Hilbert space $\mathcal{H}$. Then $\textbf{W}$ is a pure $\Gamma_{E(3; 2; 1, 2)}$ isometry if and only if there exist a separable Hilbert space $\mathcal{F}$, a unitary $\tilde{U} : \mathcal{H} \to H^2(\mathcal{F})$, functions $\tilde{\Phi}_i$ in $H^{\infty}(\mathcal B(\mathcal{F}))$ and $\tilde{\Psi}_j$ in $H^{\infty}(\mathcal B(\mathcal{F}))$ for $1 \leq i, j\leq 2$  and bounded operators $G_1, 2G_2,2\tilde{G}_1,\tilde{G}_2$ on $\mathcal{F }$ such that
		\begin{enumerate}
			\item $W_3 = \tilde{U}^*M^{\mathcal{F}}_z\tilde{U}, W_1 = \tilde{U}^*M^{\mathcal{F}}_{\tilde{\Phi}_1}\tilde{U}, W_2 = \tilde{U}^*M^{\mathcal{F}}_{\tilde{\Psi}_1}\tilde{U}, \tilde{W}_1 = \tilde{U}^*M^{\mathcal{F}}_{\tilde{\Psi}_2}\tilde{U}$ and $ \tilde{W}_2 =\tilde{U}^*M^{\mathcal{F}}_{\tilde{\Phi}_2}\tilde{U},$ where $\tilde{\Phi}_1(z)=G_1+\tilde{G}_2^*z, \tilde{\Phi}_2(z)=\tilde{G}_2+zG_1^*$, $\tilde{\Psi}_1(z)=2G_2+2\tilde{G}_1^*z$ and $\tilde{\Psi}_2(z)=2\tilde{G}_1+2G_2^*z$ for all $z\in \mathbb T;$		
			\item the $H^{\infty}$ norm of the operator functions $G_1+\tilde{G}_{2}^*z, \tilde{G}_2+G_{1}^*z,G_2+\tilde{G}_{1}^*z$ and $\tilde{G}_1+G_{2}^*z$ are at most $1$;
			
			\item
			\begin{enumerate}
				\item $[G_1, \tilde{G}_2] = 0, [G_1, G_1^*] = [\tilde{G}_2, \tilde{G}_2^*]$;
				
				\item $[G_1,G_2]=[G_2,G_1],[G_1,\tilde{G}_1^*]=[G_2,\tilde{G}_2^*],[\tilde{G}_1,\tilde{G}_2]=[\tilde{G}_2,\tilde{G}_1]$;
				
				\item $[G_1, \tilde{G}_1] = 0, [G_1,G_2^*] = [\tilde{G}_1,\tilde{G}_2^*],[G_2,\tilde{G}_2]=0$;
				
				\item $[G_2,\tilde{G}_1] = 0, [G_2, G_2^*] = [\tilde{G}_1,\tilde{G}_1^*]$.
			\end{enumerate}
		\end{enumerate}
	\end{thm}
		
We will construct a concrete and explicit functional model for a  $\Gamma_{E(3; 2; 1, 2)} $-contraction $\textbf{S} = (S_1, S_2, S_3, \tilde{S}_1, \tilde{S}_2)$, under the assumption that the adjoint $\textbf{S}^* = (S_1^*, S_2^*, S_3^*, \tilde{S}_1^*, \tilde{S}_2^*)$ has fundamental operators.
	\begin{thm}
		Let $\textbf{S} = (S_1, S_2, S_3, \tilde{S}_1, \tilde{S}_2)$ be a $\Gamma_{E(3; 2; 1, 2)}$-contraction on a Hilbert space $\mathcal{H}$ such that the adjoint $\textbf{S}^* = (S^*_1, S^*_2, S^*_3, \tilde{S}^*_1, \tilde{S}^*_2)$ has fundamental operators  $G_1, 2G_2, 2\tilde{G}_1, \tilde{G}_2$ which satisfy the following conditions:
		\begin{enumerate}
			\item[$(i)$] $[G_1,\tilde{G}_i]=0$ for $1 \leq i \leq 2,$ $[G_2,\tilde{G}_j]=0$ for $1 \leq j \leq 2,$ and $[G_1, G_2] = [\tilde{G}_1, \tilde{G}_2]  = 0$;
			
			\item[$(ii)$] $[G_1, G^*_1] = [\tilde{G}_2, \tilde{G}^*_2], [G_2, G^*_2] = [\tilde{G}_1, \tilde{G}^*_1], [G_1, \tilde{G}^*_1] = [G_2, \tilde{G}^*_2], [\tilde{G}_1, G^*_1] = [\tilde{G}_2, G^*_2],$\\$ [G_1, G^*_2] = [\tilde{G}_1, \tilde{G}^*_2], [G^*_1, G_2] = [\tilde{G}^*_1, \tilde{G}_2]$.
		\end{enumerate}	
Let  $\mathcal{\tilde{K}}^* = \mathcal{H} \oplus \mathcal{D}_{S^*_3} \oplus \mathcal{D}_{S^*_3} \oplus \dots = \mathcal{H} \oplus \ell^2(\mathcal{D}_{S^*_3}).$ Let $\hat{\textbf{W}} = (\hat{W}_1, \hat{W}_2, \hat{W}_3, \hat{\tilde{W}}_1,\hat{\tilde{W}}_2)$ be a $5$-tuple of operators defined on $\mathcal{\tilde{K}}^*$ by
		\begin{equation}\label{tildeWW}
			\begin{aligned}
				&\hat{W}_1 =
				\begin{bmatrix}
					S_1 & D_{S^*_3}\tilde{G}_2 & 0 & 0 & \dots\\
					0 & G^*_1 & \tilde{G}_2 & 0 & \dots\\
					0 & 0 & G^*_1 & \tilde{G}_2 & \dots\\
					0 & 0 & 0 & G^*_1 & \dots\\
					\vdots & \vdots & \vdots & \vdots & \ddots
				\end{bmatrix}, \,\,
				\hat{W}_2 =
				\begin{bmatrix}
					S_2 & 2D_{S^*_3}\tilde{G}_1 & 0 & 0 & \dots\\
					0 & 2G^*_2 & 2\tilde{G}_1 & 0 & \dots\\
					0 & 0 & 2G^*_2 & 2\tilde{G}_1 & \dots\\
					0 & 0 & 0 & 2G^*_2 & \dots\\
					\vdots & \vdots & \vdots & \vdots & \ddots
				\end{bmatrix},
				\hat{W}_3 =
				\begin{bmatrix}
					S_3 & D_{S^*_3} & 0 & 0 & \dots\\
					0 & 0 & I & 0 & \dots\\
					0 & 0 & 0 & I & \dots\\
					0 & 0 & 0 & 0 & \dots\\
					\vdots & \vdots & \vdots & \vdots & \ddots
				\end{bmatrix},\\
				&\hspace{2cm}
				\hat{\tilde{W}}_1 =
				\begin{bmatrix}
					\tilde{S}_1 & 2D_{S^*_3}G_2 & 0 & 0 & \dots\\
					0 & 2\tilde{G}^*_1 & 2G_2 & 0 & \dots\\
					0 & 0 & 2\tilde{G}^*_1 & 2G_2 & \dots\\
					0 & 0 & 0 & 2\tilde{G}^*_1 & \dots\\
					\vdots & \vdots & \vdots & \vdots & \ddots
				\end{bmatrix}, \,\,
				\hat{\tilde{W}}_2 =
				\begin{bmatrix}
					\tilde{S}_2 & D_{S^*_3}G_1 & 0 & 0 & \dots\\
					0 & \tilde{G}^*_2 & G_1 & 0 & \dots\\
					0 & 0 & \tilde{G}^*_2 & G_1 & \dots\\
					0 & 0 & 0 & \tilde{G}^*_2 & \dots\\
					\vdots & \vdots & \vdots & \vdots & \ddots
				\end{bmatrix}.
			\end{aligned}
		\end{equation}
		Then
		\begin{enumerate}
			\item $\hat{\textbf{W}}$ is a $\Gamma_{E(3; 2; 1, 2)}$-co-isometric dilation of $\textbf{S}$,  $\mathcal{H}$ is a common invariant subspace of the operators $\hat{W}_1, \hat{W}_2, \hat{W}_3, \hat{\tilde{W}}_1,\hat{\tilde{W}}_2$ and $\hat{W}_1|_{\mathcal{H}} = S_1, \hat{W}_2|_{\mathcal{H}} = S_2, \hat{W}_3|_{\mathcal{H}} = S_3, \hat{\tilde{W}}_1|_{\mathcal{H}} = \tilde{S}_1, \hat{\tilde{W}}_2|_{\mathcal{H}} = \tilde{S}_2$,
			
			\item there exists an orthogonal decomposition $\mathcal{\tilde{K}}^* = \mathcal{\tilde{K}}_1 \oplus \mathcal{\tilde{K}}_2$ of $\mathcal{\tilde{K}}^*$ into reducing subspaces such that $(\hat{W}_1|_{\mathcal{\tilde{K}}_1}, \hat{W}_2|_{\mathcal{\tilde{K}}_1}, \hat{W}_3|_{\mathcal{\tilde{K}}_1}, \hat{\tilde{W}}_1|_{\mathcal{\tilde{K}}_1}, \hat{\tilde{W}}_2|_{\mathcal{\tilde{K}}_1})$ is a $\Gamma_{E(3; 2; 1, 2)}$-unitary and $(\hat{W}_1|_{\mathcal{\tilde{K}}_2}, \hat{W}_2|_{\mathcal{\tilde{K}}_2}, \hat{W}_3|_{\mathcal{\tilde{K}}_2}, \hat{\tilde{W}}_1|_{\mathcal{\tilde{K}}_2}, \hat{\tilde{W}}_2|_{\mathcal{\tilde{K}}_2})$ is a \textit{pure} $\Gamma_{E(3; 2; 1, 2)}$-co-isometry.
			
			\item $\mathcal{\tilde{K}}_2$ can be identified as $H^2 (\mathcal{D}_{\hat{W}_3})$, where $\mathcal{D}_{\hat{W}_3}$ has same dimension as of $\mathcal{D}_{S_3}$. The operators  $\hat{W}_1|_{\mathcal{\tilde{K}}_2}, \hat{W}_2|_{\mathcal{\tilde{K}}_2},$\\$ \hat{W}_3|_{\mathcal{\tilde{K}}_2}, \hat{\tilde{W}}_1|_{\mathcal{\tilde{K}}_2}, \hat{\tilde{W}}_2|_{\mathcal{\tilde{K}}_2}$ are unitarily equivalent to the multiplication operators $M_{B_1 + \tilde{B}^*_2\overline{z}}, M_{2B_2 + 2\tilde{B}^*_1\overline{z}}, M_{\overline{z}}, $\\$M_{2\tilde{B}_1 + 2B^*_2\overline{z}}, M_{\tilde{B}_2 + B^*_1\overline{z}},$ respectively, on the vector-valued Hardy space $H^2 (\mathcal{D}_{\hat{W}_3})$, where $B_1,2B_2,2\tilde{B}_1,\tilde{B}_2$ are the fundamental operators of   $\hat{\textbf{W}} = (\hat{W}_1, \hat{W}_2, \hat{W}_3, \hat{\tilde{W}}_1,\hat{\tilde{W}}_2)$.		\end{enumerate}
	\end{thm}
	
	\begin{proof}
It follows from \eqref{tildeWW} that $\hat{W}^*_3$ on $\mathcal{\tilde{K}}^*$ is the minimal isometric dilation of $S^*_3$. Note that
		\begin{equation}\label{tildeWWW}
			\begin{aligned}
				&\hat{W}^*_1 =
				\begin{bmatrix}
					S^*_1 & 0 & 0 & \dots\\
					\tilde{G}^*_2D_{S^*_3} & G_1 & 0 & \dots\\
					0 & \tilde{G}^*_2 & G_1 & \dots\\
					0 & 0 & \tilde{G}^*_2 & \dots\\
					\vdots & \vdots & \vdots & \ddots
				\end{bmatrix}, \,\,
				\hat{W}^*_2 =
				\begin{bmatrix}
					S^*_2 & 0 & 0 & \dots\\
					2\tilde{G}^*_1D_{S^*_3} & 2G_2 & 0 & \dots\\
					0 & 2\tilde{G}^*_1 & 2G_2 & \dots\\
					0 & 0 & 2\tilde{G}^*_1 & \dots\\
					\vdots & \vdots & \vdots & \ddots
				\end{bmatrix},
				\hat{W}^*_3 =
				\begin{bmatrix}
					S^*_3 & 0 & 0 & \dots\\
					D_{S^*_3} & 0 & 0 & \dots\\
					0 & I & 0 & \dots\\
					0 & 0 & I & \dots\\
					\vdots & \vdots & \vdots & \ddots
				\end{bmatrix},\\
				&\hspace{2cm}
				\hat{\tilde{W}}^*_1 =
				\begin{bmatrix}
					\tilde{S}^*_1 & 0 & 0 & \dots\\
					2G^*_2D_{S^*_3} & 2\tilde{G}_1 & 0 & \dots\\
					0 & 2G^*_2 & 2\tilde{G}_1 & \dots\\
					0 & 0 & 2G^*_2 & \dots\\
					\vdots & \vdots & \vdots & \ddots
				\end{bmatrix}, \,\,
				\hat{\tilde{W}}^*_2 =
				\begin{bmatrix}
					\tilde{S}^*_2 & 0 & 0 & \dots\\
					G^*_1D_{S^*_3} & \tilde{G}_2 & 0 & \dots\\
					0 & G^*_1 & \tilde{G}_2 & \dots\\
					0 & 0 & G^*_1 & \dots\\
					\vdots & \vdots & \vdots & \ddots
				\end{bmatrix}.
			\end{aligned}
		\end{equation}
It yields from Theorem \ref{condilation1} that  $(\hat{W}^*_1, \hat{W}^*_2, \hat{W}^*_3, \hat{\tilde{W}}^*_1, \hat{\tilde{W}}^*_2)$ is a minimal $\Gamma_{E(3; 2; 1, 2)}$-isometric dilation of $\textbf{S}^* = (S^*_1, S^*_2, S^*_3, \tilde{S}^*_1, \tilde{S}^*_2)$.  Therefore,  $\hat{\textbf{W}}$ is $\Gamma_{E(3; 2; 1, 2)}$-co-isometric dilation of $\textbf{S}$. It is evident that  $\mathcal{H}$ is a common invariant subspace for $\hat{W}_1, \hat{W}_2, \hat{W}_3, \hat{\tilde{W}}_1, \hat{\tilde{W}}_2$ and  $\hat{W}_1|_{\mathcal{H}} = S_1, \hat{W}_2|_{\mathcal{H}} = S_2, \hat{W}_3|_{\mathcal{H}} = S_3, \hat{\tilde{W}}_1|_{\mathcal{H}} = \tilde{S}_2, \hat{\tilde{W}}_2|_{\mathcal{H}} = \tilde{S}_2$. As $\hat{\textbf{W}}^*$ is a $\Gamma_{E(3; 2; 1, 2)}$-isometry, it implies  from \textit{Wold decomposition theorem} that there exists an orthogonal decompositon of  $\mathcal{\tilde{K}}^* = \mathcal{\tilde{K}}_1 \oplus \mathcal{\tilde{K}}_2$ into  two reducing subspaces $\mathcal{\tilde{K}}_1$ and $\mathcal{\tilde{K}}_2$ of $\mathcal{\tilde{K}}^*$ such that $(\hat{W}_1^*|_{\mathcal{\tilde{K}}_1}, \hat{W}_2^*|_{\mathcal{\tilde{K}}_1}, \hat{W}_3^*|_{\mathcal{\tilde{K}}_1}, \hat{\tilde{W}}_1^*|_{\mathcal{\tilde{K}}_1}, \\\hat{\tilde{W}}_2^*|_{\mathcal{\tilde{K}}_1})$ is a $\Gamma_{E(3; 2; 1, 2)}$-unitary and $(\hat{W}_1^*|_{\mathcal{\tilde{K}}_2}, \hat{W}_2^*|_{\mathcal{\tilde{K}}_2}, \hat{W}_3^*|_{\mathcal{\tilde{K}}_2}, \hat{\tilde{W}}_1^*|_{\mathcal{\tilde{K}}_2}, \hat{\tilde{W}}_2^*|_{\mathcal{\tilde{K}}_2})$ is a \textit{pure} $\Gamma_{E(3; 2; 1, 2)}$-isometry. \\Since $(\hat{W}_1^*|_{\mathcal{\tilde{K}}_1}, \hat{W}_2^*|_{\mathcal{\tilde{K}}_1}, \hat{W}_3^*|_{\mathcal{\tilde{K}}_1}, \hat{\tilde{W}}_1^*|_{\mathcal{\tilde{K}}_1}, \hat{\tilde{W}}_2^*|_{\mathcal{\tilde{K}}_1})$ is a $\Gamma_{E(3; 2; 1, 2)}$-unitary, it follows from [Theorem $3.7$, \cite{apal2}] that $(\hat{W}_1^*|_{\mathcal{\tilde{K}}_1}, \hat{W}_2^*|_{\mathcal{\tilde{K}}_1},\hat{W}_3^*|_{\mathcal{\tilde{K}}_1}, \hat{\tilde{W}}_1^*|_{\mathcal{\tilde{K}}_1}, \hat{\tilde{W}}_2^*|_{\mathcal{\tilde{K}}_1})$ is a $\Gamma_{E(3; 2; 1, 2)}$-contraction and $\hat{W}_3^*|_{\mathcal{\tilde{K}}_1}$ is unitary, and hence $(\hat{W}_1|_{\mathcal{\tilde{K}}_1}, \hat{W}_2|_{\mathcal{\tilde{K}}_1}, \\\hat{W}_3|_{\mathcal{\tilde{K}}_1}, \hat{\tilde{W}}_1|_{\mathcal{\tilde{K}}_1}, \hat{\tilde{W}}_2|_{\mathcal{\tilde{K}}_1})$ is a $\Gamma_{E(3; 2; 1, 2)}$-contraction and  $\hat{W}_3|_{\mathcal{\tilde{K}}_1}$ is unitary. Thus, by [Theorem $3.7$, \cite{apal2}], we conclude that $(\hat{W}_1|_{\mathcal{\tilde{K}}_1}, \hat{W}_2|_{\mathcal{\tilde{K}}_1}, \hat{W}_3|_{\mathcal{\tilde{K}}_1}, \hat{\tilde{W}}_1|_{\mathcal{\tilde{K}}_1}, \hat{\tilde{W}}_2|_{\mathcal{\tilde{K}}_1})$ is a $\Gamma_{E(3; 2; 1, 2)}$-unitary.  Set  $\hat{W}_{i}^{(1)}:= \hat{W}_i|_{\mathcal{\tilde{K}}_1}, \hat{W}_{i}^{(2)}:= \hat{W}_i|_{\mathcal{\tilde{K}}_2}$, $\hat{\tilde{W}}^{(1)}_j:= \hat{\tilde{W}}_j|_{\mathcal{\tilde{K}}_1}$ and $\hat{\tilde{W}}^{(2)}_j:= \hat{\tilde{W}}_j|_{\mathcal{\tilde{K}}_2}$ for $1\leq i \leq 3, 1\leq j \leq 2$.
Therefore, $(\hat{W}_1, \hat{W}_2, \hat{W}_3, \hat{\tilde{W}}_1,\hat{\tilde{W}}_2)$ have the following form:
		\begin{equation}
			\begin{aligned}
				&\hat{W}_1 =
				\begin{bmatrix}
					\hat{W}_{1}^{(1)} & 0\\
					0 & \hat{W}^{(2)}_{1}
				\end{bmatrix},
				\hat{W}_2 =
				\begin{bmatrix}
					\hat{W}^{(1)}_2 & 0\\
					0 & \hat{W}^{(2)}_2
				\end{bmatrix},
				\hat{W}_3 =
				\begin{bmatrix}
					\hat{W}^{(1)} _3& 0\\
					0 & \hat{W}^{(2)}_3
				\end{bmatrix},\\
				&\hspace{2cm}\hat{\tilde{W}}_1 =
				\begin{bmatrix}
					\hat{\tilde{W}}^{(1)}_1& 0\\
					0 & \hat{\tilde{W}}^{(2)}_1
				\end{bmatrix},
				\hat{\tilde{W}}_2 =
				\begin{bmatrix}
					\hat{\tilde{W}}^{(1)}_2 & 0\\
					0 & \hat{\tilde{W}}^{(2)}_2
				\end{bmatrix}
			\end{aligned}
		\end{equation}
with respect to the decomposition $\mathcal{\tilde{K}}^* = \mathcal{\tilde{K}}_1 \oplus \mathcal{\tilde{K}}_2$. 

 Since $(\hat{W}_{1}^{(1)}, \hat{W}_{2}^{(1)}, \hat{W}_{3}^{(1)}, \hat{\tilde{W}}^{(1)}_1,\hat{\tilde{W}}^{(1)}_2$) is a $\Gamma_{E(3; 2; 1, 2)}$-unitary, it yields that $(\hat{W}_1, \hat{W}_2, \hat{W}_3, \hat{\tilde{W}}_1, \hat{\tilde{W}}_2)$ and \\ $(\hat{W}^{(1)}_2, \hat{W}^{(2)}_2, \hat{W}^{(3)}_2, \hat{\tilde{W}}^{(1)}_2, \hat{\tilde{W}}^{(2)}_2)$ have the same fundamental operators $B_1,2B_2,2\tilde{B}_1,\tilde{B}_2.$ From Theorem \ref{thm-15}, we also conclude that
		\begin{enumerate}
			\item $\mathcal{\tilde{K}}_2$ can be identified with $H^2 (\mathcal{D}_{\hat{W}_3})$;
			
			\item the operators $\hat{W}_1|_{\mathcal{\tilde{K}}_2}, \hat{W}_2|_{\mathcal{\tilde{K}}_2}, \hat{W}_3|_{\mathcal{\tilde{K}}_2}, \hat{\tilde{W}}_1|_{\mathcal{\tilde{K}}_2}, \hat{\tilde{W}}_2|_{\mathcal{\tilde{K}}_2}$ are unitarily equivalent to the multiplication operators $M_{B_1 + \tilde{B}^*_2\overline{z}}, M_{2B_2 + 2\tilde{B}^*_1\overline{z}}, M_{\overline{z}},M_{2\tilde{B}_1 + 2B^*_2\overline{z}}, M_{\tilde{B}_2 + B^*_1\overline{z}},$ respectively, on the vector-valued Hardy space $H^2 (\mathcal{D}_{\hat{W}_3})$, where $B_1,2B_2,2\tilde{B}_1,\tilde{B}_2$ are the fundamental operators of   $\hat{\textbf{W}} = (\hat{W}_1, \hat{W}_2, \hat{W}_3, \hat{\tilde{W}}_1,\hat{\tilde{W}}_2)$.
		\end{enumerate}
		Since $\hat{W}^*_3$ is a minimal isometric dilation of $S^*_3$, by Proposition \ref{vstar}, we conclude that $\mathcal{D}_{S_3}$ and $\mathcal{D}_{\hat{W}_3}$ have same dimension. This completes the proof.
	\end{proof}
	
	\section{Canonical Decomposition of $\Gamma_{E(3; 3; 1, 1, 1)}$-contraction and $\Gamma_{E(3; 2; 1, 2)}$-contraction}
	
We recall the definition of completely non-unitary contraction from \cite{Nagy}. A contraction $T$  on a Hilbert space $\mathcal H$  is said to be  completely non-unitary (c.n.u.) contractions if there exists no nontrivial reducing subspace $\mathcal L$  for $T$ such that $T |_{\mathcal L}$ is a unitary operator. This section presents the canonical decomposition of the $\Gamma_{E(3; 3; 1, 1, 1)}$-contraction and the $\Gamma_{E(3; 2; 1, 2)}$-contraction.  Any contraction $T$ on a Hilbert space $\mathcal{H}$ can be expressed as the orthogonal direct sum of a unitary and a completely non-unitary contraction. The details can be found in [Theorem 3.2, \cite{Nagy}]. We start with the following definition, which will be essential for the canonical decomposition of the $\Gamma_{E(3; 3; 1, 1, 1)}$-contraction and the $\Gamma_{E(3; 2; 1, 2)}$-contraction.	
	\begin{defn}
		\begin{enumerate}
			\item A $\Gamma_{E(3; 3; 1, 1, 1)}$-contraction $\textbf{T} = (T_1, \dots, T_7)$ is said to be completely non-unitary  {$\Gamma_{E(3; 3; 1, 1, 1)}$-contraction} if $T_7$ is a completely non-unitary contraction.
			
			\item A $\Gamma_{E(3; 2; 1, 2)}$-contraction $\textbf{S} = (S_1, S_2, S_3, \tilde{S}_1, \tilde{S}_2)$ is said to be completely non-unitary {$\Gamma_{E(3; 2; 1, 2)}$-contraction} if $S_3$ is a completely non-unitary contraction.		\end{enumerate}
	\end{defn}
We only state the following lemma from [Proposition $1.3.2$, \cite{Bhatia}].
\begin{lem}\label{bhatia}
Let $M=\begin{pmatrix}
A& X\\
X^* & B
\end{pmatrix}$ with $A\geq 0$ and $B\geq 0.$ Then $M\geq 0$ if and only if $X=A^{\frac{1}{2}}KB^{\frac{1}{2}}$ for some contraction $K.$
\end{lem}

We will now demonstrate that for each $\Gamma_{E(3; 3; 1, 1, 1)}$-contraction $\textbf{T} = (T_1, \dots, T_7),$ if we decompose $T_7$ into the same way as described before, then $T_i,1\leq i \leq 6,$ decomposes into the direct sum of operators on the same subspaces.

\begin{thm}[Canonical Decomposition for $\Gamma_{E(3; 3; 1, 1, 1)}$-Contraction]\label{thm-16}

		Let $\textbf{T} = (T_1, \dots, T_7)$ be a $\Gamma_{E(3; 3; 1, 1, 1)} $-contraction on a Hilbert space $\mathcal{H},$ and $\mathcal{H}_1$ be the maximal subspace of $\mathcal{H}$ that reduces $T_7$ and on which $T_7$ is unitary. Let $\mathcal{H}_2 = \mathcal{H} \ominus \mathcal{H}_1$. Then
\begin{enumerate}
\item $\mathcal{H}_1, \mathcal{H}_2$ are reducing subspaces of $T_1, \dots, T_6$;
\item $(T_1|_{\mathcal{H}_1}, \dots, T_6|_{\mathcal{H}_1}, T_7|_{\mathcal{H}_1}) $ is a $\Gamma_{E(3; 3; 1, 1, 1)} $-unitary and $(T_1|_{\mathcal{H}_2}, \dots, T_6|_{\mathcal{H}_2}, T_7|_{\mathcal{H}_2})$ is a completely non-unitary $\Gamma_{E(3; 3; 1, 1, 1)} $-contraction;
\item the subspaces $\mathcal{H}_1$ or $\mathcal{H}_2$ may be equal to $\{0\}$, the trivial subspace.

\end{enumerate}

\end{thm}	
	\begin{proof}
Let $\textbf{T} = (T_1, \dots, T_7)$ be a $\Gamma_{E(3; 3; 1, 1, 1)}$-contraction. It is evident that if $T_7$ is a completely non-unitary contraction, then $\mathcal H_1 = \{0\}$. If $T_7$ is unitary, then $\mathcal H = \mathcal H_1$, which implies that $\mathcal H_2 = \{0\}$. In either of these cases, the theorem follows easily. Let us consider the case where $T_7$ is neither a unitary operator nor a completely non-unitary contraction. Let
\begin{equation}\label{matrixr123}
\begin{aligned}
T_i &= \begin{pmatrix}
T_{11}^{(i)} & T_{12}^{(i)}\\
T_{21}^{(i)} & T_{22}^{(i)}
\end{pmatrix},1\leq i \leq 6~{\rm{and}},~~	
T_7 &= \begin{pmatrix}
T_{1} ^{(7)}& 0\\
0 & T_{2}^{(7)}
\end{pmatrix},
\end{aligned}
\end{equation}
with respect to the decomposition $\mathcal{H} = \mathcal{H}_1 \oplus \mathcal{H}_2$. Here, $T_{1} ^{(7)}$ is a unitary and $T_{2}^{(7)}$ is a completely non-unitary. Since  $T_{2}^{(7)}$ is completely non-unitary, it implies that if $x\in \mathcal H_2$ and 
\begin{equation}
			\begin{aligned}
				\Big|\Big|(T^{(7)}_2)^{*n}x\Big|\Big| &= ||x|| = \Big|\Big|(T^{(7)}_2)^{n}x\Big|\Big|,~n=1,2,\dots
			\end{aligned}
		\end{equation}
then it must be that $x = 0$. Because $T_iT_7=T_7T_i$ for $1\leq i \leq 6,$ we have 
		\begin{equation}\label{result23}
			\begin{aligned}
				&T^{(i)}_{11}T_{1}^{(7)} = T_{1}^{(7)}T^{(i)}_{11}, \hspace{0.5cm}
				T_{12}^{(i)}T_{2}^{(7)} = T_{1}^{(7)}T_{12}^{(i)},\\
				&T_{21}^{(i)}T_{1}^{(7)} = T_{2}^{(7)}T_{21}^{(i)}, \hspace{0.5cm}
				T_{22}^{(i)}T_{2}^{(7)} = T_{2}^{(7)}T_{22}^{(i)}.
			\end{aligned}
		\end{equation}	
As $\textbf{T} = (T_1, \dots, T_7)$ is a $\Gamma_{E(3; 3; 1, 1, 1)}$-contraction, it follows from Theorem \ref{fundam} that for $1\leq i \leq 6$, the following holds:	\begin{equation}
\begin{aligned}
\rho_{G_{E(2; 2; 1, 1)}}(z_1T_i, z_1T_{7-i}, z^2_1T_7) \geqslant 0 \,\, \text{and} \,\, \rho_{G_{E(2; 2; 1, 1)}}(z_2T_{7-i}, z_2T_i, z^2_2T_7) \geqslant 0 \,\, \text{for all} \,\, z_1,z_2 \in \mathbb{T}.
\end{aligned}
\end{equation}
Adding $\rho_{G_{E(2; 2; 1, 1)}}(z_1T_i, z_1T_{7-i}, z^2_1T_7)$ and   $\rho_{G_{E(2; 2; 1, 1)}}(z_2T_{7-i}, z_2T_i, z^2_2T_7)$ for $1\leq i \leq 6,$ we get
		\begin{equation}\label{eqa99}
			\begin{aligned}
				(I - T^*_7T_7) - \operatorname{Re} {z_1}(T_{7-i} - T^*_iT_7)- \operatorname{Re} z_2 (T_i - T^*_{7-i}T_7) &\geqslant 0.
			\end{aligned}
		\end{equation}
	Set $M_i:=(I - T^*_7T_7) - \operatorname{Re} {z_1}(T_{7-i} - T^*_iT_7)- \operatorname{Re} z_2 (T_i - T^*_{7-i}T_7)$ for $1\leq i\leq 6$.	
From \eqref{matrixr123} and \eqref{eqa99}, we have 
		\begin{equation}\label{matrixrr}
			\begin{aligned}
				0 &\leq M_i\\
				&=\begin{pmatrix}
					0 & 0\\
					0 & I - {T^{(7)}_{2}}^*T^{(7)}_{2}
				\end{pmatrix}
				-
				\operatorname{Re} z_1
				\begin{pmatrix}
					T_{11}^{(7-i)} - {T_{11}^{(i)}}^*T_{1} ^{(7)}& T_{12}^{(7-i)} - {T_{21}^{(i)}}^*T_{2} ^{(7)}\\
					T_{21}^{(7-i)} - {T_{12}^{(i)}}^*T_{1} ^{(7)}& T_{22}^{(7-i)} - {T_{22}^{(i)}}^*T_{2} ^{(7)}
				\end{pmatrix}\\
				&- \operatorname{Re} z_2
				\begin{pmatrix}
					T_{11}^{(i)} - {T_{11}^{(7-i)}}^*T_{1} ^{(7)}& T_{12}^{(i)} - {T_{21}^{(7-i)}}^*T_{2} ^{(7)}\\
					T_{21}^{(i)} - {T_{12}^{(7-i)}}^*T_{1} ^{(7)}& T_{22}^{(i)} - {T_{22}^{(7-i)}}^*T_{2} ^{(7)} 
					\end{pmatrix}.\\
			\end{aligned}
		\end{equation}
From \eqref{matrixrr}, the $(1,1)$ entry gives that 	
\begin{equation}\label{result}
\begin{aligned}
\operatorname{Re} z_1(T_{11}^{(7-i)} - {T_{11}^{(i)}}^*T_{1} ^{(7)})+\operatorname{Re} z_2(T_{11}^{(i)} - {T_{11}^{(7-i)}}^*T_{1} ^{(7)})\leq 0, 1\leq i \leq 6,
\end{aligned}
\end{equation} for all $z_1,z_2\in \mathbb T$.
Putting $z_1=\pm 1$ in \eqref{result}, we get 		
\begin{equation}\label{result1}
\begin{aligned}
\operatorname{Re} (T_{11}^{(7-i)} - {T_{11}^{(i)}}^*T_{1} ^{(7)})+\operatorname{Re} z_2(T_{11}^{(i)} - {T_{11}^{(7-i)}}^*T_{1} ^{(7)})\leq 0,1\leq i \leq 6,
\end{aligned}
\end{equation}
$${\rm{and }}$$		
\begin{equation}\label{result2}
\begin{aligned}
-\operatorname{Re} (T_{11}^{(7-i)} - {T_{11}^{(i)}}^*T_{1} ^{(7)})+\operatorname{Re} z_2(T_{11}^{(i)} - {T_{11}^{(7-i)}}^*T_{1} ^{(7)})\leq 0, 1\leq i \leq 6,
\end{aligned}
\end{equation}	for all $z_2\in \mathbb T.$	
From \eqref{result1} and \eqref{result2}, we have 		
\begin{equation}
\operatorname{Re} z_2(T_{11}^{(i)} - {T_{11}^{(7-i)}}^*T_{1} ^{(7)})\leq 0, 1\leq i \leq 6,
\end{equation}
for all $z_2\in \mathbb T,$ and hence we deduce that $T_{11}^{(i)} ={T_{11}^{(7-i)}}^*T_{1} ^{(7)}$ for $1\leq i \leq 6.$ Similarly we can also show that $T_{11}^{(7-i)} ={T_{11}^{(i)}}^*T_{1} ^{(7)}$ for $1\leq i \leq 6.$ As $T_{1} ^{(7)}$ is unitary and $\| T_{11}^{(i)}\|\leq 1$ for $1\leq i \leq 6,$ it follows from [Theorem $3.2$, \cite{apal2}] that $(T_{11}^{(1)}, \dots, T_{11}^{(6)}, T_{1} ^{(7)}) $ is a $\Gamma_{E(3; 3; 1, 1, 1)} $-unitary.
		
The matrix \( M_i \) is positive semi-definite, so it can be expressed as $M_i=\begin{pmatrix} A_i& X_i\\ X_i^* & B_i \end{pmatrix}$ for $1\leq i\leq 6$. Since $T_{11}^{(i)} ={T_{11}^{(7-i)}}^*T_{1} ^{(7)}$,  we deduce that $A_i=0$ for $1\leq i \leq 6$. It yields from Lemma \ref{bhatia} that $X_i=0.$
As $X_i=X_i^*=0,$ we derive  the following equation from \eqref{matrixrr}
\begin{equation}\label{result11}
\begin{aligned}
\operatorname{Re} z_1(T_{12}^{(7-i)} - {T_{21}^{(i)}}^*T_{2} ^{(7)})+\operatorname{Re} z_2(T_{12}^{(i)} - {T_{21}^{(7-i)}}^*T_{2} ^{(7)})=0
\end{aligned}
\end{equation}
$${\rm{and}}$$ 
\begin{equation}\label{result12}
\begin{aligned}
\operatorname{Re} z_1(T_{21}^{(7-i)} - {T_{12}^{(i)}}^*T_{1} ^{(7)})+\operatorname{Re} z_2(T_{21}^{(i)} - {T_{12}^{(7-i)}}^*T_{1} ^{(7)})=0
\end{aligned}
\end{equation}		
for all $z_1,z_2\in \mathbb T$ and for $1\leq i \leq 6$. By using similar argument as above, from \eqref{result11} and \eqref{result12}, we have for $1\leq i \leq 6,$
\begin{equation}\label{result22}
\begin{aligned}
T_{12}^{(7-i)} = {T_{21}^{(i)}}^*T_{2} ^{(7)}, T_{12}^{(i)} ={T_{21}^{(7-i)}}^*T_{2} ^{(7)}, T_{21}^{(7-i)} = {T_{12}^{(i)}}^*T_{1} ^{(7)} ~{\rm{and}}~T_{21}^{(i)} = {T_{12}^{(7-i)}}^*T_{1} ^{(7)}.
\end{aligned}
\end{equation}		
For $1\leq i \leq 6,$ it follows from \eqref{result23} and \eqref{result22} that 
\begin{equation}\label{t1}
\begin{aligned}
{T_{12}^{(7-i)}}^*(T_{1} ^{(7)})^2&=T_{21}^{(i)}T_{1} ^{(7)}\\&=T_{2} ^{(7)}T_{21}^{(i)}\\&=T_{2} ^{(7)} {T_{12}^{(7-i)}}^*T_{1} ^{(7)}
\end{aligned}
\end{equation}		
and hence 
\begin{equation}\label{t2}
{T_{12}^{(7-i)}}^*T_{1} ^{(7)}=T_{2} ^{(7)} {T_{12}^{(7-i)}}^*~{\rm{for}}~1\leq i \leq 6.
\end{equation}
By iterating the equations in \eqref{result23} and \eqref{t2}	 we deduce that, for any $n\geq 1,$
\begin{equation}\label{t22}
(T_{1} ^{(7)})^n{T_{12}^{(i)}}={T_{12}^{(i)}}(T_{2} ^{(7)} )^n ~{\rm{and}}~(T_{1} ^{(7)*})^n{T_{12}^{(7-i)}}={T_{12}^{(7-i)}}(T_{2} ^{(7)*} )^n~{\rm{for}}~1\leq i \leq 6.
\end{equation}
For $1\leq i \leq 6$, it follows from \eqref{t22} that  
\begin{equation}\label{t122}
\begin{aligned}
{T_{12}^{(i)}}(T_{2} ^{(7)} )^n(T_{2} ^{(7)*})^n&=(T_{1} ^{(7)})^n{T_{12}^{(i)}}(T_{2} ^{(7)*})^n\\&=(T_{1} ^{(7)})^n(T_{1} ^{(7)*})^n{T_{12}^{(i)}}\\&={T_{12}^{(i)}}.
\end{aligned}
\end{equation}		
Similarly, we can also show that 
\begin{equation}\label{1222}{T_{12}^{(i)}}(T_{2} ^{(7)*})^n(T_{2} ^{(7)} )^n={T_{12}^{(i)}} ~{\rm{for}}~ 1\leq i \leq 6.
\end{equation}
For  $x\in \mathcal H_1$ and  $n\geq 1,$ it follows from \eqref{t122} and \eqref{1222} that 	
\begin{equation}
\begin{aligned}
\Big|\Big|(T_{2} ^{(7)*} )^n{T_{12}^{(i)*}}x\Big|\Big|=\Big|\Big|{T_{12}^{(i)*}}x\Big|\Big|=\Big|\Big|(T_{2} ^{(7)} )^n{T_{12}^{(i)*}}x\Big|\Big|~{\rm{for}}~1\leq i \leq 6.
\end{aligned}
\end{equation}		
Since $T_{2}^{(7)}$ is completely non-unitary, we get ${T_{12}^{(i)*}}x=0,$ and so ${T_{12}^{(i)}}=0$ for $1\leq i \leq 6.$ Similarly, we can also prove that ${T_{21}^{(i)}}=0$ for $1\leq i \leq 6.$ Thus, $T_i$ has the following form \begin{equation}\label{matrixr}
\begin{aligned}
T_i &= \begin{pmatrix}
T_{11}^{(i)} & 0\\
0 & T_{22}^{(i)}
\end{pmatrix},1\leq i \leq 6,\end{aligned}
\end{equation}
with respect to the decomposition $\mathcal{H} = \mathcal{H}_1 \oplus \mathcal{H}_2$. Consequently, we conclude that $\mathcal{H}_1, \mathcal{H}_2$ are reducing subspaces of $T_1, \dots, T_6$. Moreover, $(T_{22}^{(1)}, T_{22}^{(2)},T_{22}^{(3)},T_{22}^{(4)},T_{22}^{(5)}, T_{22}^{(6)}, T_{2}^{(7)}) $, being a restriction of a  $\Gamma_{E(3; 3; 1, 1, 1)} $-contraction $\textbf{T} = (T_1, \dots, T_7)$ to the joint invariant subspace $\mathcal{H}_2,$  is a $\Gamma_{E(3; 3; 1, 1, 1)}$-contraction. As $T_{2}^{(7)}$ is a completely non-unitary contraction, we conclude that $(T_1|_{\mathcal{H}_2}, \dots, T_6|_{\mathcal{H}_2}, T_7|_{\mathcal{H}_2})$ is a completely non-unitary $\Gamma_{E(3; 3; 1, 1, 1)} $-contraction. We have already shown that $(T_1|_{\mathcal{H}_1}, \dots, T_6|_{\mathcal{H}_1}, T_7|_{\mathcal{H}_1})$ is a $\Gamma_{E(3; 3; 1, 1, 1)} $-unitary. This completes the proof.
	\end{proof}
We will now describe that for any $\Gamma_{E(3; 2; 1, 2)} $-contraction $\textbf{S} = (S_1, S_2, S_3, \tilde{S}_1, \tilde{S}_2)$, if we decompose $S_3$ as previously mentioned in this section, then $S_1, S_2, \tilde{S}_1$, and $\tilde{S}_2$ decompose into the direct sum of operators on the same subspaces.
	
	\begin{thm}[Canonical Decomposition of $\Gamma_{E(3; 2; 1, 2)}$-Contraction]\label{thm-17}
		Let $\textbf{S} = (S_1, S_2, S_3, \tilde{S}_1, \tilde{S}_2)$ be a $\Gamma_{E(3; 2; 1, 2)}$-contraction on a Hilbert space $\mathcal{H}$, and $\mathcal{H}_1$ be the maximal subspace of $\mathcal{H}$ that reduces $S_3$ and on which $S_3$ is unitary. Let $\mathcal{H}_2 = \mathcal{H} \ominus \mathcal{H}_1$. Then
		\begin{enumerate}
			\item $\mathcal{H}_1, \mathcal{H}_2$ are reducing subspaces of $S_1, S_2, \tilde{S}_1, \tilde{S}_2$;
			
			\item $(S_1|_{\mathcal{H}_1}, S_2|_{\mathcal{H}_1}, S_3|_{\mathcal{H}_1}, \tilde{S}_1|_{\mathcal{H}_1}, \tilde{S}_2|_{\mathcal{H}_1})$ is $\Gamma_{E(3; 2; 1, 2)}$-unitary and $(S_1|_{\mathcal{H}_2}, S_2|_{\mathcal{H}_2}, S_3|_{\mathcal{H}_2}, \tilde{S}_1|_{\mathcal{H}_2}, \tilde{S}_2|_{\mathcal{H}_2})$ is  a completely non-unitary $\Gamma_{E(3; 2; 1, 2)}$-contraction;
			
			\item the subspaces $\mathcal{H}_1$ or $\mathcal{H}_2$ may be equal to $\{0\}$, the trivial subspace.
		\end{enumerate}
	\end{thm}
	
	\begin{proof}
Let $\textbf{S} = (S_1, S_2, S_3, \tilde{S}_1, \tilde{S}_2)$ be a $\Gamma_{E(3; 2; 1, 2)}$-contraction. It is evident that if $S_3$ is a completely non-unitary contraction, then $\mathcal H_1 = \{0\}$. If $S_3$ is unitary, then $\mathcal H = \mathcal H_1$, which implies that $\mathcal H_2 = \{0\}$. In either of these cases, the theorem follows easily. Let us consider the case where $S_3$ is neither a unitary operator nor a completely non-unitary contraction. 	Let \begin{equation}\label{s112}
\begin{aligned}
S_i &= \begin{pmatrix}
					S_{11} ^{(i)}& S_{12}^{(i)}\\
					S_{21}^{(i)} & S_{22}^{(i)}
				\end{pmatrix},\;\;\; S_3&=\begin{pmatrix}
					S_{1}^{(3)} & 0\\
					0 & S_{2}^{(3)}
				\end{pmatrix}
				~\text{and}~
				\tilde{S}_j = \begin{pmatrix}
					\tilde{S}_{11}^{(j)}& \tilde{S}_{12}^{(j)}\\
					\tilde{S}_{21}^{(j)} & \tilde{S}_{22}^{(j)}
				\end{pmatrix}, \, 1\leq i, j \leq 2,
			\end{aligned}
		\end{equation}
with respect to the decomposition $\mathcal{H} = \mathcal{H}_1 \oplus \mathcal{H}_2$, such that $S_{1} ^{(3)}$ is unitary and $S_{2}^{(3)}$ is completely non-unitary. Since $S_iS_3=S_3S_i$ for $1\leq i \leq 2$ and $ \tilde{S}_jS_3=S_3\tilde{S}_j$ for $1\leq j \leq 2,$ we have the following:
		\begin{equation}
			\begin{aligned}
				S_{11}^{(i)}S_{1} ^{(3)}= S_{1}^{(3)}S_{11}^{(i)}, 
				S_{12}^{(i)}S_{2}^{(3)} = S_{1}^{(3)}S_{12}^{(i)},
				S_{21}^{(i)}S_{1}^{(3)} = S_{2}^{(3)}S_{21}^{(i)},
				S_{22}^{(i)}S_{2}^{(3)} = S_{2}^{(3)}S_{22}^{(3)}~{\rm{for}}~1\leq i \leq 2,
			\end{aligned}
		\end{equation}
		$${\rm{and}}$$
		\begin{equation}
			\begin{aligned}
				\tilde{S}_{11}^{(j)}S_{1} ^{(3)}= S_{1} ^{(3)}\tilde{S}_{11}^{(j)}, 
				\tilde{S}_{12}^{(j)}S_{2}^{(3)}= S_{1} ^{(3)}\tilde{S}_{12}^{(j)},
				\tilde{S}_{(21}^{(j)}S_{1} ^{(3)} = S_{2}^{(3)}\tilde{S}_{21}^{(j)}, 
				\tilde{S}_{22}^{(j)}S_{2}^{(3)}= S_{2}^{(3)}\tilde{S}_{22}^{(j)}~{\rm{for}}~1\leq j \leq 2.
			\end{aligned}
		\end{equation}
Since $\textbf{S} = (S_1, S_2, S_3, \tilde{S}_1, \tilde{S}_2)$ is a $\Gamma_{E(3; 2; 1, 2)}$-contraction,  it yields from Theorem \ref{s1s2} that 		\begin{equation}
				\begin{aligned}
					&\rho_{G_{E(2; 2; 1, 1)}}(z_1S_1, z_1\tilde{S}_2, z_1^2S_3) \geqslant 0 ~\text{and}~
					\rho_{G_{E(2; 2; 1, 1)}}(z_2\tilde{S}_2, z_2S_1, z_2^2S_3) \geqslant 0,~{\rm{for~all}}~z_1,z_2 \in \mathbb T
				\end{aligned} \end{equation}
				$${\rm{and}}$$
				\begin{equation}
				\begin{aligned}
					&\rho_{G_{E(2; 2; 1, 1)}}\left(\omega_1\frac{S_2}{2}, \omega_1\frac{\tilde{S}_1}{2}, \omega_1^2S_3\right) \geqslant 0 ~\text{and}~
					\rho_{G_{E(2; 2; 1, 1)}}\left(\omega_2\frac{\tilde{S}_1}{2}, \omega_2\frac{S_2}{2}, \omega_2^2S_3\right) \geqslant 0, ~{\rm{for~all}}~\omega_1,\omega_2 \in \mathbb T
				\end{aligned}
			\end{equation}
			and the spectral radius of $\tilde{S}_z$  and $\hat{S}_z$ are not bigger than $2$.
By using the similar argument as in Theorem \ref{thm-16}, we demonstrate that  
\begin{equation}\label{s11s12}
S_{11}^{(1)}=\tilde{S}_{11}^{(2)*}S_1^{(3)}, \|S_{11}^{(1)}\|\leq 1,S_{11}^{(2)}=\tilde{S}_{11}^{(1)*}S_1^{(3)}, \|S_{11}^{(2)}\|\leq 2,S^{(i)}_{12}=0=S^{(i)}_{21}~{\rm{and}}~\tilde{S}_{12}^{(j)}=0=\tilde{S}_{21}^{(j)}~{\rm{for}}~1\leq i,j\leq 2.
\end{equation}
It is clear from \eqref{s112} and \eqref{s11s12} that $\mathcal{H}_1, \mathcal{H}_2$ are reducing subspace of $S_1, S_2, \tilde{S}_1, \tilde{S}_2$. As $S_{1} ^{(3)}$ is unitary, it yields from [Theorem $3.7$,\cite{apal2}] and \eqref{s11s12} that $$(S_{11}^{(1)}, S_{11}^{(2)}, S_{1}^{(3)}, \tilde{S}_{11}^{(1)}, \tilde{S}_{11}^{(2)}) = (S_1|_{\mathcal{H}_1}, S_2|_{\mathcal{H}_1}, S_3|_{\mathcal{H}_1}, \tilde{S}_1|_{\mathcal{H}_1}, \tilde{S}_2|_{\mathcal{H}_1})$$ is a  $\Gamma_{E(3; 2; 1, 2)}$-unitary.  Furthermore, $(S_{22}^{(1)}, S_{22}^{(2)}, S_{2}^{(3)}, \tilde{S}_{22}^{(1)}, \tilde{S}_{22}^{(2)})$, being the restriction of a $\Gamma_{E(3; 2; 1, 2)}$-contraction $\textbf{S} = (S_1, S_2, S_3, \tilde{S}_1, \tilde{S}_2)$ to the joint invariant subspace $\mathcal{H}_2,$  is a $\Gamma_{E(3; 2; 1, 2)}$-contraction.
Since $S_{2}^{(3)}$ is completely non-unitary, we conclude that  $(S_1|_{\mathcal{H}_2}, S_2|_{\mathcal{H}_2}, S_3|_{\mathcal{H}_2}, \tilde{S}_1|_{\mathcal{H}_2}, \tilde{S}_2|_{\mathcal{H}_2})$ is  a completely non-unitary $\Gamma_{E(3; 2; 1, 2)}$-contraction.
This completes the proof.
	\end{proof}
	
	\begin{rem}
It is important to note that the Wold decompositions for $\Gamma_{E(3; 3; 1, 1, 1)}$-isometries and $\Gamma_{E(3; 2; 1, 2)}$-isometries are specific examples of the canonical decomposition of $\Gamma_{E(3; 3; 1, 1, 1)}$-contractions and $\Gamma_{E(3; 2; 1, 2)}$-contractions, respectively.	\end{rem}
	
\textsl{Acknowledgements:}
The second-named author is supported by the research project of SERB with ANRF File Number: CRG/2022/003058, by the Science and Engineering Research Board (SERB), Department of Science and Technology (DST), Government of India. 

	%
\end{document}